\documentclass[final,onefignum,onetabnum]{siamart190516}

\usepackage{lipsum}  % default
\usepackage{amsfonts}  % default
\usepackage{graphicx}  % default
\usepackage{epstopdf}  % default
\usepackage{algorithmic}  % default

\ifpdf  % default
  \DeclareGraphicsExtensions{.eps,.pdf,.png,.jpg}  % default
\else  % default
  \DeclareGraphicsExtensions{.eps}  % default
\fi  % default

\usepackage{mathtools} % for coloneqq etc. without changing \mathcal
\usepackage{amsmath}
\usepackage{amssymb}
\usepackage{enumerate}
\usepackage{mathptmx} % use Times fonts if available on your TeX system
\DeclareMathAlphabet{\mathcal}{OMS}{cmsy}{m}{n}
\usepackage{newtxtext,newtxmath}
\usepackage{latexsym}
\usepackage{autobreak}

\usepackage{multirow}
\usepackage[tight-spacing=true]{siunitx}

\newcommand{\paren}[1]{\left(#1\right)}  % ()
\newcommand{\brc}[1]{\left\{#1\right\}}  % {}
\newcommand{\sbra}[1]{\left\lbrack#1\right\rbrack}  % []
\newcommand{\abs}[1]{\left|#1\right|}  % ||
\newcommand{\norm}[1]{\left\lVert #1\right\rVert}  % || ||
\newcommand{\ceil}[1]{\left\lceil#1\right\rceil}  % ceil 
\newcommand{\plchold}{\:\cdot\:} % placeholder
\newcommand{\RSQO}{\mbox{RSQO}}
\newcommand{\mani}[1][M]{\mathcal{#1}}  % manifold
\newcommand{\tanspc}[1][\point]{T_{#1}}  % tangent space
\newcommand{\cotanspc}[1][\point]{T^{\ast}_{#1}}  % tangent space
\newcommand{\metr}[3][\point]{\left\langle #2, #3 \right\rangle_{#1}}  % metric
\newcommand{\relmiddle}[1]{\mathrel{}\middle#1\mathrel{}}  % relmiddle
\newcommand{\eucli}[1][]{\mathbb{R}^{#1}}  % Euclidean space
\newcommand{\coordmetr}[1][\point]{\widehat{G}_{\widehat{#1}}}  % metric under coordinates
\newcommand{\partdiff}[2]{\frac{\partial #1}{\partial #2}} % partial differential

\definecolor{green}{rgb}{ .0, .8, .6}

\usepackage{ifthen}
\newcommand{\COMM}[2]{
    {\ifthenelse{\equal{#1}{MO}}{\color{red}}{
            \ifthenelse{\equal{#1}{AT}}{\color{blue}}{
                \ifthenelse{\equal{#1}{TO}}{\color{green}}
                }
            }
        [#2]
    }
}

\newsiamremark{remark}{Remark}  % default
\newsiamremark{inprfof}{Hypothesis}  % default

\newcommand{\cond}{P} 
\newcommand{\fact}{F}
\newcommand{\casecom}{Case} 

\usepackage{breqn}
\usepackage{cite}

\headers{Riemannian SQO method}{M. Obara, T. Okuno, and A. Takeda}

\title{Sequential Quadratic Optimization for Nonlinear Optimization Problems on Riemannian Manifolds\thanks{submitted to the editors on September 29, 2020 and revised on June 15, 2021. This is the extended version of a paper submitted to a journal.
\funding{This work was supported by the Japan Society for the Promotion of Science KAKENHI under 17H01699, 19H04069, and 20K19748.}}}

\author{
    Mitsuaki Obara\thanks{Graduate School of Information Science and Technology, The University of Tokyo, Tokyo, Japan (\email{mitsuaki\_obara@mist.i.u-tokyo.ac.jp, takeda@mist.i.u-tokyo.ac.jp})}
    \and
    Takayuki Okuno\thanks{Center for Advanced Intelligence Project, RIKEN, Tokyo, Japan (\email{takayuki.okuno.ks@riken.jp, akiko.takeda@riken.jp})}
    \and
    Akiko Takeda\footnotemark[2] \footnotemark[3]
}

\usepackage{amsopn}

\ifpdf
\hypersetup{
  pdftitle={Riemannian SQO for Riemannian Nonlinear Optimization Problems},
  pdfauthor={M. Obara, T. Okuno, and A. Takeda}
}
\fi

\begin{document}

\maketitle

\begin{abstract}
We consider optimization problems on Riemannian manifolds with equality and inequality constraints, which we call Riemannian nonlinear optimization (RNLO) problems. Although they have numerous applications, the existing studies on them are limited especially in terms of algorithms. In this paper, we propose Riemannian sequential quadratic optimization (RSQO) that uses a line-search technique with an $\ell_{1}$ penalty function as an extension of the standard SQO algorithm for constrained nonlinear optimization problems in Euclidean spaces to Riemannian manifolds. We prove its global convergence to a Karush-Kuhn-Tucker point of the RNLO problem by means of parallel transport and the exponential mapping. Furthermore, we establish its local quadratic convergence by analyzing the relationship between sequences generated by RSQO and the Riemannian Newton method. Ours is the first algorithm that has both global and local convergence properties for constrained nonlinear optimization on Riemannian manifolds. Empirical results show that RSQO finds solutions more stably and with higher accuracy compared with the existing Riemannian penalty and augmented Lagrangian methods.
\end{abstract}



\section{Introduction}
In this paper, we consider the following problem:
\begin{align}
    \begin{split}
    \label{RNLO}
    \underset{x\in\mani}{\text{minimize}} \quad & f\paren{x} \\
    \text{subject to} \quad  & g_{i}\paren{x} \leq 0, \text{ for all } i \in \mathcal{I} \coloneqq \brc{1,\ldots,m}, \\ 
    & h_{j}\paren{x} = 0, \text{ for all } j \in \mathcal{E} \coloneqq \brc{1,\ldots,n},
    \end{split}
\end{align}
where $\mani$ is a $d$-dimensional Riemannian manifold and $f, \brc{g_{i}}_{i\in\mathcal{I}}$, and $\brc{h_{j}}_{j\in\mathcal{E}}$ are continuously differentiable
functions from $\mani$ to $\eucli$. Moreover, $\mani$ is assumed to be connected and complete. Throughout this paper, we call problem \eqref{RNLO} the Riemannian nonlinear optimization problem and abbreviate it as the RNLO problem. This problem is a natural extension of the standard constrained nonlinear optimization problem in a Euclidean space to a Riemannian manifold. Indeed, if $\mani=\eucli[d]$, \eqref{RNLO} reduces to the standard problem on $\eucli[d]$.

By virtue of its versatility, many applications of RNLO \eqref{RNLO} arise naturally in various fields such as machine learning and control theory. For instance, nonnegative low-rank matrix completion~\cite{SongNg20NonnegRankMatApprox} can be formulated as an optimization problem on a fixed-rank manifold with nonnegative inequality constraints. $k$-means clustering~\cite{Carsonetal17ManiOptforkmeans} can be represented as a problem on the Stiefel manifold with equality and inequality constraints. Robotic posture computations~\cite{Brossetteetal18SQPforRobotics} and nonnegative principal component analysis~\cite{ZassShashua06NNPCA} are also representative examples.

Optimization on Riemannian manifolds, called Riemannian optimization, has seen extensive development in the last few decades for unconstrained cases, namely, RNLO~\eqref{RNLO} with $\mathcal{I}=\emptyset$ and $\mathcal{E} = \emptyset$. Absil et al.~\cite{Absiletal08OptBook} laid out theories for algorithms such as the geometric Newton method and Riemannian trust-region method. On the basis of their work, unconstrained Riemannian optimization algorithms and their applications have advanced in various ways; see \cite{Huangetal18RBFGSforNonconv,Boumaletal18RiemGDandTRwithGlobRate,Bortolotietal18dampedNewton,SatoAihara19CholQRRetrGeneStiefelMani,ZhuSato20RiemCGwithInvRetr,Bonnabel13RiemSGD, Zhangetal16RiemSVRG,Zhangetal18RSPIDER}, for example. We also refer the reader to the latest book~\cite{Boumal20IntroRiemOptBook} by Boumal for an introduction to unconstrained Riemannian optimization and a comprehensive survey article~\cite{Huetal19ManSurvey} by Hu et al. for recent developments on Riemannian optimization.
Connections between unconstrained Riemannian optimization and nonlinear optimization methods in Euclidean spaces have been also investigated. For example, Absil et al.~\cite{Absiletal09ConnectionRiemNewtonFPSQP} showed that the Riemannian Newton method and feasibility perturbed SQO (FP-SQO)~\cite{WrightTenny04FeasibleTrustRegionSQP} produce the same iterates when we consider an equality-constrained optimization problem in a Euclidean space and the constraints define an embedded submanifold of the Euclidean space. 
For the same problem, Bai and Mei~\cite{BaiMei19AnalSQPthruLensofRiemOpt} analyzed a certain SQO in a Euclidean space called a first-order SQO therein by using Riemannian techniques and derived global and local convergence rates.

In contrast, studies on constrained nonlinear optimization problems on Riemannian manifolds are still very scarce. Yang et al.~\cite{Yangetal14OptCond} provided the Karush-Kuhn-Tucker (KKT) conditions and second-order necessary and sufficient conditions for RNLO~\eqref{RNLO}. Bergmann and Herzog~\cite{BergmannHerzog19IntKKT} extended constraint qualifications from Euclidean spaces to smooth manifolds. Liu and Boumal~\cite{LiuBoumal19Simple} developed an augmented Lagrangian method and an exact penalty method combined with smoothing techniques. They also showed the global convergence properties of the algorithms. As far as we know, they were the first to present algorithms for constrained Riemannian optimization problems of the \eqref{RNLO} form. Moreover, sequential quadratic optimization (SQO) or sequential quadratic programming (SQP) algorithms, which are our interest here, have been extended from Euclidean spaces to Riemannian manifolds in several ways, as explained below.

The SQO algorithm is one of the most effective algorithms for constrained nonlinear optimization in a Euclidean space. The strength of SQO is that it has both global and fast local convergence guarantees under certain assumptions~\cite{NocedalWright06NumOpt}. We refer readers to a survey article~\cite{BoggsTolle96SQPsurvey} by Boggs and Tolle for details on SQO in a Euclidean space. Below, we shall review the existing work on SQO on Riemannian manifolds. Schiela and Ortiz~\cite{SchielaOrtiz20SQPforECP} proposed an SQO method for problems on smooth Hilbert manifolds with only equality-constrained cases, i.e., RNLO \eqref{RNLO} with $\mathcal{I}=\emptyset$, and studied its local convergence property. Their algorithmic policy is to perform two steps, called normal and tangential steps, to improve feasibility and optimality. Brossette et al.~\cite{Brossetteetal18SQPforRobotics} proposed an SQO algorithm for problems having only inequality constraints, i.e., RNLO \eqref{RNLO} with $\mathcal{E}=\emptyset$. However, they did not theoretically examine the convergence of SQO and instead focused on its application to a problem in robotics.

It is worthwhile to note that as yet there is no SQO algorithm on Riemannian manifolds which is ensured to have global convergence to a point satisfying the KKT conditions, to the best of our knowledge. Moreover, there is no SQO algorithm for RNLO \eqref{RNLO} with $\mathcal{E}\neq \emptyset$ and $\mathcal{I}\neq \emptyset$. One may think that such an RNLO can be handled by the existing SQO methods mentioned above because inequality constraints $g_i(x)\le 0\ (i\in \mathcal{I})$ can be transformed into the equality constraints $g_i(x)+u_i^2=0\ (i\in \mathcal{I})$ by means of squared slack variables $u_i \in \mathbb{R}\ (i\in \mathcal{I})$, and moreover, equality constraints $h_j(x)=0\ (j\in \mathcal{E})$ can be expressed as the two inequalities $h_j(x)\le 0,\ h_j(x)\ge 0\ (j\in \mathcal{E})$. However, these manipulations may impair the solution of the problem. For example, the use of squared slack variables may increase the number of KKT points that do not satisfy the KKT conditions of the original problem. Moreover, by splitting equality constraints into two inequalities, the linear independence constraint qualification necessarily fails at any feasible point. From the above standpoint, it would be advantageous to have an algorithm that can directly solve RNLO \eqref{RNLO} with both inequality and equality constraints.

\subsection{Our contribution}
In this paper, we propose an SQO algorithm for RNLO \eqref{RNLO}. We will often call this algorithm a Riemannian SQO algorithm, or RSQO algorithm for short, while we call an SQO in a Euclidean space Euclidean SQO. Given an iterate, the proposed RSQO algorithm finds a search direction by solving a quadratic subproblem that is organized on a tangent space of the manifold $\mani$. Unlike Euclidean SQO, we make use of retraction, which is a concept specific to manifolds for determining the next iterate in $\mani$. Next, along the curve defined by the retraction, we further utilize the $\ell_1$ penalty function, which is presented in \cite{LiuBoumal19Simple} for RNLO~\eqref{RNLO}, as a merit function so as to compute an appropriate step length in accordance with a backtracking line search.
Note that, at every iteration, RSQO does not necessarily satisfy the whole constraints, $g_{i}\paren{x} \leq 0 \paren{i\in\mathcal{I}}$ and $h_{j}\paren{x} = 0 \paren{j \in \mathcal{E}}$, while satisfying $x\in\mani$ certainly.
Particularly when $\mani$ is the whole Euclidean space, the proposed RSQO reduces to the Han's Euclidean SQO~\cite{Han77GlobConvSQPwithExactPenalty} using the backtracking line search.

We will prove global convergence to a point satisfying the KKT conditions of RNLO~\eqref{RNLO}. We will also prove local quadratic convergence under certain assumptions 
by considering the relationship between sequences produced by the RSQO algorithm and the Riemannian Newton method~\cite[Chapter 6]{Absiletal08OptBook}.

Now, our contributions are summarized as follows:
\begin{enumerate}
    \item Our RSQO algorithm adequately handles both inequality and equality constraints. Moreover, it is the first one that ensures both global and local convergence for constrained nonlinear optimization on Riemannian manifolds. The previous Riemannian algorithms have no convergence guarantee or either global or local convergence properties, not both.
    \item We conduct numerical experiments clarifying that RSQO is very promising; it solved the problems in our experiments more stably and with 
    higher accuracy in comparison with the existing Riemannian penalty and augmented Lagrangian methods.
\end{enumerate}

\subsection{Organization of the paper}
The rest of this paper is organized as follows. In Section~\ref{preliminaries}, we review fundamental concepts from Riemannian geometry and Riemannian optimization. In Section~\ref{method}, we describe RSQO and analyze its global and local convergence properties.
In Section~\ref{experiment}, we provide numerical results on nonnegative low-rank matrix completion problems. 
We also compare our algorithm with the existing methods. 
In Section~\ref{conclusion}, we summarize our research and state future work.

\section{Preliminaries}\label{preliminaries}
\subsection{Notation and terminology from Riemannian geometry}
Let us briefly review some concepts from Riemannian geometry, following the notation of \cite{Absiletal08OptBook}. Let $x\in\mani$ and $\tanspc\mani$ be the tangent space to $\mani$ at $x$. A Riemannian manifold is a smooth manifold endowed with a smooth mapping $\metr[]{\cdot}{\cdot}:x\mapsto\metr{\cdot}{\cdot}$ such that $\metr{\cdot}{\cdot}:\tanspc\mani\times\tanspc\mani\rightarrow\eucli$ is an inner product called a Riemannian metric at $x$. The Riemannian metric induces the norm $\norm{\xi}_{x} \coloneqq \sqrt{\metr[x]{\xi}{\xi}}$ for $\xi\in \tanspc\mani$ and $\text{dist}\paren{\cdot,\cdot}:\mani\times\mani\rightarrow\eucli$, the Riemannian distance between two points. Let $\paren{\mathcal{U},\varphi}$ be a chart of $\mani$. Here, $\mathcal{U}\subseteq\mani$ is an open set and $\varphi: \mathcal{U} \rightarrow \varphi\paren{\mathcal{U}}\subseteq \eucli[d]$ is a homeomorphism. When $\mani = \eucli[d]$, $\mathcal{U}$ is any open ball in the usual sense and $\varphi$ equals the identity map. We will often omit the subscript $x$ when it is clear from the context. From \cite[Theorem 13.29]{Lee12IntrotoSmoManibook2ndedn}, $\mani$ is a metric space under the Riemannian distance. According to the Hopf-Rinow theorem (see e.g. O'Neil~\cite{OwithPrimeSymbolNeil83Semi}), every closed bounded subset of $\mani$ is compact for a finite-dimensional connected complete manifold by regarding $\mani$ as a metric space.

Let $V$ be a finite-dimensional vector space and $l:V\rightarrow\eucli[]$ be a continuous function. Here, we define the one-sided directional derivative at $p\in V$ along $v \in V$, denoted by $l^{\prime}\paren{p;v}$, as
\begin{align*}
	l^{\prime}\paren{p;v} \coloneqq \lim_{t\downarrow 0}\frac{l\paren{p+tv}-l\paren{p}}{t}
\end{align*}
if the limit exists. Given a sufficiently smooth function $\theta:\mani\rightarrow\eucli$, we denote by $\text{D}\theta\paren{x}\sbra{\xi}\in\eucli$ the differential of $\theta$ at $x\in \mani$ along $\xi\in \tanspc\mani$. Particularly when $\mani=V$, we have $\text{D}\theta\paren{x}\sbra{\xi} = \theta^{\prime}\paren{x;\xi}$ under $\tanspc[x]V\simeq V$, where $\tanspc[x]V\simeq V$ is the canonical identification. Throughout this paper, for a given vector space $V$ and $p\in V$, we write $\tanspc[p]V\simeq V$ when $\tanspc[p]V$ is canonically identified with $V$. For a precise definition of the differential on manifolds, see, e.g., Absil et al.~\cite{Absiletal08OptBook}. 

The gradient of $\theta$ at $x$, denoted by $\text{grad}\,\theta\paren{x}$, is defined as a unique element of $\tanspc\mani$ that satisfies
\begin{align}\label{def:grad}
    \metr{\text{grad}\,\theta\paren{x}}{\xi} = \text{D}\theta\paren{x}\sbra{\xi}, \quad \forall \xi \in \tanspc\mani.
\end{align}
Note that, for any $x\in\mani$, the above operator grad is $\eucli$-linear: for any continuously differentiable functions $\theta_{1},\theta_{2}:\mani\rightarrow\eucli$ and all $a,b\in\eucli$,  $\text{grad}\paren{a\theta_{1}+b\theta_{2}}\paren{x} = a \, \text{grad}\,\theta_{1}\paren{x} + b \, \text{grad}\,\theta_{2}\paren{x}$ holds. We will use the hat symbol to represent the corresponding counterparts in $\eucli[d]$, called coordinate expressions, of the objects related to $\mani$ or $\tanspc[x]\mani$: for any chart $\paren{\mathcal{U},\varphi}$ containing $x$, we write 
\begin{align*}
    \widehat{x} \coloneqq \varphi\paren{x}, \  \widehat{\xi} \coloneqq \text{D}\varphi\paren{x}\sbra{\xi}, \text{ and } \widehat{\theta}\coloneqq \theta\circ \varphi^{-1}
\end{align*}
for any $\xi\in\tanspc\mani$ and $\theta:\mani\rightarrow\eucli$. Note that $\text{D}\theta\paren{x}\sbra{\xi} = {\text{D}\widehat{\theta}\paren{\widehat{x}}}^{\top}\widehat{\xi}$, where $\text{D}\widehat{\theta}\paren{\widehat{x}}$ is the standard gradient of $\widehat{\theta}$ in $\eucli[d]$; i.e., $\text{D}\widehat{\theta}\paren{\widehat{x}}$ is a $d$-dimensional vector whose $i$-th element is $\frac{\partial \hat{\theta}(\hat{x})}{\partial e_i}\in\eucli$. We denote by $\coordmetr$ the coordinate expression of the Riemannian metric at $\widehat{x}$ under the chart. Here, $\coordmetr$ is a positive-definite matrix of size $d$ whose $\paren{i,j}$-th element is $\metr[x]{\partdiff{}{e_{i}}}{\partdiff{}{e_{j}}}$, where $\partdiff{}{e_{i}}$ and $\partdiff{}{e_{j}}$ denote the $i$-th and $j$-th bases of $\tanspc\mani$. When $\mani = \eucli[d]$, we can choose the canonical scalar product as a Riemannian metric. Let $\tanspc[]\mani \coloneqq \cup_{x\in\mani}\tanspc\mani$ be the tangent bundle. A retraction is a smooth mapping $R:\tanspc[]\mani\rightarrow\mani$ with the following properties: let $R_{x}$ denote the restriction of $R$ to $x$. Then, it holds that
\begin{align}
    \begin{split}\label{def:retr}
        &R_{x}\paren{0_{x}} = x, \\
        &\text{D}R_{x}\paren{0_{x}} = \text{id}_{\tanspc\mani},\text{ under }\tanspc[0_{x}]\paren{\tanspc\mani} \simeq \tanspc\mani, 
    \end{split}
\end{align}
where $0_{x}$ is the zero element of $\tanspc\mani$ and $\text{id}_{\tanspc\mani}$ denotes the identity mapping on $\tanspc\mani$. Note that, when $\mani=\eucli[d]$, $R_{x}\paren{\xi} = x + \xi$ is one of the retractions under $\tanspc[x]\eucli[d]\simeq\eucli[d]$, for example. The Riemannian Hessian operator at $x$ of $\theta:\mani\rightarrow\eucli$ is the linear mapping $\text{Hess}\,\theta\paren{x}$ from $\tanspc\mani$ to itself, defined by $\text{Hess}\,\theta\paren{x}\sbra{\xi_{x}} \coloneqq \nabla_{\xi_{x}}\text{grad}\,\theta$ for all $\xi_{x} \in \tanspc\mani$, where $\nabla$ is the Levi-Civita connection on $\mani$, the unique symmetric connection compatible with the Riemannian metric. Note that, for any $x\in\mani$, the operator Hess is $\eucli$-linear: for any twice continuously differentiable functions $\theta_{1},\theta_{2}:\mani\rightarrow\eucli$ and all $a,b\in\eucli$, $\text{Hess}\paren{a\theta_{1}+b\theta_{2}}\paren{x} = a \, \text{Hess}\,\theta_{1}\paren{x} + b \, \text{Hess}\,\theta_{2}\paren{x}$ holds. For $i,j,\ell = 1,\ldots,d$, let the real-valued function $\Gamma_{ij}^{\ell}:\mathcal{U}\rightarrow\eucli[]$ be the Christoffel symbol associated with the Levi-Civita connection (or thus the Riemannian metric) and the chart. Note that $\Gamma^{\ell}_{ij}$ is symmetric with respect to $i$ and $j$, i.e., $\Gamma^{\ell}_{ij} = \Gamma^{\ell}_{ji}$ for each $i,j,\ell=1,\ldots,d$. Particularly when $\mani=\eucli[d]$, $\Gamma^{\ell}_{ij}=0$ for all $i,j,\ell=1,\ldots,d$. Using the Christoffel symbols, we obtain the coordinate expression of the inner product involving the Riemannian Hessian operator as follows: for all $\xi, \eta \in \tanspc[x]\mani$, 
\begin{align}\label{Hesscoordmetr}
    &\metr[]{\text{Hess}\,\theta\paren{x}\sbra{\xi}}{\eta} = 
    {\widehat{\xi}}^{\top}\paren{\text{D}^{2}\widehat{\theta}\paren{\widehat{x}} - \widehat{\Gamma}_{\widehat{x}}\sbra{\text{D}\widehat{\theta}\paren{\widehat{x}}}}
    \widehat{\eta},
\end{align}
where 
\begin{align*}
    \widehat{\Gamma}_{\widehat{x}}\sbra{\text{D}\widehat{\theta}\paren{\widehat{x}}} \coloneqq
    \begin{pmatrix}
    \sum_{\ell=1}^{d}\Gamma_{11}^{\ell}\paren{\widehat{x}} \frac{\partial \widehat{\theta}\paren{\widehat{x}}}{\partial e_{\ell}}  & \cdots & \sum_{\ell=1}^{d}\Gamma_{1d}^{\ell}\paren{\widehat{x}} \frac{\partial \widehat{\theta}\paren{\widehat{x}}}{\partial e_{\ell}}  \\
    \vdots & \ddots & \vdots \\
    \sum_{\ell=1}^{d}\Gamma_{d1}^{\ell}\paren{\widehat{x}} \frac{\partial \widehat{\theta}\paren{\widehat{x}}}{\partial e_{\ell}} & \cdots & \sum_{\ell=1}^{d}\Gamma_{dd}^{\ell}\paren{\widehat{x}} \frac{\partial \widehat{\theta}\paren{\widehat{x}}}{\partial e_{\ell}} \\
    \end{pmatrix}
\end{align*}
and $\text{D}^{2}\widehat{\theta}\paren{\widehat{x}}$ is the Hessian matrix of $\widehat{\theta}$ at $\widehat{x}$ in the Euclidean sense, whose $\paren{i,j}$-th element is $\frac{\partial^{2} \widehat{\theta}\paren{\widehat{x}}}{\partial e_{j} \partial e_{i}}\in\eucli$. Note that $\widehat{\Gamma}_{\widehat{x}}\sbra{\text{D}\widehat{\theta}\paren{\widehat{x}}}$ is symmetric by the symmetry of the Christoffel symbols.

\subsection{Optimality conditions for RNLO\label{Pre:RiemOpt}}
We define $\mathcal{L}\paren{x,\mu,\lambda} \coloneqq f\paren{x} \allowbreak  + \allowbreak \sum_{i\in \mathcal{I}} \mu_{i}g_{i}\paren{x} \allowbreak + \allowbreak \sum_{j\in\mathcal{E}}\lambda_{j}h_{j}\paren{x}$ for $x\in \mani$, $\mu\in \eucli[m]$, and $\lambda\in \eucli[n]$. The function $\mathcal{L}$ is called the Lagrangian of RNLO~\eqref{RNLO} and $\mu\in\eucli[m]$ and $\lambda\in\eucli[n]$ are called Lagrange multipliers for the inequality and equality constraints, respectively. For given $\mu\in \eucli[m]$ and $\lambda\in \eucli[n]$, we will often write $\mathcal{L}_{\mu,\lambda}\paren{x} \coloneqq \mathcal{L}\paren{x,\mu,\lambda}$ for $x\in \mani$. Let $\Omega$ denote the set of feasible points of RNLO \eqref{RNLO}. For $x\in\Omega$, let $\mathcal{I}_{a}\paren{x}$ denote the index set that corresponds to the active inequality constraints at $x\in\Omega$, that is, $\mathcal{I}_{a}\paren{x} \coloneqq \brc{i\in\mathcal{I}\relmiddle{|}g_{i}\paren{x}=0}$.

\begin{definition} {\upshape (\cite[eq. (4.3)]{Yangetal14OptCond})}
    We say that the linear independence constraint qualification (LICQ) holds at $x\in\Omega$ if
    \begin{align*}
        \brc{\text{\upshape grad}\,g_{i}\paren{x}, \text{\upshape  grad}\,h_{j}\paren{x}}_{i\in\mathcal{I}_{a}\paren{x},j\in\mathcal{E}}\text{ are linearly independent in } \tanspc\mani.
    \end{align*}
\end{definition}

\begin{definition} {\upshape (\cite[eq. (4.8)]{Yangetal14OptCond})} 
    We say that $x^{\ast}\in\Omega$ satisfies the Karush-Kuhn-Tucker conditions (KKT conditions) of RNLO \eqref{RNLO} if there exist Lagrange multipliers $\mu^{\ast}\in\eucli[m]$ and $\lambda^{\ast}\in\eucli[n]$ such that the following hold:
    \begin{subequations}\label{KKT}
	    \begin{align}
	        &\text{\upshape grad}\,f\paren{x^{\ast}} + \sum_{i\in\mathcal{I}} \mu_{i}^{\ast} \text{\upshape grad}\,g_{i}\paren{x^{\ast}} + \sum_{j\in\mathcal{E}}\lambda_{j}^{\ast} \text{\upshape grad}\,h_{j}\paren{x^{\ast}} = 0,\label{KKTLag} \\
	        &\mu^{\ast}_i\geq 0, \, g_{i}\paren{x^{\ast}}\leq 0, \text{ and} \label{KKTineq}\\
	        &\mu^{\ast}_{i}g_{i}\paren{x^{\ast}}=0,\text{ for all } i\in\mathcal{I},\label{KKTcompl}\\
	        &h_{j}\paren{x^{\ast}} = 0,\text{ for all } j\in\mathcal{E}.\label{KKTeq}
	    \end{align}
    \end{subequations}
    We call $x^{\ast}$ a KKT point of RNLO \eqref{RNLO} and refer to $\paren{x^{\ast}, \mu^{\ast}, \lambda^{\ast}}$ as a KKT triplet of RNLO \eqref{RNLO}.
\end{definition}

\begin{proposition} {\upshape (\cite[Theorem 4.1]{Yangetal14OptCond})}
    Suppose that $x^{\ast}\in\Omega$ is a local minimum of RNLO~\eqref{RNLO} and that the LICQ holds at $x^{\ast}$. Then, $x^{\ast}$ satisfies the KKT conditions.
\end{proposition}

\begin{definition} {\upshape (\cite[Theorem 4.3]{Yangetal14OptCond})}
    We say that a feasible point $x^{\ast}\in\Omega$ satisfies the second-order sufficient conditions (SOSCs) if the KKT conditions hold at $x^{\ast}$ with associated Lagrange multipliers $\mu^{\ast}$ and $\lambda^{\ast}$, and
    \begin{align*}
        \metr[x^{\ast}]{\text{\rm Hess}\,\mathcal{L}_{\mu^{\ast},\lambda^{\ast}}\paren{x^{\ast}}\sbra{\xi}}{\xi}>0,\quad  \forall\xi\in\mathcal{F}\paren{x^{\ast},\mu^{\ast},\lambda^{\ast}}\backslash\brc{0},
    \end{align*}
    where
	\begin{align*}
		\mathcal{F}\paren{x^{\ast}, \mu^{\ast}, \lambda^{\ast}}
		\coloneqq \brc{\xi\in\tanspc[x^{\ast}]\mani \relmiddle{|}
		\begin{aligned}
			&\metr[]{\xi}{\text{\rm grad}\,h_{j}\paren{x^{\ast}}} = 0, \text{ for all } j\in\mathcal{E},\\
			&\metr[]{\xi}{\text{\rm grad}\,g_{i}\paren{x^{\ast}}} = 0, \text{ for all } i\in\mathcal{I}_{a}\paren{x^{\ast}}\text{ with }\mu^{\ast}_{i}>0,\\
			&\metr[]{\xi}{\text{\rm grad}\,g_{i}\paren{x^{\ast}}} \leq 0, \text { for all } i\in\mathcal{I}_{a}\paren{x^{\ast}}\text{ with }\mu^{\ast}_{i}=0.\\
		\end{aligned}
		}.
	\end{align*}
\end{definition}
 
\begin{definition}
    Given a point $x^{\ast}\in\Omega$ that satisfies the KKT conditions with associated Lagrange multipliers $\mu^{\ast}$ and $\lambda^{\ast}$, we say that the strict complementary condition (SC) holds if exactly one of $\mu^{\ast}_{i}$ and $g_{i}\paren{x^{\ast}}$ is zero for each index $i\in\mathcal{I}$. Hence, under the SC, we have $\mu^{\ast}_i>0$ for each $i\in\mathcal{I}_{a}\paren{x^{\ast}}$.
\end{definition}

\section{Sequential quadratic optimization on a Riemannian manifold}\label{method}
\subsection{Description of proposed algorithm}
Sequential quadratic optimization (SQO), or sequential quadratic programming (SQP), is a well-known iterative method for constrained nonlinear optimization in a Euclidean space~\cite{NocedalWright06NumOpt, BoggsTolle96SQPsurvey}. In this section, we extend it to a Riemannian manifold $\mani$ and analyze its global and local convergence properties. The proposed algorithm is referred to as Riemannian SQO, or RSQO for short. In contrast, we will often refer to SQO methods in a Euclidean space as Euclidean SQO methods.

Let $x_{k}\in\mani$ be a current iterate. In RSQO, we solve the following subproblem at $x_k$ to have a search direction\footnote{Although we adopt the exact solution of \eqref{QP} as the search direction in this paper, practically it can be hard to calculate the exact one depending on the problem size and an algorithm for solving the subproblem. In Section~\ref{conclusion}, we consider the issue with Euclidean SQO methods that take the inexactness into account.}:
\begin{align}
    \begin{split}\label{QP}
        \underset{\Delta x_{k} \in \tanspc[x_{k}]\mani}{\text{minimize}} \quad & \frac{1}{2} \metr[]{B_{k}\sbra{\Delta x_{k}}}{\Delta x_{k}}  + \metr[]{\text{grad}\,f\paren{x_{k}}}{\Delta x_{k}} \\
        \text{subject to}\quad  & g_{i}\paren{x_{k}} + \metr[]{\text{grad}\,g_{i}\paren{x_{k}}}{\Delta x_{k}} \leq 0, \text{ for all }i \in \mathcal{I},\\ 
        & h_{j}\paren{x_{k}} + \metr[]{\text{grad}\,h_{j}\paren{x_{k}}}{\Delta x_{k}} = 0, \text{ for all } j \in \mathcal{E},
    \end{split}
\end{align}
where $B_{k}: \tanspc[x_{k}]\mani\rightarrow\tanspc[x_{k}]\mani$ is a linear operator that is assumed to be symmetric and positive-definite, that is,
\begin{align*}
    &\text{symmetry: }\metr[]{B_{k}\sbra{\xi}}{\zeta} = \metr[]{\xi}{B_{k}\sbra{\zeta}},\quad \forall \xi,\zeta \in \tanspc[x_{k}]\mani, \\
    &\text{positive-definiteness: }\metr[]{B_{k}\sbra{\xi}}{\xi} > 0,\quad \forall \xi \in \tanspc[x_{k}]\mani\backslash \brc{0_{x_{k}}}.
\end{align*}
    
In the subproblem, the constraints in \eqref{QP} are linearizations of the original ones at $x_{k}$. As for the objective function, in principle, we can set any linear and symmetric operator to $B_{k}$, for example, the identity operator on $\tanspc[x_{k}]\mani$, as long as Assumption~A\ref{Bkbounded} that will appear later is satisfied. Yet, the Hessian of the Lagrangian or its approximation is preferable for the sake of rapid convergence. By virtue of the positive-definiteness of $B_k$, \eqref{QP} is strongly convex and thus has a unique optimum, say $\Delta x_{k}^{\ast}$, if it is feasible. In terms of the coordinate expression, we can transform the subproblem into a certain quadratic optimization problem on $\eucli[d]$ such that the optimum is $\widehat{\Delta x_{k}^{\ast}}\in \eucli[d]$, the coordinate expression of $\Delta x_{k}^{\ast}$, which can be computed using existing algorithms such as an interior-point method~\cite{NocedalWright06NumOpt}. See Section~\ref{subsection:experimental_environment} for a more specific manner of organizing the subproblem. Since the constraints of \eqref{QP} are formed by affine functions defined on $\tanspc[x_{k}]\mani$, the KKT conditions hold at the optimum $\Delta x_{k}^{\ast}$ in the absence of constraint qualifications. Hence, we ensure that the equality and inequality constraints of \eqref{QP} have Lagrange multiplier vectors $\lambda_k^{\ast}$ and $\mu_k^{\ast}$, respectively, which compose the KKT conditions for \eqref{QP}.

$\RSQO$ employs the optimum $\Delta x_{k}^{\ast}$ as the search direction. In the ordinary Euclidean SQO method equipped with a line-search technique, the next iterate $x_{k+1}$ is defined by $x_k+\alpha_k\Delta x_k^{\ast}$ with an appropriate step length $\alpha_k>0$. However, in our Riemannian setting, $x_{k+1}\in \mani$ cannot be generated in this way because the sum operation $\paren{x_{k}, \Delta x^{\ast}_{k}} \mapsto x_k+\alpha_k\Delta x_k^{\ast}$ is not generally defined between the different spaces $\mani$ and $\tanspc[x_k]\mani$. To circumvent this difficulty, we utilize a retraction $R$ and set $x_{k+1}=R_{x_k}(\alpha_k\Delta x_k^{\ast})$. For the definition of $R$, see \eqref{def:retr}.

Next, we explain how the step length $\alpha_{k}$ is computed. Similar to the Euclidean SQO method, we make use of the following $\ell_1$ penalty function defined on $\mani$ as a merit function, which was first introduced together with the Riemannian penalty methods by Liu and Boumal~\cite{LiuBoumal19Simple}:
\begin{align}\label{ell1merit}
    P_{\rho_k}\paren{x} \coloneqq f\paren{x} + \rho_{k} \paren{\sum_{i\in\mathcal{I}} \max\brc{0,g_{i}\paren{x}} + \sum_{j\in\mathcal{E}}\abs{h_{j}\paren{x}}},
\end{align}
where $\rho_k>0$ is a penalty parameter. The parameter $\rho_k$ is determined from the previous one $\rho_{k-1}$ and the Lagrange multiplier vectors $\lambda_{k}^{\ast}$ and $\mu_{k}^{\ast}$ obtained by solving subproblem~\eqref{QP}. Specifically, we set 
\begin{align}\label{updaterho}
    \rho_k \coloneqq
    \begin{cases}
    \rho_{k-1}, &\mbox{if {$\rho_{k-1} \geq \upsilon_k$}},\\
    \upsilon_{k} + \varepsilon, &\mbox{otherwise} 
    \end{cases}
\end{align}
with $\upsilon_{k}:=\max\brc{\max_{i\in\mathcal{I}}\mu_{ki}^{\ast},\max_{j\in\mathcal{E}}\abs{\lambda_{kj}^{\ast}}}$ and $\varepsilon>0$ being a prescribed algorithmic parameter. The step length $\alpha_k$ is then determined in accordance with a backtracking line search 
using the composite function $P_{\rho_k}\circ R_{x_{k}}\paren{\cdot}$ along with $\Delta x_k^{\ast}$: we find the smallest nonnegative integer $r$ such that
\begin{align}\label{Armijorule}
    \gamma \beta^{r}\metr[]{B_{k}\sbra{\Delta x_{k}^{\ast}}}{\Delta x_{k}^{\ast}} \leq P_{\rho_{k}} \paren{x_{k}} - P_{\rho_{k}}\circ R_{x_{k}}\paren{\beta^{r} \Delta x_{k}^{\ast}}
\end{align}
and set $\alpha_k=\beta^r$. Procedure~\eqref{Armijorule} is well-defined in the sense that we can always find $r$ within a finite number of trials, as is verified in Remark~\ref{armijowelldefined}. In addition, the Lagrange multipliers are updated by $(\lambda_{k+1},\mu_{k+1})=(\lambda_{k}^{\ast},\mu_{k}^{\ast})$. Algorithm~\ref{SQPonmani} formally states the procedure of $\RSQO$.

\begin{algorithm}[t]
    \caption{Riemannian sequential quadratic optimization (RSQO)} 
    \label{SQPonmani}
    \begin{algorithmic}          
        \REQUIRE Riemannian manifold $\mani$, Riemannian metric$\metr[]{\cdot}{\cdot}$, twice continuously differentiable functions $f,\brc{g_{i}}_{i\in \mathcal{I}},\brc{h_{j}}_{j\in \mathcal{E}}:\mathcal{M}\rightarrow\eucli$, merit function $P_{\rho}:\mani\rightarrow\eucli$, retraction $R: \tanspc[]\mani\rightarrow\mani$, $\varepsilon>0$, $\rho_{-1}>0$, $\beta\in\paren{0,1}$, $\gamma\in\paren{0,1}$.
        \ENSURE Initial iterate $x_{0} \in \mani$, initial linear operator $B_{0}:\tanspc[x_{0}]\mani\rightarrow\tanspc[x_{0}]\mani$.
        \FOR  {$k=0,1,\ldots$}
        \STATE Compute $\Delta x_{k}^{\ast}$ -- a solution to \eqref{QP} with Lagrange multipliers $\mu_{k}^{\ast}$ and $\lambda_{k}^{\ast}$;
        \STATE Update $\rho_{k}$ according to \eqref{updaterho};
        \STATE Determine the integer $r$ according to the backtracking line search
        \eqref{Armijorule} and set $\alpha_{k} = \beta^{r}$;
        \STATE Update $x_{k+1} = R_{x_{k}}\paren{\alpha_{k}\Delta x_{k}^{\ast}}$, $\mu_{k+1} = \mu_{k}^{\ast}$, and $\lambda_{k+1} = \lambda_{k}^{\ast}$;
        \STATE Set $B_{k+1}:\tanspc[x_{k+1}]\mani\rightarrow\tanspc[x_{k+1}]\mani$;
        \ENDFOR
    \end{algorithmic}
\end{algorithm}

In fact, under the Euclidean setting with $\mani\coloneqq\eucli[d]$ and $R$ being defined with a straight line, the proposed RSQO becomes a basic SQO or SQP, found in many textbooks, e.g., \cite{Connetal00trustregionbook}. Accordingly, the lines of the convergence analyses of the RSQO
are analogous to those of Euclidean SQO: as for the global convergence analysis, we prove that $\Delta x^{\ast}_{k}$ is a descent direction of the merit function as well as Han~\cite{Han77GlobConvSQPwithExactPenalty} did in the Euclidean setting. The local convergence analysis is also conducted in a fashion similar to \cite{BoggsTolle96SQPsurvey, NocedalWright06NumOpt} for the Euclidean SQO. Nevertheless, the theoretical results we will establish are nontrivial because of difficulties peculiar to the Riemanian setting. Especially, we introduce new terminologies from Riemannian geometry, such as parallel transport, along with several tools from nonsmooth optimization so as to prove a certain inequality; see Proposition~\ref{LimBIneqLimP} for the detail.

Let us end this subsection by presenting a simpler form of the KKT conditions for subproblem \eqref{QP}.
\begin{lemma}
    The KKT conditions \eqref{KKTLag} -- \eqref{KKTeq} for subproblem \eqref{QP} at $\Delta x^{\ast}_{k} \in \tanspc[x_{k}]\mani$ are equivalent to the following conditions with $\mu_{k}^{\ast}\in\eucli[m]$ and $\lambda_{k}^{\ast}\in\eucli[n]$:
    \begin{subequations}\label{subprobKKT}
        \begin{align}
            &B_{k}\sbra{\Delta x^{\ast}_{k}} + \text{\upshape grad}\,f\paren{x_{k}} + \sum_{i\in\mathcal{I}} \mu_{ki}^{\ast} \text{\upshape grad}\,g_{i}\paren{x_{k}} + \sum_{j\in\mathcal{E}}\lambda_{kj}^{\ast} \text{\upshape grad}\,h_{j}\paren{x_{k}} = 0, \label{subprobKKTLagsimple}\\
            &\mu^{\ast}_{ki}\geq 0, \, g_{i}\paren{x_{k}} + \metr[]{\text{\upshape grad}\,g_{i}\paren{x_{k}}}{\Delta x^{\ast}_{k}}\leq 0,\text{ and}\label{subprobKKTineq}\\
            &\mu^{\ast}_{ki}\paren{g_{i}\paren{x_{k}} + \metr[]{\text{\upshape grad}\,g_{i}\paren{x_{k}}}{\Delta x^{\ast}_{k}}} = 0, \text{ for all } i\in\mathcal{I},\label{subprobKKTcompl}\\
            & h_{j}\paren{x_{k}} + \metr[]{\text{\upshape grad}\,h_{j}\paren{x_{k}}}{\Delta x^{\ast}_{k}} = 0, \text{ for all } j\in\mathcal{E},\label{subprobKKTeq}
        \end{align}
    \end{subequations}
\end{lemma}

\begin{proof}
    Conditions \eqref{subprobKKTineq}, \eqref{subprobKKTcompl}, and \eqref{subprobKKTeq} directly follow from \eqref{KKTineq}, \eqref{KKTcompl}, and \eqref{KKTeq}, respectively. As for \eqref{subprobKKTLagsimple}, we will start by describing the original form \eqref{KKTLag} of subproblem \eqref{QP} by noting the translations $\metr[]{\text{grad}f\paren{x_{k}}}{\Delta x^{\ast}_{k}} = \text{D}f\paren{x_{k}}\sbra{\Delta x^{\ast}_{k}}$, $\metr[]{\text{grad}g_{i}\paren{x_{k}}}{\Delta x^{\ast}_{k}} = \text{D}g_{i}\paren{x_{k}}\sbra{\Delta x^{\ast}_{k}}$ for $i\in\mathcal{I}$, and $\metr[]{\text{grad}h_{j}\paren{x_{k}}}{\Delta x^{\ast}_{k}} = \text{D}h_{j}\paren{x_{k}}\sbra{\Delta x^{\ast}_{k}}$ for $j\in\mathcal{E}$ from \eqref{def:grad}:
    \begin{align}\label{subprobKKTLagorgn}
        &\text{grad}\paren{\frac{1}{2} \langle B_{k}\rangle + \text{D}f\paren{x_{k}} + \sum_{i\in\mathcal{I}}\mu^{\ast}_{ki}\text{D}g_{i}\paren{x_{k}} + \sum_{j\in\mathcal{E}} \lambda_{kj}^{\ast}\text{D}h_{j}\paren{x_{k}}}\paren{\Delta x_{k}^{\ast}} = 0,
    \end{align}
    where $\langle B_{k}\rangle\paren{\Delta x_{k}}\coloneqq \metr[]{B_{k}\sbra{\Delta x_{k}}}{\Delta x_{k}}$ for $\Delta x_{k}\in \tanspc[x_k]\mani$. Define $F_{k}:\tanspc[x_{k}]\mani\rightarrow\eucli[]$ by
    \begin{align*}
        F_{k}(\Delta x_k) \coloneqq \text{D}f\paren{x_{k}}\sbra{\Delta x_{k}} + \sum_{i\in\mathcal{I}}\mu^{\ast}_{ki}\text{D}g_{i}\paren{x_{k}}\sbra{\Delta x_{k}} + \sum_{j\in\mathcal{E}} \lambda^{\ast}_{kj}\text{D}h_{j}\paren{x_{k}}\sbra{\Delta x_{k}}
    \end{align*}
    for $\Delta x_k\in \tanspc[x_k]\mani$. By taking the inner product on the tangent space and using \eqref{def:grad} with $\paren{\theta, \mani, x}$ replaced by $\paren{\frac{1}{2}\langle B_{k}\rangle + F_{k}, \, \tanspc[x_{k}]\mani, \, \Delta x^{\ast}_{k}}$, \eqref{subprobKKTLagorgn} is equivalent to 
    \begin{align}\label{ToSimplesubKKTLag}
        \frac{1}{2}\text{D}\langle B_{k}\rangle\paren{\Delta x_{k}^{\ast}}\sbra{\plchold} \allowbreak + \text{D}F_{k}\paren{\Delta x_{k}^{\ast}}\sbra{\plchold} = 0,
    \end{align} 
    where the left-hand side is a mapping from $\tanspc[\Delta x^{\ast}_{k}]\paren{\tanspc[x_{k}]\mani}$ to $\eucli[]$. Since $F_{k}$ is linear, we can identify $\text{D}F_{k}\paren{\Delta x_{k}^{\ast}}$ with $F_{k}$ under $\tanspc[\Delta x^{\ast}_{k}]\paren{\tanspc[x_{k}]\mani} \simeq \tanspc[x_{k}]\mani$. Moreover, it follows that $\frac{1}{2}\text{D}\langle B_{k}\rangle\paren{\Delta x_{k}^{\ast}}\sbra{\plchold} \allowbreak = \metr[]{B_{k}\sbra{\Delta x^{\ast}_{k}}}{\plchold}$ under $\tanspc[\Delta x^{\ast}_{k}]\paren{\tanspc[x_{k}]\mani} \simeq \tanspc[x_{k}]\mani$, which is proved as follows: we actually have
    \begin{align*}
        \widehat{\langle B_{k}\rangle}\paren{\widehat{\Delta x_{k}}}
        &=\widehat{\Delta x_{k}}^{\top}\coordmetr[x_{k}]\paren{ \text{D}\varphi\paren{x_{k}}\sbra{B_{k}\sbra{\Delta x_{k}}}}\\
        &=\widehat{\Delta x_{k}}^{\top}\coordmetr[x_{k}]\paren{ \text{D}\varphi\paren{x_{k}}\circ B_{k}\circ \paren{\text{D}\varphi\paren{x_{k}}}^{-1} \circ \text{D}\varphi\paren{x_{k}}\sbra{\Delta x_{k}} }\\
        &= \widehat{\Delta x_{k}}^{\top}\coordmetr[x_{k}]\widehat{B_{k}}\widehat{\Delta x_{k}},
    \end{align*}
    where $\widehat{B_{k}}\coloneqq \text{D}\varphi\paren{x_{k}}\circ B_{k}\circ \paren{\text{D}\varphi\paren{x_{k}}}^{-1}$. Note that $\widehat{B_{k}}$ is a linear operator from $\eucli[d]$ to $\eucli[d]$, that is, a $d\times d$ matrix. Thus under $\tanspc[\Delta x^{\ast}_{k}]\paren{\tanspc[x_{k}]\mani} \simeq \tanspc[x_{k}]\mani$, we have
    \begin{align*}
        \frac{1}{2}\text{D}\langle B_{k}\rangle\paren{\Delta x^{\ast}_{k}}\sbra{\xi}
        &= \frac{1}{2}\text{D}\widehat{\langle B_{k}\rangle}\paren{\widehat{\Delta x^{\ast}_{k}}}^{\top}\widehat{\xi}\\
        &= \frac{1}{2} \widehat{\Delta x^{\ast}_{k}}^{\top}\paren{\widehat{B_{k}}^{\top}\coordmetr[x_{k}]+\coordmetr[x_{k}]\widehat{B_{k}}}\widehat{\xi}\\
        &= \widehat{\Delta x^{\ast}_{k}}^{\top}\widehat{B_{k}}^{\top}\coordmetr[x_{k}]\widehat{\xi}\\
        &=\metr[]{B_{k}\sbra{\Delta x^{\ast}_{k}}}{\xi}
    \end{align*}
    for any $\xi\in\tanspc[x_{k}]\mani$, where the third equality derives from the fact that ${\widehat{B_{k}}}^{\top}\coordmetr[x_{k}] = \coordmetr[x_{k}]\widehat{B_{k}}$, since, by the symmetry of the linear operator $B_{k}$, we obtain ${\widehat{\xi}}^{\top}\widehat{B_{k}}^{\top}\coordmetr[x_{k}]\widehat{\zeta} = \widehat{\xi}^{\top}\coordmetr[x_{k}]\widehat{B_{k}}\widehat{\zeta}$ for any $\widehat{\xi},\widehat{\zeta}\in\eucli[d]$. 
    
    Hence, combining the above with \eqref{ToSimplesubKKTLag} yields $\metr[]{B_{k}\sbra{\Delta x^{\ast}_{k}}}{\plchold} + F_{k}\sbra{\cdot} = 0$. Thus, we have $\metr[]{B_{k}\sbra{\Delta x^{\ast}_{k}}}{\plchold} + \text{D}f\paren{x_{k}}\sbra{\plchold} + \sum_{i\in\mathcal{I}}\mu^{\ast}_{ki}\text{D}g_{i}\paren{x_{k}}\sbra{\plchold} + \sum_{j\in\mathcal{E}} \lambda^{\ast}_{kj}\text{D}h_{j}\paren{x_{k}}\sbra{\plchold} = 0$, which is equivalent to condition \eqref{subprobKKTLagsimple} by recalling \eqref{def:grad}.
\end{proof}

\subsection{Global convergence}\label{subsec:globalconv}
In this subsection, we prove that $\RSQO$ has the global convergence property under the following assumptions:
\begin{enumerate}[\bf{A}1]
    \item Subproblem \eqref{QP} is feasible at every iteration. \label{subprobfeasi}
    \item There exist $m>0$ and $M>0$ such that, for any $k$, $m \norm{\xi}^{2} \leq \metr[]{B_{k}\sbra{\xi}}{\xi} \leq M \norm{\xi}^{2}$ holds for all $\xi \in \tanspc[x_{k}]\mani$.\label{Bkbounded}
    \item The generated sequence $\brc{\paren{x_{k},\mu_{k},\lambda_{k}}}$ is bounded.\label{genebound}
\end{enumerate}
As for Assumption A\ref{subprobfeasi}, the following sufficient condition holds.
\begin{proposition}\label{prop:SuffCondforSubProbFeasi}
    Suppose that RNLO \eqref{RNLO} has a feasible solution $\bar{x}\in\Omega$, $g_{i}$ is a geodesically convex function for all $i\in\mathcal{I}$, and $h_{j}$ is a geodesically linear function for all $j\in\mathcal{E}$. Then, Assumption A\ref{subprobfeasi} holds.
\end{proposition}
\begin{proof}
See Appendix~\ref{SubsecAppen:PrfofPropSuffcond} for the proof and the definitions of geodesically convex and linear functions.
\end{proof}

Assumption~A\ref{Bkbounded} is fulfilled with the identity mapping on $\tanspc[x_{k}]\mani$, for example. Assumption A\ref{genebound} often appears in the literature on Euclidean SQO methods~\cite{KatoFukushima07SQPforNSOCP, Fukushima86SQPavoidMaratos, NocedalWright06NumOpt}. If $\mani$ is compact such as a sphere and the Stiefel manifold, the boundedness of $\brc{x_{k}}$ in Assumption~A\ref{genebound} holds since $\mani$ itself is bounded by the Hopf-Rinow theorem. Note that, however, these assumptions are mitigated in the state-of-the-art SQO methods in Euclidean spaces. In Section~\ref{conclusion}, we discuss techniques to weaken the assumptions.

Recall that \eqref{QP} is a convex optimization problem in the tangent space whose objective function is strongly convex and constraints are all affine functions. Thus, under Assumption A\ref{subprobfeasi}, \eqref{QP} has the unique optimum $\Delta x_{k}^{\ast}$ and the KKT conditions for \eqref{QP} become a certificate for $\Delta x^{\ast}_{k}$ to be a global optimum.

The following lemma will be used combined with Assumption~A\ref{Bkbounded}.
\begin{lemma}\label{AkOperNormBounded}
    Let $x\in\mani$ and $\mathcal{A}_{x}:\tanspc[x]\mani\rightarrow\tanspc[x]\mani$ be a symmetric positive-definite linear operator. Suppose that the uniform positive-definiteness of $\mathcal{A}_{x}$ holds; that is, there exist $m,M>0$ such that $m \norm{\xi}^{2} \leq \metr{\mathcal{A}_{x}\sbra{\xi}}{\xi} \leq M \norm{\xi}^{2}$ for all $\xi\in\tanspc[x]\mani$. Then, it follows that
    \begin{align*}
        \norm{\mathcal{A}_{x}}_{\rm op} \leq M\sqrt{d} \text{ and } \norm{\mathcal{A}_{x}^{-1}}_{\rm op} \leq \frac{\sqrt{d}}{m},
    \end{align*}
    where $d$ is the dimension of $\mani$ and $\norm{\cdot}_{\rm op}$ denotes the operator norm on $\tanspc\mani$. \footnote{Given a tangent space $\tanspc[x]\mani$ and a linear mapping $\mathcal{A}_{x}:\tanspc[x]\mani\rightarrow\tanspc[x]\mani$, we define the operator norm as $\norm{\mathcal{A}_{x}}_{\rm op} \coloneqq \sup_{\xi\in\tanspc[x]\mani} \norm{\mathcal{A}_{x}\sbra{\xi}}_{x} \slash \norm{\xi}_{x}$, where $\norm{\cdot}_{x}$ is the norm on $\tanspc[x]\mani$. In this paper, although the operator norm differs depending on the tangent space, we will use the same notation $\norm{\cdot}_{\rm op}$ for brevity.}
\end{lemma}
\begin{proof}
See Appendix~\ref{appendix:prfofAkOperNormBounded}.
\end{proof}

Lemma~\ref{AkOperNormBounded} implies the boundedness of the search directions.
\begin{proposition}\label{DeltaxkBounded}
    Under Assumptions~A\ref{subprobfeasi}, A\ref{Bkbounded}, and A\ref{genebound}, $\brc{\norm{\Delta x^{\ast}_{k}}}$ is bounded.
\end{proposition}

\begin{proof}
    Since $\paren{\Delta x_{k}^{\ast}, \mu^{\ast}_{k}, \lambda^{\ast}_{k}}$ satisfies \eqref{subprobKKTLagsimple} for every $k$, we have
    \begin{align*}
        \begin{split}
            \norm{\Delta x_{k}^{\ast}}_{x_{k}}
            &= \norm{- B_{k}^{-1}\sbra{\text{\upshape grad}\,f\paren{x_{k}} + \sum_{i\in\mathcal{I}} \mu_{ki}^{\ast} \text{\upshape grad}\,g_{i}\paren{x_{k}} + \sum_{j\in\mathcal{E}}\lambda_{kj}^{\ast} \text{\upshape grad}\,h_{j}\paren{x_{k}}}}\\
            &\leq \norm{B_{k}^{-1}}_{\rm op} \norm{\text{\upshape grad}\,f\paren{x_{k}} + \sum_{i\in\mathcal{I}} \mu_{ki}^{\ast} \text{\upshape grad}\,g_{i}\paren{x_{k}} + \sum_{j\in\mathcal{E}}\lambda_{kj}^{\ast} \text{\upshape grad}\,h_{j}\paren{x_{k}}}\\
            &\leq \frac{\sqrt{d}}{m} \norm{\text{\upshape grad}\,f\paren{x_{k}} + \sum_{i\in\mathcal{I}} \mu_{ki}^{\ast} \text{\upshape grad}\,g_{i}\paren{x_{k}} + \sum_{j\in\mathcal{E}}\lambda_{kj}^{\ast} \text{\upshape grad}\,h_{j}\paren{x_{k}}},
        \end{split}
    \end{align*}
    where the last inequality follows from Lemma~\ref{AkOperNormBounded}. From Assumption A\ref{genebound} and the continuity of $\brc{\text{grad}g_{i}}_{i\in\mathcal{I}}$ and $\brc{\text{grad}h_{j}}_{j\in\mathcal{E}}$, the above inequality implies the boundedness of $\brc{\norm{\Delta x^{\ast}_{k}}}$. The proof is complete.
\end{proof}

The following proposition ensures that the penalty parameter $\rho_{k}$ eventually reaches a constant under Assumption~A\ref{genebound}.
\begin{proposition}\label{prop:ensurebarrho}
    Under Assumption~A\ref{genebound}, there exist $\tilde{k}_{1}\in\mathbb{N}$ and $\bar{\rho}\in\eucli$ such that $\rho_{k}=\bar{\rho}$ holds for any $k\geq \tilde{k}_{1}$.
\end{proposition}
\begin{proof}
    From Assumption A\ref{genebound} with rule \eqref{updaterho} for updating the penalty parameter, it follows that $\brc{\rho_{k}}$ is monotonically nondecreasing and bounded. This implies that $\brc{\rho_{k}}$ converges to some real value denoted by $\bar{\rho}$. Moreover since $\rho_k$ increases by at least prefixed $\varepsilon > 0$ when it does, the assertion is ensured.
\end{proof}

We next prove Proposition~\ref{prop:meritineq} that asserts $\Delta x_{k}^{\ast}$ is a descent direction for the merit function $P_{\rho_{k}} \circ R_{x_{k}}\paren{\cdot}$ with $\rho_{k}$ sufficiently large when $\norm{\Delta x_{k}^{\ast}} \neq 0$. To this end, we first present the three lemmas; Lemmas~\ref{lemm:DDsigmacases}, \ref{lemm:DDsigmaIneq}, and \ref{lemm:DDtau}. To prove Lemmas~\ref{lemm:DDsigmaIneq} and \ref{lemm:DDtau}, we exploit the specific properties of the retraction together with a chain rule. This fact is worth mentioning as a peculiar manner to the Riemannian setting.

Let us define the functions $\sigma_{i x_k}:\tanspc[x_{k}]\mani\rightarrow\eucli$ for $i\in\mathcal{I}$ and $\tau_{x_{k}}:\tanspc[x_{k}]\mani\rightarrow\eucli$ by
\begin{align*}
    \sigma_{i x_k}\paren{\zeta} \coloneqq \max\brc{0,g_{i}\circ R_{x_{k}}\paren{\zeta}} \text{ for all $i\in\mathcal{I}$} \text{ and }  \tau_{x_{k}}\paren{\zeta} \coloneqq \sum_{j\in\mathcal{E}}\abs{h_{j} \circ R_{x_{k}}\paren{\zeta}}
\end{align*}
for $\zeta \in\tanspc[x_{k}]\mani$. Note that these functions are continuous, but not differentiable. The next lemma shows specific formulae of the one-sided directional derivative of $\sigma_{i x_{k}}$ for $i\in\mathcal{I}$.
\begin{lemma}\label{lemm:DDsigmacases}
For any $i\in\mathcal{I}$ and all $\zeta, \xi \in \tanspc[x_{k}]\mani$, the one-sided directional derivative of $\sigma_{i x_{k}}$ at $\zeta$ along $\xi$ is given by
    \begin{align*}
        \sigma^{\prime}_{i x_{k}}\paren{\zeta;\xi} = 
        \begin{cases}
            \paren{g_{i}\circ R_{x_{k}}}^{\prime}\paren{\zeta; \xi}, &
            \begin{aligned}
                &\text{ if } g_{i}\circ R_{x_{k}}\paren{\zeta} > 0 \text{ or}\\
                &\text{ if }g_{i}\circ R_{x_{k}}\paren{\zeta} = 0 \text{ and } \paren{g_{i}\circ R_{x_{k}}}^{\prime}\paren{\zeta;\xi} \geq 0,
            \end{aligned}
            \\
            0, &
            \begin{aligned}
                &\text{ if } g_{i}\circ R_{x_{k}}\paren{\zeta} < 0 \text{ or}\\
                &\text{ if }g_{i}\circ R_{x_{k}}\paren{\zeta} = 0 \text{ and } \paren{g_{i}\circ R_{x_{k}}}^{\prime}\paren{\zeta;\xi} < 0.
            \end{aligned}
        \end{cases}
    \end{align*}
\end{lemma}
\begin{proof}
    Choose $i\in\mathcal{I}$ and $\zeta,\xi\in\tanspc[x_{k}]\mani$ arbitrarily. We will consider the following three cases.
    
    \begin{enumerate}[({\bf\casecom} 1)]
    \item If $ g_{i}\circ R_{x_{k}}\paren{\zeta} > 0$, then by taking sufficiently small $c>0$, we have $\sigma_{i x_{k}}\paren{\chi} = g_{i}\circ R_{x_{k}}\paren{\chi}$ for all $\chi \in \mathbb{B}_{c,x_{k}}\paren{\zeta}\coloneqq \brc{\chi\in\tanspc[x_{k}]\mani\relmiddle{|}\norm{\chi-\zeta}\leq c}$. Hence, it holds that $\sigma^{\prime}_{i x_{k}}\paren{\zeta;\xi} \allowbreak = \paren{g_{i}\circ R_{x_{k}}}^{\prime}\paren{\zeta; \xi}$.
    
    \item If $g_{i}\circ R_{x_{k}}\paren{\zeta} = 0$, then from the definition of the one-sided derivative, we have
    \begin{align*}
    \sigma^{\prime}_{i x_{k}}\paren{\zeta;\xi} 
    &= \lim_{t\downarrow 0}\frac{\max\brc{0,g_{i}\circ R_{x_{k}}\paren{\zeta+t\xi}}-0}{t}\\
    &= \max\brc{0, \lim_{t\downarrow 0}\frac{g_{i}\circ R_{x_{k}}\paren{\zeta+t\xi}- g_{i}\circ R_{x_{k}}\paren{\zeta}}{t}}\\
    &= \max\brc{0, \paren{g_{i}\circ R_{x_{k}}}^{\prime}\paren{\zeta; \xi}}\\
    &=\begin{cases}
             \paren{g_{i}\circ R_{x_{k}}}^{\prime}\paren{\zeta; \xi}, & \text{if } \paren{g_{i}\circ R_{x_{k}}}^{\prime}\paren{\zeta; \xi} \geq 0,\\
             0, & \text{otherwise}.
            \end{cases}
    \end{align*}
    
    \item Otherwise, if $g_{i}\circ R_{x_{k}}\paren{\zeta} < 0$, then considering a sufficiently small neighborhood of $\zeta$ in the same way as the first case, we have $\sigma^{\prime}_{i x_{k}}\paren{\zeta;\xi} =0$.
    \end{enumerate}
\end{proof}

In the following lemma, we prove an inequality on the one-sided directional derivative of $\sigma_{ix_{k}}$ by using Lemma~\ref{lemm:DDsigmacases}.
\begin{lemma}\label{lemm:DDsigmaIneq}
    Let $\paren{\Delta x_{k}^{\ast}, \mu_{k}^{\ast}, \lambda_{k}^{\ast}}$ be a KKT triplet satisfying \eqref{subprobKKT}. For any $i\in\mathcal{I}$, if $\rho \geq \mu_{ki}^{\ast}$, then
    \begin{align*}
        \mu_{ki}^{\ast}g_{i}\paren{x_{k}} + \rho \sigma^{\prime}_{i x_{k}}\paren{0_{x_{k}};\Delta x^{\ast}_{k}} \leq 0.
    \end{align*}
\end{lemma}
\begin{proof}
    Choose $i\in\mathcal{I}$ arbitrarily. Consider the following three cases.

    \begin{enumerate}[({\bf\casecom} 1)]
        \item If $g_{i}\circ R_{x_{k}}\paren{0_{x_{k}}} > 0$, then, since it follows from \eqref{def:retr} that
    \begin{align*}
        \paren{g_{i}\circ R_{x_{k}}}^{\prime}\paren{0_{x_{k}};\Delta x^{\ast}_{k}}
        = \text{D}g_{i}\paren{R_{x_{k}}\paren{0_{x_{k}}}}\sbra{\text{D}R_{x}\paren{0_{x}}\sbra{\Delta x_{k}^{\ast}}}
        =\text{D}g_{i}\paren{x_{k}}\sbra{\Delta x_{k}^{\ast}}
    \end{align*}
    under $\tanspc[0_{x_{k}}]\paren{\tanspc[x_{k}]\mani}\simeq\tanspc[x_{k}]\mani$, we have
    \begin{align*}
        &\mu_{ki}^{\ast}g_{i}\paren{x_{k}} + \rho \sigma^{\prime}_{i x_{k}}\paren{0_{x_{k}};\Delta x^{\ast}_{k}}\\
        &=\mu_{ki}^{\ast}g_{i}\paren{x_{k}} + \rho\paren{g_{i}\circ R_{x_{k}}}^{\prime}\paren{0_{x_{k}};\Delta x^{\ast}_{k}}\\        
        &= \mu_{ki}^{\ast}g_{i}\paren{x_{k}} + \rho \text{D}g_{i}\paren{x_{k}}\sbra{\Delta x_{k}^{\ast}}\\
        &= \mu_{ki}^{\ast}g_{i}\paren{x_{k}} + \rho\metr[]{\text{grad}\,g_{i}\paren{x_{k}}}{\Delta x_{k}^{\ast}}\\
        &= \mu_{ki}^{\ast}\paren{g_{i}\paren{x_{k}} + \metr[]{\text{grad}\,g_{i}\paren{x_{k}}}{\Delta x_{k}^{\ast}}} + \paren{\rho-\mu^{\ast}_{ki}}\metr[]{\text{grad}\,g_{i}\paren{x_{k}}}{\Delta x_{k}^{\ast}}\\
        &= \paren{\rho-\mu^{\ast}_{ki}}\metr[]{\text{grad}\,g_{i}\paren{x_{k}}}{\Delta x_{k}^{\ast}} \leq \paren{\mu^{\ast}_{ki}-\rho}g_{i}\paren{x_{k}}\\
        &\leq 0,
    \end{align*}
    where the first equality follows from Lemma~\ref{lemm:DDsigmacases} with $\zeta=0_{x_{k}}$ and $\xi = \Delta x_{k}^{\ast}$, the fifth equality holds by \eqref{subprobKKTcompl}, the first inequality follows from \eqref{subprobKKTineq} and the assumption $\rho \geq \mu_{ki}^{\ast}$, and the second inequality holds from the assumptions $g_{i}\paren{x_{k}} = g_{i}\circ R_{x_{k}}\paren{0_{x_{k}}} > 0$ and, again, $\rho \geq \mu_{ki}^{\ast}$.

    \item If $g_{i}\circ R_{x_{k}}\paren{0_{x_{k}}} = 0$ and $\paren{g_{i}\circ R_{x_{k}}}^{\prime}\paren{0_{x_{k}};\Delta x^{\ast}_{k}} \geq 0$, then we have 
    \begin{align*}
        \text{D}g_{i}\paren{x_{k}}\sbra{\Delta x_{k}^{\ast}} \allowbreak = \paren{g_{i}\circ R_{x_{k}}}^{\prime}\paren{0_{x_{k}};\Delta x^{\ast}_{k}} \geq 0.
    \end{align*}
    In addition, \eqref{def:grad} and \eqref{subprobKKTineq}, together with the assumption $g_{i}\paren{x_{k}} = g_{i}\circ R_{x_{k}}\paren{0_{x_{k}}} =0$, give 
    \begin{align*}
        \text{D}g_{i}\paren{x_{k}}\sbra{\Delta x_{k}^{\ast}} = \metr[]{\text{grad}\,g_{i}\paren{x_{k}}}{\Delta x_{k}^{\ast}} \leq 0.
    \end{align*}
    Therefore, the above two inequalities yield 
    \begin{align*}
        \text{D}g_{i}\paren{x_{k}}\sbra{\Delta x_{k}^{\ast}} = 0,
    \end{align*}
    which, together with Lemma~\ref{lemm:DDsigmacases} with $\zeta=0_{x_{k}}$ and $\xi = \Delta x_{k}^{\ast}$ and the assumption $g_{i}\paren{x_{k}} = 0$, implies
    \begin{align*}
        &\mu_{ki}^{\ast}g_{i}\paren{x_{k}} + \rho \sigma^{\prime}_{i x_{k}}\paren{0_{x_{k}};\Delta x^{\ast}_{k}}\\
        &=\mu_{ki}^{\ast}g_{i}\paren{x_{k}} + \rho \paren{g_{i}\circ R_{x_{k}}}^{\prime}\paren{0_{x_{k}};\Delta x^{\ast}_{k}}\\
        &=\mu_{ki}^{\ast}g_{i}\paren{x_{k}} + \rho \text{D}g_{i}\paren{x_{k}}\sbra{\Delta x_{k}^{\ast}}\\
        &= 0.
    \end{align*}

    \item Otherwise, if $g_{i}\circ R_{x_{k}}\paren{0_{x_{k}}} = g_{i}\paren{x_{k}} = 0$ and $\paren{g_{i}\circ R_{x_{k}}}^{\prime}\paren{0_{x_{k}};\Delta x^{\ast}_{k}} < 0$ hold simultaneously or $g_{i}\circ R_{x_{k}}\paren{0_{x_{k}}} = g_{i}\paren{x_{k}} < 0$ holds, then 
    \begin{align*}
        \mu_{ki}^{\ast}g_{i}\paren{x_{k}} + \rho  \sigma^{\prime}_{i x_{k}}\paren{0_{x_{k}};\Delta x^{\ast}_{k}} = \mu_{ki}^{\ast}g_{i}\paren{x_{k}} \leq 0,
    \end{align*}
    where the equality follows from Lemma~\ref{lemm:DDsigmacases} and the inequality holds under the assumption $g_{i}\paren{x_{k}}\leq 0$ and $\mu_{ki}^{\ast} \geq 0$ in \eqref{subprobKKTineq}.
    \end{enumerate}
\end{proof}

Next, we consider the one-sided directional derivative of $\tau_{x_{k}}$ at $0_{x_k}\in \tanspc[x_k]\mani$ along $\Delta x_k^{\ast}$.
\begin{lemma}\label{lemm:DDtau}
    Let $\Delta x_{k}^{\ast}$ be a KKT point satisfying \eqref{subprobKKT}. Then,
    \begin{align*}
        \tau_{x_{k}}^{\prime}\paren{0_{x_{k}};\Delta x^{\ast}_{k}} = - \sum_{j\in\mathcal{E}}\abs{h_{j}\paren{x_{k}}}.
    \end{align*}
\end{lemma}
\begin{proof}
    Recall that $\tau_{x_{k}}\paren{0_{x_{k}}} = \sum_{j\in\mathcal{E}}\abs{h_{j}\circ R_{x_{k}}\paren{0_{x_{k}}}}$ by definition. Choose $j\in\mathcal{E}$ arbitrarily. Similarly to Lemma~\ref{lemm:DDsigmacases}, we consider the following three cases.

    \begin{enumerate}[({\bf\casecom} 1)]
    \item If $h_{j}\paren{x_{k}} = h_{j}\circ R_{x_{k}}\paren{0_{x_{k}}} > 0$, then by taking sufficiently small $c>0$, we obtain $\allowbreak \abs{h_{j} \circ R_{x_{k}}\paren{\chi}} \allowbreak = h_{j} \circ R_{x_{k}}\paren{\chi}$ for all $\chi\in\tanspc[x_{k}]\mani$ such that $\norm{\chi}\leq c$. Hence, under $\allowbreak \tanspc[0_{x_{k}}]\paren{\tanspc[x_{k}]\mani} \allowbreak \simeq \tanspc[x_{k}]\mani$, the directional derivative of $\abs{h_{j}\circ R_{x_{k}}}:\tanspc[x_{k}]\mani\rightarrow\eucli[]$ at $0_{x_{k}}$ along $\Delta x_{k}^{\ast}\in\tanspc[x_{k}]\mani$ is represented as
    \begin{align*}
        \paren{\abs{h_{j}\circ R_{x_{k}}}}^{\prime}\paren{0_{x_{k}};\Delta x^{\ast}_{k}} 
        &=\paren{h_{j}\circ R_{x_{k}}}^{\prime}\paren{0_{x_{k}};\Delta x^{\ast}_{k}}\\
        &=\text{D}h_{j}\paren{R_{x_{k}}\paren{0_{x_{k}}}}\sbra{\text{D}R_{x}\paren{0_{x}}\sbra{\Delta x_{k}^{\ast}}}\\
        &= \text{D}h_{j}\paren{x_{k}}\sbra{\Delta x_{k}^{\ast}}\\
        &= -h_{j}\paren{x_{k}}\\
        &= -\abs{h_{j}\paren{x_{k}}},
    \end{align*}
    where the third equality holds by \eqref{def:retr} and fourth one by \eqref{subprobKKTeq}.

    \item If $h_{j}\paren{x_{k}} = h_{j}\circ R_{x_{k}}\paren{0_{x_{k}}} = 0$, then it follows from the definition of the one-sided derivative that 
    {\small
    \begin{align*}
        &\paren{\abs{h_{j}\circ R_{x_{k}}}}^{\prime}\paren{0_{x_{k}};\Delta x^{\ast}_{k}}\\
        &= \lim_{t \downarrow 0}\frac{\max\brc{h_{j}\circ R_{x_{k}}\paren{t\Delta x_{k}^{\ast}}, -h_{j}\circ R_{x_{k}}\paren{t\Delta x_{k}^{\ast}}}-0}{t}\\
        &= \max\brc{\lim_{t\downarrow 0}\frac{ h_{j}\circ R_{x_{k}}\paren{t\Delta x_{k}^{\ast}} - h_{j}\circ R_{x_{k}}\paren{0_{x_{k}}}}{t}, \lim_{t\downarrow 0}\frac{-\paren{h_{j}\circ R_{x_{k}}\paren{t\Delta x_{k}^{\ast}} - h_{j}\circ R_{x_{k}}\paren{0_{x_{k}}}}}{t}}\\
        &=\abs{ \paren{h_{j}\circ R_{x_{k}}}^{\prime}\paren{0_{x_{k}};\Delta x^{\ast}_{k}} }\\
        &=\abs{\text{D}h_{j}\paren{x_{k}}\sbra{\Delta x^{\ast}_{k}}}\\
        &=\abs{-h_{j}\paren{x_{k}}}\\
        &= 0.
    \end{align*}
    }
    
    \item Otherwise, if $h_{j}\paren{x_{k}} = h_{j}\circ R_{x_{k}}\paren{0_{x_{k}}} < 0$, then considering a sufficiently small neighborhood around $0_{x_{k}}$ in the same way as in the first case, we have 
    \begin{align*}
        \paren{\abs{h_{j}\circ R_{x_{k}}}}^{\prime}\paren{0_{x_{k}};\Delta x^{\ast}_{k}} = -\abs{h_{j}\paren{x_{k}}}.
    \end{align*}
    \end{enumerate}
\end{proof}

Combining the preceding lemmas, we obtain an upper bound on $\paren{P_{\rho}\circ R_{x_{k}}}^{\prime}\paren{0_{x_{k}};\Delta x_{k}^{\ast}}$.
\begin{proposition}\label{prop:meritineq}
    Let $\paren{\Delta x_{k}^{\ast},\mu_{k}^{\ast},\lambda_{k}^{\ast}}$ be a KKT triplet of \eqref{QP}, i.e., $\paren{\Delta x_{k}^{\ast},\mu_{k}^{\ast},\lambda_{k}^{\ast}}$ satisfies \eqref{subprobKKT} at $x_{k}$ with the symmetric positive-definite operator $B_{k}:\tanspc[x_{k}]\mani\rightarrow\tanspc[x_{k}]\mani$. If $\rho \geq \max\paren{\max_{i\in\mathcal{I}}\mu_{ki}^{\ast},\max_{j\in\mathcal{E}} \abs{\lambda_{kj}^{\ast}}}$, then
    \begin{align*}
        \paren{P_{\rho}\circ R_{x_{k}}}^{\prime}\paren{0_{x_{k}};\Delta x_{k}^{\ast}} \leq - \metr[]{B_{k}\sbra{\Delta x^{\ast}_{k}}}{\Delta x^{\ast}_{k}}.
    \end{align*}
\end{proposition}
\begin{proof}
    First, let us show 
    \begin{align}\label{diffofobjeq}
        \text{D}f\paren{x_{k}}\sbra{\Delta x_{k}^{\ast}} = - \metr[]{B_{k}\sbra{\Delta x^{\ast}_{k}}}{\Delta x^{\ast}_{k}} + \sum_{i\in\mathcal{I}}\mu_{ki}^{\ast}g_{i}\paren{x_{k}} + \sum_{j\in\mathcal{E}}\lambda_{kj}^{\ast}h_{j}\paren{x_{k}}.
    \end{align}
    From \eqref{subprobKKTLagsimple}, it holds that 
    \begin{align*}
        \text{D}f\paren{x_{k}}\sbra{\Delta x_{k}^{\ast}} =  -\metr[]{B_{k}\sbra{\Delta x_{k}^{\ast}}}{\Delta x_{k}^{\ast}} - \sum_{i\in\mathcal{I}}\mu^{\ast}_{ki}\text{D}g_{i}\paren{x_{k}}\sbra{\Delta x_{k}^{\ast}} - \sum_{j\in\mathcal{E}} \lambda_{kj}^{\ast}\text{D}h_{j}\paren{x_{k}}\sbra{\Delta x_{k}^{\ast}},
    \end{align*}
    which, together with the two equations 
    \begin{align*}
        \sum_{i\in\mathcal{I}}\mu_{ki}^{\ast}\text{D}g_{i}\paren{x_{k}}\sbra{\Delta x_{k}^{\ast}} \allowbreak = \allowbreak - \allowbreak \sum_{i\in\mathcal{I}}\mu_{ki}^{\ast}g_{i}\paren{x_{k}} \text{ and } \sum_{j\in\mathcal{E}} \lambda_{kj}^{\ast}\text{D}h_{j}\paren{x_{k}}\sbra{\Delta x_{k}^{\ast}} = \allowbreak -\sum_{j\in\mathcal{E}} \lambda_{kj}^{\ast} h_{j}\paren{x_{k}}
    \end{align*}
    from \eqref{subprobKKTcompl} and \eqref{subprobKKTeq}, implies \eqref{diffofobjeq}.

    Under $\tanspc[0_{x_{k}}]\paren{\tanspc[x_{k}]\mani}\simeq\tanspc[x_{k}]\mani$, we have
    \begin{align*}
        &\paren{P_{\rho}\circ R_{x_{k}}}^{\prime}\paren{0_{x_{k}};\Delta x_{k}^{\ast}}\\
        &= \text{D}f\paren{x_{k}}\sbra{\Delta x_{k}^{\ast}} + \rho \paren{\sum_{i\in\mathcal{I}} \sigma_{i x_{k}}^{\prime}\paren{0_{x_{k}};\Delta x_{k}^{\ast}} + \tau_{x_{k}}^{\prime}\paren{0_{x_{k}};\Delta x_{k}^{\ast}}}\\
        &= - \metr[]{B_{k}\sbra{\Delta x^{\ast}_{k}}}{\Delta x^{\ast}_{k}}\\
        &\qquad +\sum_{i\in\mathcal{I}}\mu_{ki}^{\ast}g_{i}\paren{x_{k}} + \sum_{j\in\mathcal{E}}\lambda_{kj}^{\ast}h_{j}\paren{x_{k}} + \rho \paren{\sum_{i\in\mathcal{I}}  \sigma_{i x_{k}}^{\prime}\paren{0_{x_{k}};\Delta x_{k}^{\ast}} - \sum_{j\in\mathcal{E}}\abs{h_{j}\paren{x_{k}}}}\\
        &\leq - \metr[]{B_{k}\sbra{\Delta x^{\ast}_{k}}}{\Delta x^{\ast}_{k}} + \sum_{j\in\mathcal{E}}\paren{\abs{\lambda_{kj}^{\ast}}-\rho}\abs{h_{j}\paren{x_{k}}} + \sum_{i\in\mathcal{I}}\mu_{ki}^{\ast}g_{i}\paren{x_{k}} + \rho\sum_{i\in\mathcal{I}}  \sigma_{i x_{k}}^{\prime}\paren{0_{x_{k}};\Delta x_{k}^{\ast}}\\
        &\leq  - \metr[]{B_{k}\sbra{\Delta x^{\ast}_{k}}}{\Delta x^{\ast}_{k}},
    \end{align*}
    where the second equality follows from \eqref{diffofobjeq} and Lemma~\ref{lemm:DDtau} and the second inequality follows from Lemma~\ref{lemm:DDsigmaIneq} and the assumption that $\abs{\lambda_{kj}^{\ast}} \leq \rho$ for all $i\in\mathcal{I}$.
\end{proof}
    
From the proposition and the positive-definiteness of $B_k$, if $\norm{\Delta x_{k}^{\ast}} \neq 0$, then we have 
\begin{align*}
    \paren{P_{\rho}\circ R_{x_{k}}}^{\prime}\paren{0_{x_{k}};\Delta x_{k}^{\ast}} \allowbreak \leq - \metr[]{B_{k}\sbra{\Delta x^{\ast}_{k}}}{\Delta x^{\ast}_{k}} < 0,
\end{align*}
which implies that $\Delta x_{k}^{\ast}$ is a descent direction for the merit function $P_{\rho}\circ R_{x_k}\paren{\cdot}$ for sufficiently large $\rho$.

\begin{remark}\label{armijowelldefined}
    Suppose that $\norm{\Delta x_{k}^{\ast}} \neq 0$. Then, we can always determine the step length $\alpha_{k}$ in $\RSQO$. Indeed, from Proposition~\ref{prop:meritineq}, we have
    \begin{align*}
        &P_{\rho}\paren{x_{k}}-P_{\rho}\circ R_{x_{k}}\paren{t\Delta x_{k}^{\ast}} - \gamma t\metr[]{B_{k}\sbra{\Delta x_{k}^{\ast}}}{\Delta x_{k}^{\ast}}\\
        &=P_{\rho}\circ R_{x_{k}}\paren{0_{x_{k}}}-P_{\rho}\circ R_{x_{k}}\paren{t\Delta x_{k}^{\ast}} - \gamma t\metr[]{B_{k}\sbra{\Delta x_{k}^{\ast}}}{\Delta x_{k}^{\ast}}\\
        &=-t \paren{P_{\rho}\circ R_{x_{k}}}^{\prime}\paren{0_{x_{k}};\Delta x_{k}^{\ast}} - \gamma t\metr[]{B_{k}\sbra{\Delta x_{k}^{\ast}}}{\Delta x_{k}^{\ast}} + o\paren{t}\\
        &\geq t\paren{1-\gamma}\metr[]{B_{k}\sbra{\Delta x_{k}^{\ast}}}{\Delta x_{k}^{\ast}} + o\paren{t}
    \end{align*}
    for $t>0$. Hence, because $1>\gamma>0$, the left-hand side is positive for any sufficiently small $t>0$. This ensures the existence of $r$ satisfying the backtracking line search 
    \eqref{Armijorule} at each iteration of $\RSQO$, and hence, we can always find such $r$ within finitely many trials.
\end{remark}

Before moving on to the global convergence theorem, we give one more proposition, which becomes a crucial ingredient for proving the theorem. To prove the proposition, we need to introduce more concepts and terminologies, such as Clarke regularity, and to prove some lemmas. We defer them and the proof of the proposition to Appendix~\ref{appendix:ExistenceofLimDeltaXBandProperties} for the sake of readability.
\begin{proposition}\label{LimBIneqLimP}
    Suppose that Assumptions A\ref{subprobfeasi}, A\ref{Bkbounded} and A\ref{genebound} hold. Define $\tilde{k}_{1}$ and $\bar{\rho}$ as in Proposition~\ref{prop:ensurebarrho}. Let $\mathcal{K}$ be any subsequence such that $\brc{\paren{x_{k}, \mu_{k}, \lambda_{k}, \alpha_{k}}}_{k\in\mathcal{K}}$ converges to $\paren{x^{\ast}, \mu^{\ast}, \lambda^{\ast}, 0}$.
    Then it holds that
    \begin{align}
        \limsup_{k\in\mathcal{K}, k\rightarrow\infty}\metr[x_{k}]{B_{k}\sbra{\Delta x^{\ast}_{k}}}{\Delta x^{\ast}_{k}} \leq \limsup_{k\in\mathcal{K}, k\to\infty}\frac{\beta}{\alpha_{k}} \paren{P_{\bar{\rho}}\circ R_{x_{k}}\paren{0_{x_k}}- P_{\bar{\rho}}\circ R_{x_{k}}\paren{\frac{\alpha_{k}}{\beta} \Delta x^{\ast}_{k}}}. \label{eq:key2}
    \end{align}
\end{proposition}
\begin{proof}
    See Appendix~\ref{appendix:ExistenceofLimDeltaXBandProperties} for the proof. 
\end{proof}
In Euclidean SQO, that is, in the case of $\mani=\mathbb{R}^d$, inequality~\eqref{eq:key2} can be verified by analyzing the limiting behavior of the directional derivative of $P_{\bar{\rho}} \circ R_{x_{k}}\paren{\cdot} = P_{\bar{\rho}}\paren{x_{k} + \cdot }$ in $\mathbb{R}^d$ and using Proposition~\ref{prop:meritineq}. In the manifold setting, however, the analysis of inequality~\eqref{eq:key2} is more complicated because the function $P_{\bar{\rho}}\circ R_{x_k}\paren{\cdot}$ is defined over the space $\tanspc[x_k]\mani$ that varies depending on $k$. In Appendix~\ref{appendix:ExistenceofLimDeltaXBandProperties}, we exploit Proposition~\ref{prop:meritineq} combined with parallel transport and the exponential mapping, which are important concepts on Riemannian geometry.

Now we are ready to prove the global convergence of $\RSQO$.
\begin{theorem}\label{globconv}
    Suppose that Assumptions~A\ref{subprobfeasi}, A\ref{Bkbounded}, and A\ref{genebound} hold. Let $\brc{\paren{x_{k},\mu_{k},\lambda_{k}}}$ be a sequence generated by $\RSQO$, and $\paren{x^{\ast},\mu^{\ast},\lambda^{\ast}}$ be an arbitrarily chosen accumulation point. Then, $\paren{x^{\ast},\mu^{\ast},\lambda^{\ast}}$ satisfies the KKT conditions \eqref{KKT} of RNLO \eqref{RNLO}.
\end{theorem}

\begin{proof}
    Without loss of generality, we assume $k\geq\tilde{k}_{1}$ and $\rho_{k} = \bar{\rho}$, where $\tilde{k}_{1}$ and $\bar{\rho}$ are defined in Proposition~\ref{prop:ensurebarrho}. It follows that $\brc{P_{\bar{\rho}}\paren{x_{k}}}$ is monotonically nonincreasing.
    Indeed, by the backtracking line search
    \eqref{Armijorule} and positive-definiteness of $B_{k}$, we have
    \begin{align}\label{Armijoonestepbefore}
        P_{\bar{\rho}} \paren{x_{k}} - P_{\bar{\rho}}\paren{x_{k+1}} \geq \gamma \alpha_{k} \metr[x_{k}]{B_{k}\sbra{\Delta x^{\ast}_{k}}}{\Delta x^{\ast}_{k}} \geq 0.
    \end{align}
    By Assumption A\ref{genebound}, we can take a closed bounded subset of $\mani$ including $\brc{x_{k}}$, which is actually a compact set from the Hopf-Rinow theorem. Therefore, $P_{\bar{\rho}}$ is bounded on the subset. By the monotone convergence theorem, $\brc{P_{\bar{\rho}}\paren{x_{k}}}$ converges as $k$ tends to infinity; hence, $\lim_{k\rightarrow \infty} P_{\bar{\rho}}\paren{x_{k}} - P_{\bar{\rho}}\paren{x_{k+1}} = 0$, which, together with \eqref{Armijoonestepbefore} and $\gamma > 0$, implies
    \begin{align}\label{alphaBlimit}
        \lim_{k\rightarrow \infty} \alpha_{k} \metr[x_{k}]{B_{k}\sbra{\Delta x^{\ast}_{k}}}{\Delta x^{\ast}_{k}} = 0.
    \end{align}
    
    Next, without loss of generality, by taking a subsequence $\mathcal{K}$ if necessary, we may assume that $\brc{\paren{x_{k}, \mu_{k}, \lambda_{k}}}_{k\in\mathcal{K}}$ is a sequence converging to $\paren{x^{\ast}, \mu^{\ast}, \lambda^{\ast}}$ and furthermore $\brc{\alpha_{k}}_{k\in\mathcal{K}}$ has a limit. We will prove $\lim_{k\in\mathcal{K}, k\rightarrow \infty} \norm{\Delta x_{k}^{\ast}}_{x_{k}} = 0$ by considering the following two cases for $\lim_{k\in\mathcal{K}, k\rightarrow\infty} \alpha_{k}$.

    ({\bf\casecom} 1) If $\lim_{k\in\mathcal{K}, k\rightarrow\infty} \alpha_{k} > 0$, then by (\ref{alphaBlimit}) we have $\lim_{k\in\mathcal{K}, k\rightarrow \infty} \metr[x_{k}]{B_{k}\sbra{\Delta x^{\ast}_{k}}}{\Delta x^{\ast}_{k}} \allowbreak = 0$, which, together with Assumption A\ref{Bkbounded} that $\metr[x_{k}]{B_{k}\sbra{\Delta x^{\ast}_{k}}}{\Delta x^{\ast}_{k}} \geq m \norm{\Delta x^{\ast}_{k}}^{2}_{x_{k}}$ for every $k\in\mathcal{K}$, implies $\lim_{k\in\mathcal{K}, k\rightarrow \infty} \norm{\Delta x_{k}^{\ast}}_{x_{k}} = 0$.

    ({\bf\casecom} 2) Otherwise, if $\lim_{k\in\mathcal{K}, k\rightarrow\infty} \alpha_{k} = 0$, then by the backtracking line search
    \eqref{Armijorule} in $\RSQO$, for every $k\in\mathcal{K}$,
    \begin{align}\label{al:key}
        \frac{\beta}{\alpha_{k}} \paren{P_{\bar{\rho}}\paren{x_{k}} - P_{\bar{\rho}}\circ R_{x_{k}}\paren{\frac{\alpha_{k}}{\beta} \Delta x^{\ast}_{k}}} < \gamma \metr[x_{k}]{B_{k}\sbra{\Delta x_{k}^{\ast}}}{\Delta x_{k}^{\ast}}.
    \end{align}
    To derive a contradiction, suppose that $\brc{\norm{\Delta x_k^{\ast}}_{x_k}}_{k\in\mathcal{K}}$ does not converge to $0$ as $k\in\mathcal{K}\to\infty$. Since $\brc{\norm{\Delta x^{\ast}_{k}}_{x_{k}}}_{k\in\mathcal{K}}$ is bounded from Proposition~\ref{DeltaxkBounded}, Assumption~A\ref{Bkbounded} ensures that there exists some $p>0$ such that
    \begin{align}
        \limsup_{k\in\mathcal{K}, k\to\infty}\metr[x_{k}]{B_{k}\sbra{\Delta x_{k}^{\ast}}}{\Delta x_{k}^{\ast}}=p.\label{eq:key3}
    \end{align}
    Now, let us take the limit superior on both sides in \eqref{al:key}. By combining it with \eqref{eq:key3}, \eqref{eq:key2} in Proposition~\ref{LimBIneqLimP}, and the fact that $P_{\bar{\rho}}\paren{x_k}=P_{\bar{\rho}}\circ R_{x_k}\paren{0_{x_k}}$, we have $p\le \gamma p$ implying $1\le \gamma$. However, this contradicts $\gamma<1$. Consequently, $\lim_{k\in\mathcal{K}, k\rightarrow\infty}\norm{\Delta x_k^{\ast}}_{x_k} = 0$ holds.

    Next, from the KKT conditions \eqref{subprobKKT} of the subproblem for each $k\in\mathcal{K}$ and the Cauchy-Schwarz inequality, we have $\mu^{\ast}_{k}\geq 0$ and 
    \begin{align}
    \begin{split}\label{normedsubprobKKT}
    &\norm{\text{\upshape grad}\,f\paren{x_{k}} + \sum_{i\in\mathcal{I}} \mu_{ki}^{\ast} \text{\upshape grad}\,g_{i}\paren{x_{k}} + \sum_{j\in\mathcal{E}}\lambda_{kj}^{\ast} \text{\upshape grad}\,h_{j}\paren{x_{k}}} = \norm{-B_{k}\sbra{\Delta x^{\ast}_{k}}} \leq \norm{B_{k}}_{\rm op} \norm{\Delta x^{\ast}_{k}},\\
        &g_{i}\paren{x_{k}} \leq - \metr[]{\text{\upshape grad}\,g_{i}\paren{x_{k}}}{\Delta x^{\ast}_{k}} \leq \norm{\text{\upshape grad}\,g_{i}\paren{x_{k}}} \norm{\Delta x^{\ast}_{k}} \text{ and} \\
        &\abs{\mu^{\ast}_{ki}g_{i}\paren{x_{k}}} = \abs{-\mu^{\ast}_{ki} \metr[]{\text{\upshape grad}\,g_{i}\paren{x_{k}}}{\Delta x^{\ast}_{k}}} \leq \mu^{\ast}_{ki} \norm{\text{\upshape grad}\,g_{i}\paren{x_{k}}} \norm{\Delta x^{\ast}_{k}}, \text{ for all } i\in\mathcal{I},\\
        &\abs{h_{j}\paren{x_{k}}} = \abs{-\metr[]{\text{\upshape grad}\,h_{j}\paren{x_{k}}}{\Delta x^{\ast}_{k}}} \leq \norm{\text{\upshape grad}\,h_{j}\paren{x_{k}}} \norm{\Delta x^{\ast}_{k}}, \text{ for all } j\in\mathcal{E},
    \end{split}
    \end{align}
    where $\norm{\cdot}_{\rm op}$ denotes the operator norm defined in Lemma~\ref{AkOperNormBounded}. Under Assumptions~A\ref{Bkbounded} and A\ref{genebound}, $\brc{\norm{B_{k}}_{\rm op}}_{k\in\mathcal{K}}$, $\brc{\mu^{\ast}_{k}}_{k\in\mathcal{K}}$, $\brc{\norm{\text{grad}g_{i}\paren{x_{k}}}}_{k\in\mathcal{K}}$, and $ \brc{\norm{\text{grad}h_{j}\paren{x_{k}}}}_{k\in\mathcal{K}}$ are all bounded from Lemma~\ref{AkOperNormBounded}, the boundedness of $\brc{x_{k}}_{k\in\mathcal{K}}$, and the continuity of $\norm{\cdot}$, $\brc{\text{grad}g_{i}}_{i\in\mathcal{I}}$ and $\brc{\text{grad}h_{j}}_{j\in\mathcal{E}}$. Thus, by driving $k\in\mathcal{K}\rightarrow \infty$ and noting $\lim_{k\in\mathcal{K},k\rightarrow\infty}\norm{\Delta x^{\ast}_{k}}_{x_{k}}=0$, $\mu^{\ast}\geq 0$ holds and all of the leftmost sides of \eqref{normedsubprobKKT} tend to $0$s, which imply that $\paren{x^{\ast},\mu^{\ast},\lambda^{\ast}}$ satisfies the KKT conditions \eqref{KKT} of RNLO \eqref{RNLO}.
\end{proof}

\subsection{Local convergence}
In this subsection, we study the local convergence of RSQO. In Section~\ref{RiemGeoonProd}, we introduce additional concepts from Riemannian geometry that will be needed for the local convergence analysis. In Section~\ref{localconvsubsubsec}, we prove the local quadratic convergence of RSQO.

\subsubsection{Notation and terminology for \texorpdfstring{$\mani\times\eucli[d^{\prime}]$}{the product manifold} }\label{RiemGeoonProd}
We will extend the concepts explained in Section~\ref{preliminaries} to the product manifold $\mani\times\eucli[d^{\prime}]$ with $d^{\prime}\coloneqq m+n$. Recall that $\mani$ is a $d$-dimensional Riemannian manifold endowed with a Riemannian metric $\metr[]{\cdot}{\cdot}$. We define the product manifold $\mani\times\eucli[d^{\prime}]$ by regarding the Euclidean space $\eucli[d^{\prime}]$ as a Riemannian manifold with the canonical inner product as the Riemannian metric.
Let $x\in\mani$ and $\eta\in\eucli[d^{\prime}]$. Here, the decomposition $\tanspc[\paren{x,\eta}]\paren{\mani\times\eucli[d^{\prime}]} = \tanspc[x]\mani\oplus \tanspc[\eta]\eucli[d^{\prime}]$, where $\oplus$ is the direct sum, ensures that all tangent vectors on $\mani\times\eucli[d^{\prime}]$ can be decomposed as $\xi_{x} \oplus \zeta_{\eta}$ with $\xi_{x}\in\tanspc[x]\mani$ and $\zeta_{\eta}\in\tanspc[\eta]\eucli[d^{\prime}]$. The product manifold $\mani \times \eucli[d^{\prime}]$ has the natural Riemannian metric $\metr[\paren{x,\eta}]{\cdot}{\cdot}: \tanspc[x]\mani\oplus\eucli[d^{\prime}] \times \tanspc[x]\mani\oplus \eucli[d^{\prime}] \rightarrow \eucli$, called the product metric, under the canonical identification $\tanspc[\eta]\eucli[d^{\prime}]\simeq\eucli[d^{\prime}]$, defined by $\metr[\paren{x,\eta}]{\xi_{x}^{1} \oplus \zeta_{\eta}^{1}}{\xi_{x}^{2} \oplus \zeta^{2}_{\eta}} \coloneqq \metr[x]{\xi_{x}^{1}}{\xi_{x}^{2}} + {\zeta^{1}}^{\top}_{\eta}\zeta^{2}_{\eta}$ for all $\xi_{x}^{1},\xi_{x}^{2} \in \tanspc[x]\mani$ and $\zeta_{\eta}^{1}, \zeta_{\eta}^{2} \in \eucli[d^{\prime}]$. Let $R$ be a retraction on $\mani$ and recall definition \eqref{def:retr} of retractions. It follows that the mapping 
\begin{align}\label{productRetr}
    \begin{split}
         \widetilde{R}: \tanspc[]\paren{\mani\times\eucli[d^{\prime}]} &\longrightarrow \mani\times\eucli[d^{\prime}]\\
            \xi_{x} \oplus \zeta_{\eta} &\longmapsto \paren{R_{x}\paren{\xi_{x}},\eta + \zeta_{\eta}}
       \end{split}
\end{align}
is a retraction on $\mani\times\eucli[d^{\prime}]$. Denoting by $\Gamma_{ij}^{\ell}$ the Christoffel symbols associated with the Levi-Civita connection on $\mani$ for each $i,j,\ell = 1,\ldots, d$, we can extend these concepts to $\mani\times\eucli[d^{\prime}]$ as follows: for each $i,j,\ell = 1,\ldots, d+d^{\prime}$,
\begin{align}\label{productChristoffel}
    \widetilde{\Gamma}_{ij}^{\ell}\paren{x,\eta} =
    \begin{cases}
        \Gamma_{ij}^{\ell}\paren{x}, & \text{if } i,j,\ell \in \brc{1\ldots d},\\
        0, & \text{otherwise}.
    \end{cases}
\end{align}

Lastly, let us define quadratic convergence on $\mani\times \eucli[d^{\prime}]$ by tailoring \cite[Definition.4.5.2]{Absiletal08OptBook}.
\begin{definition}\label{def:quad}
    Let $\brc{\paren{x_{k},\eta_{k}}}$ be a sequence on $\mani\times \eucli[d^{\prime}]$ that converges to $\paren{x^{\ast},\eta^{\ast}}$. Let $\paren{\mathcal{U},\varphi}$ be a chart of $\mani$ containing $x^{\ast}$. If there exists a constant $c\geq 0$ such that we have
    \begin{align*}
        \norm{\begin{pmatrix}\varphi\paren{x_{k+1}} - \varphi\paren{x^{\ast}}\\
        \eta_{k+1}-\eta^{\ast}
        \end{pmatrix}
        } \leq c \norm{
        \begin{pmatrix}\varphi\paren{x_{k}} - \varphi\paren{x^{\ast}}\\
        \eta_{k}-\eta^{\ast}
        \end{pmatrix}}^2
    \end{align*}
    for all $k$ sufficiently large, then $\brc{\paren{x_{k},\eta_{k}}}$ is said to converge to $\paren{x^{\ast},\eta^{\ast}}$ quadratically.
\end{definition}

\subsubsection{Local convergence analysis}\label{localconvsubsubsec}
We will write $\eta \coloneqq \paren{\mu,\lambda}\in\eucli[m+n]$ for brevity. Let $\brc{\paren{x_{k},\eta_{k}}}$ be a sequence produced by RSQO and $\paren{x^{\ast}, \eta^{\ast}}$ be an accumulation point of $\brc{\paren{x_{k},\eta_{k}}}$. Under Assumptions~A\ref{subprobfeasi}, A\ref{Bkbounded}, and A\ref{genebound}, $\paren{x^{\ast},\eta^{\ast}}$ is a KKT pair of RNLO \eqref{RNLO}, as we proved in Theorem~\ref{globconv} in Section~\ref{subsec:globalconv}. In what follows, we will prove that, under the following four assumptions, the whole sequence $\brc{\paren{x_{k},\eta_{k}}}$ actually converges to $\paren{x^{\ast},\eta^{\ast}}$ quadratically in the sense of Definition \ref{def:quad}. The first two assumptions are
\begin{enumerate}[\textbf{B}1]
    \item $\paren{x^{\ast},\eta^{\ast}}$ satisfies the LICQ, SOSCs, and SC, \label{locmincha}
    \item $f,\brc{g_{i}}_{i\in\mathcal{I}_{a}\paren{x^{\ast}}},\brc{h_{j}}_{j\in\mathcal{E}}$ are of class $C^{3}$.\label{classC3}
\end{enumerate}
See Section~\ref{Pre:RiemOpt} for the definitions of the LICQ, SOSCs, and SC. In this subsection, $\mathcal{L}_{\eta}\paren{x}$ stands for $\mathcal{L}_{\mu,\lambda}\paren{x}$; see Section~\ref{Pre:RiemOpt}
for the definition of $\mathcal{L}_{\mu,\lambda}$. The remaining two assumptions are related to RSQO iterations: for $k\geq K_{0}$ with $K_{0}$ sufficiently large, 
\begin{enumerate}[\textbf{B}1]
    \setcounter{enumi}{2}
    \item $B_{k}=\text{Hess}\,\mathcal{L}_{\eta_{k}}\paren{x_{k}}$,\label{sethess}
    \item a step length of unity is accepted, i.e., $\alpha_{k} = 1$.\label{newtonstep}
\end{enumerate}
It would be more practical to consider the case where, in Assumption~B\ref{sethess}, $B_k$ is updated with a quasi-Newton formula such as the BFGS formula. Meanwhile, the theoretical verification of Assumption~B\ref{newtonstep} is difficult because of the Maratos effect, e.g. see \cite{NocedalWright06NumOpt}. We will touch these issues again in Section~\ref{conclusion}.
Now, we establish the local quadratic convergence of RSQO to $\paren{x^{\ast},\eta^{\ast}}$. In our analysis, we show that the sequence generated by RSQO corresponds with that by the Riemannian Newton method, whose local quadratic convergence property is established in \cite[Chapter 6]{Absiletal08OptBook}.
\begin{theorem}\label{quadlocalconvthm}
    Under Assumptions A and B, $\brc{\paren{x_{k}, \eta_{k}}}$ converges to $\paren{x^{\ast},\eta^{\ast}}$ quadratically. 
\end{theorem}
\begin{proof}  
    We will prove the theorem by showing that $\brc{\paren{x_{k},\eta_{k}}}_{k\ge K}$ with $K$ sufficiently large is actually identical to a certain quadratically convergent sequence generated by the Riemannian Newton method~\cite{Absiletal08OptBook} on $\mani\times\eucli[m+n]$; see Appendix~\ref{SubsecAppen:RiemNewton} for an overview of the Riemannian Newton method. Henceforth, for the sake of a simple explanation, we will assume that $\mathcal{I}_a\paren{x^{\ast}}=\mathcal{I}$; see Section~\ref{Pre:RiemOpt} for the definition of $\mathcal{I}_{a}$. The subsequent argument can be extended to the case of $\mathcal{I}_{a}\paren{x^{\ast}} \subset \mathcal{I}$. \footnote{For an arbitrary subsequence $\brc{\paren{x_{k}, \eta_{k}}}_{k\in\mathcal{K}}$ by RSQO converging to $\paren{x^{\ast}, \eta^{\ast}}$ and $K\in\mathcal{K}$ sufficiently large, we see that $\mu_{K\tilde{i}}^{\ast} = 0$ for all $\tilde{i}\in\mathcal{I}\backslash\mathcal{I}_{a}\paren{x^{\ast}}$. Bisect $\mu_{k}$ into $\paren{\mu_{ki},\mu_{k\tilde{i}}}$, where $i\in\mathcal{I}_{a}\paren{x^{\ast}}$ and $\tilde{i}\in\mathcal{I}\backslash\mathcal{I}_{a}\paren{x^{\ast}}$. We can similarly prove that $\brc{\paren{x_{k}, \mu^{\ast}_{ki}, \lambda^{\ast}_{k}, \mu^{\ast}_{k\tilde{i}}}}_{k\geq K}$ are identical to $\brc{\paren{x^{\prime}_{k}, \mu^{\prime}_{k}, \lambda^{\prime}_{k}, 0}}_{k=0,1,\ldots}$, where $\brc{\paren{x^{\prime}_{k}, \mu^{\prime}_{k}, \lambda^{\prime}_{k}}}$ is generated by the Riemannian Newton method with $\paren{x^{\prime}_{0}, \mu^{\prime}_{0}, \lambda^{\prime}_{0}} = \paren{x_{K}, \mu^{\ast}_{K}, \lambda^{\ast}_{K}}$.}

    Let $\eta^{\prime}\coloneqq\paren{\mu^{\prime}, \lambda^{\prime}} \in \eucli[m+n]$. Notice that $\mathcal{L}\paren{x^{\prime}, \eta^{\prime}}$ is a real-valued function defined on $\mani\times\eucli[m+n]$, whereas $\mathcal{L}_{\eta^{\prime}}\paren{x^{\prime}}$ is one defined on $\mani$ for the given $\eta^{\prime}$. Consider the Riemannian Newton method for solving $\text{grad}\mathcal{L}\paren{x^{\prime}, \eta^{\prime}} = 0$ on $\mani \times \eucli[m+n]$. Its $\paren{k+1}$-th iteration is defined by 
    \begin{subequations}\label{RiemNewtonProdMani}
        \begin{align}
            &\text{Hess}\,\mathcal{L}\paren{x^{\prime}_{k}, \eta^{\prime}_{k}}\sbra{\Delta x^{\prime}_{k} \oplus \Delta \eta^{\prime}_{k}} = -\text{grad}\,\mathcal{L}\paren{x^{\prime}_{k}, \eta^{\prime}_{k}},\label{Newtoneq:ProMani}\\
            &\paren{x^{\prime}_{k+1}, \eta^{\prime}_{k+1}} = \widetilde{R}_{\paren{x^{\prime}_{k}, \eta^{\prime}_{k}}} \paren{\Delta x^{\prime\ast}_{k} \oplus \Delta \eta_{k}^{\prime\ast}},\label{RiemNewtonRetrProdMani}
        \end{align}
    \end{subequations}
    where Hess and grad are operators defined over $\mathcal{M}\times \eucli[m+n]$ and $\Delta x^{\prime\ast}_{k} \oplus \Delta\eta_{k}^{\prime\ast}$ is the solution of \eqref{Newtoneq:ProMani}. Moreover, as defined by \eqref{productRetr}, $\widetilde{R}$ is the retraction over $\mathcal{M}\times \eucli[m+n]$. 

    First, we show that $\text{grad}\mathcal{L}\paren{x^{\ast},\eta^{\ast}}=0$ and that $\text{Hess}\mathcal{L}\paren{x^{\ast},\eta^{\ast}}$ is nonsingular. These properties are required as the assumptions of \cite[Theorem 6.3.2]{Absiletal08OptBook} concerning the local quadratic convergence of the Riemannian Newton method in solving $\text{grad}\mathcal{L}\paren{x,\eta}=0$. The equation $\text{grad}\,\mathcal{L}\paren{x^{\ast},\eta^{\ast}}=0$ is readily confirmed from the KKT condition \eqref{KKTLag} and $\mathcal{I}_{a}\paren{x^{\ast}}=\mathcal{I}$. To show the nonsingularity of $\text{Hess}\mathcal{L}\paren{x^{\ast},\eta^{\ast}}$, 
    using \eqref{Hesscoordmetr} and \eqref{productChristoffel}, we first represent the coordinate expression of $ \metr[\paren{x^{\ast},\eta^{\ast}}]{\text{Hess}\,\mathcal{L}\paren{x^{\ast},\eta^{\ast}}\sbra{\cdot}}{\cdot}$ as
    \begin{align}\label{NewtonHessmetrcoord}
        \begin{bmatrix}
            \text{D}^{2}\widehat{\mathcal{L}}_{\eta^{\ast}}\paren{\widehat{x^{\ast}}} - \widehat{\Gamma}_{\widehat{x^{\ast}}}\sbra{\text{D}\widehat{\mathcal{L}}_{\eta^{\ast}}\paren{\widehat{x^{\ast}}}}& \text{D}\widehat{\textbf{g}}\paren{\widehat{x^{\ast}}}
            &\text{D}\widehat{\textbf{h}}\paren{\widehat{x^{\ast}}}\\
            \text{D}\widehat{\textbf{g}}\paren{\widehat{x^{\ast}}}^{\top} & 0 & 0\\
            \text{D}\widehat{\textbf{h}}\paren{\widehat{x^{\ast}}}^{\top} & 0 & 0
        \end{bmatrix},
    \end{align}
    where 
    \begin{align*}
        \widehat{\textbf{g}}\paren{\widehat{x}}&\coloneqq\paren{\widehat{g}_{1}\paren{\widehat{x}}, \widehat{g}_{2}\paren{\widehat{x}}, \cdots, \widehat{g}_{m}\paren{\widehat{x}}}^{\top}\in\eucli[m],\\ \widehat{\textbf{h}}\paren{\widehat{x}}&\coloneqq\paren{\widehat{h}_{1}\paren{\widehat{x}}, \widehat{h}_{2}\paren{\widehat{x}}, \cdots, \widehat{h}_{n}\paren{\widehat{x}}}^{\top}\in\eucli[n],\\
        \text{D}\widehat{\textbf{g}}\paren{\widehat{x}} &\coloneqq\paren{\text{D}\widehat{g}_{1}\paren{\widehat{x}}, \text{D}\widehat{g}_{2}\paren{\widehat{x}}, \cdots, \text{D}\widehat{g}_{m}\paren{\widehat{x}}}\in\eucli[d\times m],\\
        \text{D}\widehat{\textbf{h}}\paren{\widehat{x}} &\coloneqq\paren{\text{D}\widehat{h}_{1}\paren{\widehat{x}}, \text{D}\widehat{h}_{2}\paren{\widehat{x}}, \cdots, \text{D}\widehat{h}_{n}\paren{\widehat{x}}}\in\eucli[d\times n].
    \end{align*}
    From the LICQ, it follows that $\mathcal{N}\coloneqq\sbra{\text{D}\widehat{\textbf{g}}\paren{\widehat{x^{\ast}}},\text{D}\widehat{\textbf{h}}\paren{\widehat{x^{\ast}}}}\in\eucli[d\times \paren{m+n}]$ is a full-column rank matrix. Furthermore, it follows from the SOSCs that $\paren{\text{D}^{2}\widehat{\mathcal{L}}_{\eta^{\ast}}\paren{\widehat{x^{\ast}}} - \widehat{\Gamma}_{\widehat{x^{\ast}}}\sbra{D\widehat{\mathcal{L}}_{\eta^{\ast}}\paren{\widehat{x^{\ast}}}}}\in\eucli[d\times d]$ is positive-definite on the null space of $\mathcal{N}$. Thus, by \cite[Lemma 16.1]{NocedalWright06NumOpt}, matrix \eqref{NewtonHessmetrcoord} is nonsingular, which implies that $\text{Hess}\,\mathcal{L}\paren{x^{\ast},\eta^{\ast}}$ is also nonsingular.

    Since the assumptions of \cite[Theorem 6.3.2]{Absiletal08OptBook} have been fulfilled as shown above, the theorem is applicable, and thus, there exists some neighborhood $\mathcal{N}(x^{\ast},\eta^{\ast}) \subseteq \mani\times \eucli[d]$ of $\paren{x^{\ast},\eta^{\ast}}$ such that 
    \begin{enumerate}[(\cond 1)]
        \item $\brc{\paren{x^{\prime}_{k},\eta^{\prime}_{k}}}$ quadratically converges to $\paren{x^{\ast},\eta^{\ast}}$,\label{nbrquad}
        \item $\brc{\paren{x^{\prime}_{k},\eta^{\prime}_{k}}}\subseteq \mathcal{N}(x^{\ast},\eta^{\ast})$,\label{nbrinclseq}
    \end{enumerate}
    where $\brc{\paren{x^{\prime}_{k},\eta^{\prime}_{k}}}$ is any sequence generated by \eqref{RiemNewtonProdMani} starting from a point in $\mathcal{N}(x^{\ast},\eta^{\ast})$. Furthermore, notice that, as both the LICQ at $x^{\ast}$ and $\mu^{\ast}>0$ hold because of Assumption B\ref{locmincha} and $\mathcal{I}_a(x^{\ast})=\mathcal{I}$, we may assume that $\mu>0$ and the linear independence of $\brc{\text{grad}g_{i}\paren{x}}_{i\in\mathcal{I}}$ and $\brc{\text{grad}h_{j}\paren{x}}_{j\in\mathcal{E}}$ hold for any $(x,\eta)=(x,\mu,\lambda)\in \mathcal{N}(x^{\ast},\eta^{\ast})$, if necessary, by replacing $\mathcal{N}(x^{\ast},\eta^{\ast})$ in the above with a sufficiently smaller neighborhood of $(x^{\ast},\eta^{\ast})$. Therefore, by (\cond\ref{nbrinclseq}), we have
    \begin{enumerate}[(\cond 1)]
        \setcounter{enumi}{2}
        \item $\brc{\text{\upshape grad}\,g_{i}\paren{x_{k}}, \text{\upshape  grad}\,h_{j}\paren{x_{k}}}_{i\in\mathcal{I},j\in\mathcal{E}}$ are linearly independent for all $k\geq 0$, \label{nbrLICQ}
        \item $\mu_k^{\prime}>0$, \quad $\forall k\ge 0$.\label{nbrsc}
    \end{enumerate}

    As $(x^{\ast},\eta^{\ast})$ is an accumulation point of $\brc{\paren{x_{k},\eta_{k}}}$ generated by RSQO, there exists some $K\ge K_0$ such that $\paren{x_{K},\eta_{K}}\in \mathcal{N}(x^{\ast},\eta^{\ast})$ and the properties in Assumptions~B\ref{sethess} and B\ref{newtonstep} are valid for all $k\ge K$. Set $\paren{x_{0}^{\prime},\eta_{0}^{\prime}}:=\paren{x_{K},\eta_{K}}$ and perform \eqref{RiemNewtonProdMani} successively to produce a sequence $\brc{\paren{x_{k}^{\prime},\eta_{k}^{\prime}}}$. In order to verify the assertion of this theorem, by virtue of (\cond\ref{nbrquad}), it suffices to prove that the two sequences $\brc{\paren{x_{k},\eta_{k}}}_{k\ge K}$ and $\brc{\paren{x_{k}^{\prime},\eta_{k}^{\prime}}}_{k\ge 0}$ are identical, namely, $\paren{x_{K+k},\eta_{K+k}}=\paren{x_{k}^{\prime},\eta_{k}^{\prime}}$ for all $ k\ge 0$.
    
    We can prove this equation by induction. The case $k=0$ is obvious from the definition. Next, consider the case of $k=1$. Note that \eqref{Newtoneq:ProMani} is equivalent to the equation $\metr[\paren{x^{\prime}_{k},\eta^{\prime}_{k}}]{\text{Hess}\,\mathcal{L}\paren{x^{\prime}_{k},\eta^{\prime}_{k}}\sbra{\Delta x^{\prime}_{k} \oplus \Delta \eta^{\prime}_{k}}}{\plchold} \allowbreak = \metr[\paren{x^{\prime}_{k},\eta^{\prime}_{k}}]{-\text{grad}\,\mathcal{L}\paren{x^{\prime}_{k},\eta^{\prime}_{k}}}{\plchold}$ and the coordinate expression of the equation at $\paren{x^{\prime}_{0}, \eta^{\prime}_{0}}$ is
    \begin{align}\label{CoordNewtonEqonProdMani}
        \begin{bmatrix}
            \text{D}^{2}\widehat{\mathcal{L}}_{\eta_{0}}\paren{\widehat{x^{\prime}_{0}}} - \widehat{\Gamma}_{\widehat{x^{\prime}_{0}}}\sbra{\text{D}\widehat{\mathcal{L}}_{\eta_{0}}\paren{\widehat{x^{\prime}_{0}}}} & \text{D}\widehat{\textbf{g}}\paren{\widehat{x^{\prime}_{0}}} & \text{D}\widehat{\textbf{h}}\paren{\widehat{x^{\prime}_{0}}} \\
            \text{D}\widehat{\textbf{g}}\paren{\widehat{x^{\prime}_{0}}}^{\top} & 0 & 0\\
            \text{D}\widehat{\textbf{h}}\paren{\widehat{x^{\prime}_{0}}}^{\top} & 0 & 0\\
        \end{bmatrix}
        \begin{bmatrix}
            \widehat{\Delta x^{\prime\ast}_{0}}\\
            \Delta \mu_{0}^{\prime\ast}\\
            \Delta \lambda_{0}^{\prime\ast}\\
        \end{bmatrix}
        = - 
        \begin{bmatrix}
            \text{D}\widehat{\mathcal{L}}_{\eta_{0}}\paren{\widehat{x^{\prime}_{0}}}\\
            \widehat{\textbf{g}}\paren{\widehat{x^{\prime}_{0}}}\\
            \widehat{\textbf{h}}\paren{\widehat{x^{\prime}_{0}}}\\
        \end{bmatrix},
    \end{align}
    where $\Delta \eta^{\prime\ast}_{0} \eqqcolon \paren{\Delta \mu^{\prime\ast}_{0}, \Delta \lambda^{\prime\ast}_{0}} \in \eucli[m]\times\eucli[n]$. By substituting 
    \begin{gather*}
        \text{D}\widehat{\mathcal{L}}_{\eta_{0}}\paren{\widehat{x^{\prime}_{0}}} = \text{D}\widehat{f}\paren{\widehat{x^{\prime}_{0}}} + \text{D}\widehat{\textbf{g}}\paren{\widehat{x^{\prime}_{0}}}\mu^{\prime}_{0} + \text{D}\widehat{\textbf{h}}\paren{\widehat{x^{\prime}_{0}}}\lambda^{\prime}_{0},\\
        \mu^{\prime}_{1} = \mu^{\prime}_{0} + \Delta \mu^{\prime\ast}_{0},\\
        \lambda^{\prime}_{1} = \lambda^{\prime}_{0} + \Delta \lambda_{0}^{\prime\ast}
    \end{gather*}
    into \eqref{CoordNewtonEqonProdMani}, we have 
    \begin{gather*}
        \paren{\text{D}^{2}\widehat{\mathcal{L}}_{\eta_{0}}\paren{\widehat{x^{\prime}_{0}}} - \widehat{\Gamma}_{\widehat{x^{\prime}_{0}}}\sbra{\text{D}\widehat{\mathcal{L}}_{\eta_{0}}\paren{\widehat{x^{\prime}_{0}}}}}\widehat{\Delta x^{\prime\ast}_{0}} + \text{D}\widehat{f}\paren{\widehat{x^{\prime}_{0}}} + \text{D}\widehat{\textbf{g}}\paren{\widehat{x^{\prime}_{0}}}\mu^{\prime}_{1} +
        \text{D}\widehat{\textbf{h}}\paren{\widehat{x^{\prime}_{0}}}\lambda^{\prime}_{1} = 0,\\
        \textbf{g}\paren{\widehat{x^{\prime}_{0}}} + \text{D}\widehat{\textbf{g}}\paren{\widehat{x^{\prime}_{0}}}\widehat{\Delta x^{\prime\ast}_{0}} = 0,\\
        \textbf{h}\paren{\widehat{x^{\prime}_{0}}} + \text{D}\widehat{\textbf{h}}\paren{\widehat{x^{\prime}_{0}}}\widehat{\Delta x^{\prime\ast}_{0}}= 0.
    \end{gather*}
    Note that $\mu^{\prime}_{1}>0$ holds by (\cond\ref{nbrsc}). Hence, by using $\eta^{\prime}_{1} = \paren{\mu^{\prime}_{1}, \lambda^{\prime}_{1}}$ in \eqref{RiemNewtonRetrProdMani} with $k=0$ as a pair of Lagrange multiplier vectors, $\paren{\widehat{\Delta x^{\prime\ast}_{0}}, \eta^{\prime}_{1}}$ satisfies the KKT conditions of the following problem:
    \begin{align}
        \begin{split}\label{coordNewtonsubprob}
            \underset{\widehat{\Delta x_{0}} \in \eucli[d]}{\text{minimize}} \quad & \frac{1}{2} \widehat{\Delta x_{0}}^{\top} \paren{\text{D}^{2}\widehat{\mathcal{L}}_{\eta_{0}}\paren{\widehat{x^{\prime}_{0}}} - \widehat{\Gamma}_{\widehat{x^{\prime}_{0}}}\sbra{\text{D}\widehat{\mathcal{L}}_{\eta_{0}}\paren{\widehat{x^{\prime}_{0}}}}}\widehat{\Delta x_{0}}  + {\text{D}\widehat{f}\paren{\widehat{x^{\prime}_{0}}}}^{\top}\widehat{\Delta x_{0}} \\
            \text{subject to}\quad  & \widehat{g_{i}}\paren{\widehat{x^{\prime}_{0}}} + {\text{D}\widehat{g_{i}}\paren{\widehat{x^{\prime}_{0}}}}^{\top}\widehat{\Delta x_{0}} \leq 0, \text{ for all }i \in \mathcal{I},\\ 
            & \widehat{h_{j}}\paren{\widehat{x^{\prime}_{0}}} + {\text{D}\widehat{h_{j}}\paren{\widehat{x^{\prime}_{0}}}}^{\top}\widehat{\Delta x_{0}} = 0, \text{ for all } j \in \mathcal{E},
        \end{split}
    \end{align}
    which is nothing but the coordinate expression of subproblem \eqref{QP} that is solved in RSQO at $x_0^{\prime}=x_K$. Note also that it follows from (\cond\ref{nbrLICQ}) and \cite[Theorem 4.1 (d)]{BergmannHerzog19IntKKT} that $\eta^{\prime}_{1}$ is the unique Lagrange multiplier vector of subproblem \eqref{coordNewtonsubprob}. Thus, we obtain $(\Delta {x^{\prime}_0}^{\ast},\eta^{\prime}_1)=(\Delta x_K^{\ast},\eta_{K+1})$, which, together with $(x^{\prime}_0,\eta_0^{\prime})=(x_K,\eta_K)$ and $\alpha_K=1$ by $K\ge K_0$ and Assumption~\text{B}\ref{newtonstep}, implies 
    \begin{align*}
        \paren{x^{\prime}_{1}, \eta^{\prime}_{1}}=\paren{R_{x_{K}}\paren{\alpha_K\Delta x^{\ast}_{K}},\eta_{K+1}}=\paren{x_{K+1}, \eta_{K+1}},
    \end{align*}
    where the first equality is derived from $x^{\prime}_{1} = R_{x^{\prime}_{0}}\paren{\Delta x^{\prime\ast}_{0}}$, which is implied by \eqref{RiemNewtonRetrProdMani}, and the second equality comes from the definition of $x_{K+1}$ in RSQO. In a similar way, we can prove $\paren{x^{\prime}_{k}, \eta^{\prime}_{k}} = \paren{x_{K+k}, \eta_{K+k}}$ for $k=2,3,\ldots$ in order. This ensures $\paren{x_{K+k},\eta_{K+k}}=\paren{x_{k}^{\prime},\eta_{k}^{\prime}}$ for all $ k\ge 0$, and thus, the proof is complete.
\end{proof}

\section{Numerical experiments}\label{experiment}
We will demonstrate the efficiency of RSQO by using it to numerically solve a problem, nonnegative low-rank matrix completion. For the sake of comparison, we will also solve these problems by using the Riemannian methods presented by Liu and Boumal~\cite{LiuBoumal19Simple}. Parts of the results and the discussions are deferred to Appendix~\ref{appx:suppexperiment}. In addition, we will also experiment on a minimum balanced cut problem in \cite{LiuBoumal19Simple} in Appendix~\ref{appx:suppexperiment}.
All the experiments are implemented in Matlab\_R2020b and Manopt 6.0~\cite{Boumaletal14Manopt} on a Macbook Pro 2019 with 2.4 GHz 8-Core Intel Core i9 CPU and 16.0 GB memory. The code is freely available.\footnote{\url{https://github.com/shirokumakur0/Sequential-quadratic-programming-on-manifold}.}

\subsection{Problem setting: nonnegative low-rank matrix completion}\label{subsec:probsetnonneglowrankcompl}
Briefly, low-rank matrix completion is the problem of recovering a matrix from a sampling of its entries. The problem setting is from Guglielmi and Scalone~\cite{GuglielmiScalone20EfficientMethodforNNLowrankCompl} with modifications of constraints. Compared to \cite{GuglielmiScalone20EfficientMethodforNNLowrankCompl}, our setting has extra equality constraints according to the reliability of the sampled data.

\newcommand{\wholeindset}{S}
\newcommand{\nonnegindset}{J}
\newcommand{\eqindset}{C}

Define a fixed-rank manifold $\mani_{p} \coloneqq \brc{X\in\eucli[q\times s] \relmiddle{|} \text{rank}\paren{X} = p}$. Let $\wholeindset\coloneqq \brc{1,\ldots,q}\times\brc{1,\ldots,s}$ and $\nonnegindset \subseteq \wholeindset$, and let $A\in\eucli[q\times s]$ be a matrix such that the entries on $\nonnegindset$ are known \textit{a priori}. Consider the case that the entries of $A$ on some set $\eqindset(\subseteq\nonnegindset)$ are exact or especially reliable, while those on $\nonnegindset \backslash \eqindset$ may contain noises. Then, the nonnegative low-rank matrix completion problem under the above setting can be represented as
\begin{align}
    \begin{split}\label{RCNN}
        \min_{X\in\mani_{p}} \quad \frac{1}{2}\norm{P_{\nonnegindset \backslash \eqindset}\paren{X-A}}^{2}_{F} \quad \text{s.t.} \quad 
        \begin{aligned}
            &X_{ij} \geq 0 \quad \text{for all } \paren{i,j} \in \wholeindset \backslash \nonnegindset,\\
            &X_{ij} = A_{ij} \quad \text{for all } \paren{i,j} \in \eqindset,
        \end{aligned}
    \end{split}
\end{align}
where, for an input matrix $Z\in\eucli[q\times s]$, $P_{\nonnegindset \backslash \eqindset}\paren{Z}$ is a $q\times s$ matrix whose $\paren{i,j}$-th entry is $Z_{ij}$ if $\paren{i,j}\in \nonnegindset \backslash \eqindset$ and $0$, otherwise. Without loss of generality, we may assume $q\leq s$. Strictly speaking, since $\mani_{p}$ is not a complete manifold~\cite{Boumal20IntroRiemOptBook}, the problem setting does not  completely match ours and that of \cite{LiuBoumal19Simple}. Yet, our algorithm is sufficiently effective as we will see later in Section~\ref{subsub:resultnonneglowrank}.

\newcommand{\leftA}{T}
\newcommand{\rightA}{V}

\textbf{Input.} We consider the cases $\paren{q,s} = \paren{4,8}$ and $\paren{5,10}$. 
We set the rank $p=2$ for each case. Our implementations follow Vandereycken~\cite{Vandereycken13LowRankMatComplbyRiemOpt} with a few modifications: for each $(p,q,s)$, we first generate $\leftA\in\eucli[q\times p]$ and $\rightA \in \eucli[p\times s]$ of uniformly distributed random numbers between 0 and 1. We repeat this process until $\text{rank}\paren{\leftA\rightA}=p$ holds. Then we adopt $A=\leftA\rightA\in\eucli[q\times s]$. We also randomly generate $J$ and $C$ satisfying $\abs{J}=\ceil{\abs{S}\slash 2}$ and $\abs{C}=\ceil{\abs{J}\slash 2}$, where $\ceil{\cdot}$ is the ceiling function.

\newcommand{\coordhessmat}{\widehat{\text{H}\mathcal{L}}_{k}}
\newcommand{\unitaryhessmat}{\widehat{U}_{k}}
\newcommand{\diaghessmat}{\widehat{\Lambda}_{k}}
\subsection{Experimental environment}\label{subsection:experimental_environment} 
Throughout the experiments, RSQO solves subproblem~\eqref{QP} in the following manner. We use a modified Hessian of the Lagrangian as $B_{k}$ in \eqref{QP}: at each $k$-th iteration, we randomly generate an orthonormal basis of $\tanspc[x_{k}]\mani$, denoted by $\brc{\partdiff{}{e_{1}},\ldots, \partdiff{}{e_{d}}}$, with the Gram-Schmidt process. Using it, we compute the Hessian matrix of the Lagrangian $\coordhessmat\in\eucli[d\times d]$ whose $\paren{i,j}$-th element is $\metr[x_{k}]{\text{Hess}\mathcal{L}_{\mu_{k}, \lambda_{k}}\paren{x_{k}}\sbra{\partdiff{}{e_{i}}}}{ \partdiff{}{e_{j}}}$. We can compute the inner product with the projections of the Euclidean Hessian and the Euclidean gradient onto $\tanspc[x_{k}]\mani$~\cite{Boumal20IntroRiemOptBook}. Here it holds that $\metr[]{\text{Hess}\mathcal{L}_{\mu_{k}, \lambda_{k}}\paren{x_{k}}\sbra{\xi}}{\zeta} = \widehat{\xi}^{\top}\coordhessmat \widehat{\zeta}$ for any $\xi,\zeta\in\tanspc[x_{k}]\mani$. We make use of the function \texttt{hessianmatrix} from Manopt~\cite{Boumaletal14Manopt} to execute these procedure. Then we decompose $\coordhessmat  = {\unitaryhessmat}^{\top}\, \diaghessmat \, \unitaryhessmat$, where $\unitaryhessmat$ is a unitary matrix and $\diaghessmat$ is the diagonal one. We modify $\diaghessmat$ to the diagonal matrix $\diaghessmat^{+}$ defined by
\begin{align*}
    \diaghessmat^{+}\paren{i,j} \coloneqq
        \begin{cases}
            \max\paren{\delta, \diaghessmat\paren{i,i}} & \text{if } i=j,\\
            \diaghessmat\paren{i,j} & \text{otherwise},
        \end{cases}
\end{align*}
where $\delta > 0$ is prefixed. Then, we set $\metr[]{B_{k}\sbra{\Delta x_{k}}}{\Delta x_{k}} = \widehat{\Delta x_{k}}^{\top}\coordhessmat^{+}\widehat{\Delta x_{k}}$ with $\coordhessmat^{+} \coloneqq {\unitaryhessmat}^{\top}\, \diaghessmat^{+}\, \unitaryhessmat$ in subproblem~\eqref{QP}. Notice that $\coordhessmat^{+}=\coordhessmat$ holds if $\coordhessmat$ is positive-definite. 
We set $\delta = 10^{-8}$ for the minimum cut problem and $\delta = 10^{-5}$ for the nonnegative low-rank matrix completion.

Similarly, we translate the Riemannian gradients into their coordinate expressions in subproblem~\eqref{QP}: using the same basis of $\tanspc[x_{k}]\mani$, we compute $\text{D}\widehat{f}\paren{\widehat{x_{k}}}\in\eucli[d]$ whose $l$-th element is $\metr[]{\text{grad}f\paren{x_{k}}}{\partdiff{}{e_{l}}}$. We also compute $\brc{\text{D}\widehat{g_{i}}\paren{\widehat{x_{k}}}}_{i\in\mathcal{I}}$ and $\brc{\text{D}\widehat{h_{j}}\paren{\widehat{x_{k}}}}_{j\in\mathcal{E}}$ in the same manner.

Using them, we organize the following Euclidean form of subproblem~\eqref{QP}:
\begin{align}
    \begin{split}\label{NumericalQP}
        \underset{\widehat{\Delta x_{k}} \in \eucli[qs]}{\text{minimize}} \quad & \frac{1}{2} \widehat{\Delta x_{k}}^{\top} \coordhessmat^{+} \widehat{\Delta x_{k}} + {\text{D}\widehat{f}\paren{\widehat{x_{k}}}}^{\top}\widehat{\Delta x_{k}} \\
        \text{subject to}\quad  & \widehat{g_{i}}\paren{\widehat{x_{k}}} + {\text{D}\widehat{g_{i}}\paren{\widehat{x_{k}}}}^{\top}\widehat{\Delta x_{k}} \leq 0, \text{ for all }i \in \mathcal{I},\\ 
        & \widehat{h_{j}}\paren{\widehat{x_{k}}} + {\text{D}\widehat{h_{j}}\paren{\widehat{x_{k}}}}^{\top}\widehat{\Delta x_{k}} = 0, \text{ for all } j \in \mathcal{E}.
    \end{split}
\end{align}
We solve \eqref{NumericalQP} by \texttt{quadprog}, a Matlab solver for quadratic optimization problems, in which \texttt{interior-point-convex} algorithm is selected with the default setting.

We compare our method with the Riemannian methods proposed by Liu and Boumal~\cite{LiuBoumal19Simple}; that is, we compare the following algorithms:
\begin{itemize}
    \renewcommand\labelitemi{--}
    \item \textbf{RSQO (Our method)}: Riemannian sequential quadratic optimization
    \item RALM: Riemannian augmented Lagrangian method in~\cite{LiuBoumal19Simple} 
    \item REPM(LQH): Riemannian exact penalty method with smoothing functions (linear-quadratic and pseudo-Huber) in~\cite{LiuBoumal19Simple}
    \item REPM(LSE): Riemannian exact penalty method with smoothing functions (log-sum-exp) in~\cite{LiuBoumal19Simple}
\end{itemize}

To measure the deviation of an iterate from the set of KKT points, we use residuals based on the KKT conditions \eqref{KKT} and the manifold constraints of the problem. 
In the nonnegative low-rank matrix completion problem, the residual is
    {\small 
        \begin{align*}
            \sqrt{\norm{\text{grad}\,\mathcal{L}_{\mu}\paren{X}}^{2} + \sum_{i\in\mathcal{I}} \paren{\max\paren{0,-\mu_{i}}^{2} + \max\paren{0,g_{i}\paren{X}}^{2} + \paren{\mu_{i}g_{i}\paren{X}}^{2} + \sum_{j\in\mathcal{E}} \abs{h_{j}\paren{X}}^{2} }} + \iota_{p}\paren{X},
        \end{align*}
    }
where each term in the square root is from the KKT conditions \eqref{KKTLag}, \eqref{KKTineq}, \eqref{KKTcompl}, and \eqref{KKTeq}, and the last one is the indicator function defined by $\iota_{p}\paren{X} \coloneqq 0$ if $\text{rank}\paren{X} = p$ and $+\infty$, otherwise.

The stopping criteria are based on a maximal iteration, maximal time, and changes in parameters, and will be explained in detail in our discussion of each experiment. We set the parameters as $\varepsilon = 0.5, \rho_{-1} = 1, \beta = 0.9$, and $\gamma = 0.25$ for RSQO. Regarding the implementation of RALM and REPMs, we utilize the environment provided by Liu~\cite{Liu19RALMcode} on Manopt, a Riemannian optimization toolbox on Matlab, after making some modifications.

\subsection{Numerical results}\label{subsub:resultnonneglowrank}

We applied the algorithms to the nonnegative low-rank matrix completion problem~\eqref{RCNN}. As for the initial point, each algorithm ran from the same feasible point that was numerically obtained in advance: given $A$, $\nonnegindset$, and $\eqindset$, we first solved 
$\min_{X\in\mani_{p}} 0 \text{ s.t. } X_{ij} \geq 0 \text{ for all } \paren{i,j}\in \wholeindset \backslash \nonnegindset \text{ and } X_{ij} = A_{ij} \text{ for all } \paren{i,j}\in \eqindset$ by REPM(LQH) until we obtained a solution with residual$=10^{-2}$. We adopted the solution as the initial point. If the spent time exceeded 600 seconds, the iteration number was over 100,000, or the algorithm did not update any parameters, the algorithm was terminated. As for the numerical settings of RALM and REPMs, we mostly employed the original ones in \cite{Liu19RALMcode}, but set $\epsilon_{\min} = 10^{-16}$ and $d_{\min} = 0$ so as to prevent the algorithms from being terminated when the step length got too small.

Figure~\ref{Fig:NonnegResTime} shows the residual of the algorithms. RSQO successfully solved the instance with the highest accuracy of the residual less than $10^{-10}$, while the accuracies of other solutions were more than $10^{-6}$. RSQO tends to steadily decrease the residual at first and then much more time is necessary to get more accurate solutions.
In the computation of RSQO, it occupied more than $80\%$ of the whole running time to call the function \texttt{hessianmatrix}; 
see Section~\ref{subsection:experimental_environment} for the detail of the function. This means that constructing the Hessian of the Lagrangian is the most expensive.

\begin{figure}[t]
    \begin{center}
        \includegraphics[width=0.8\linewidth]{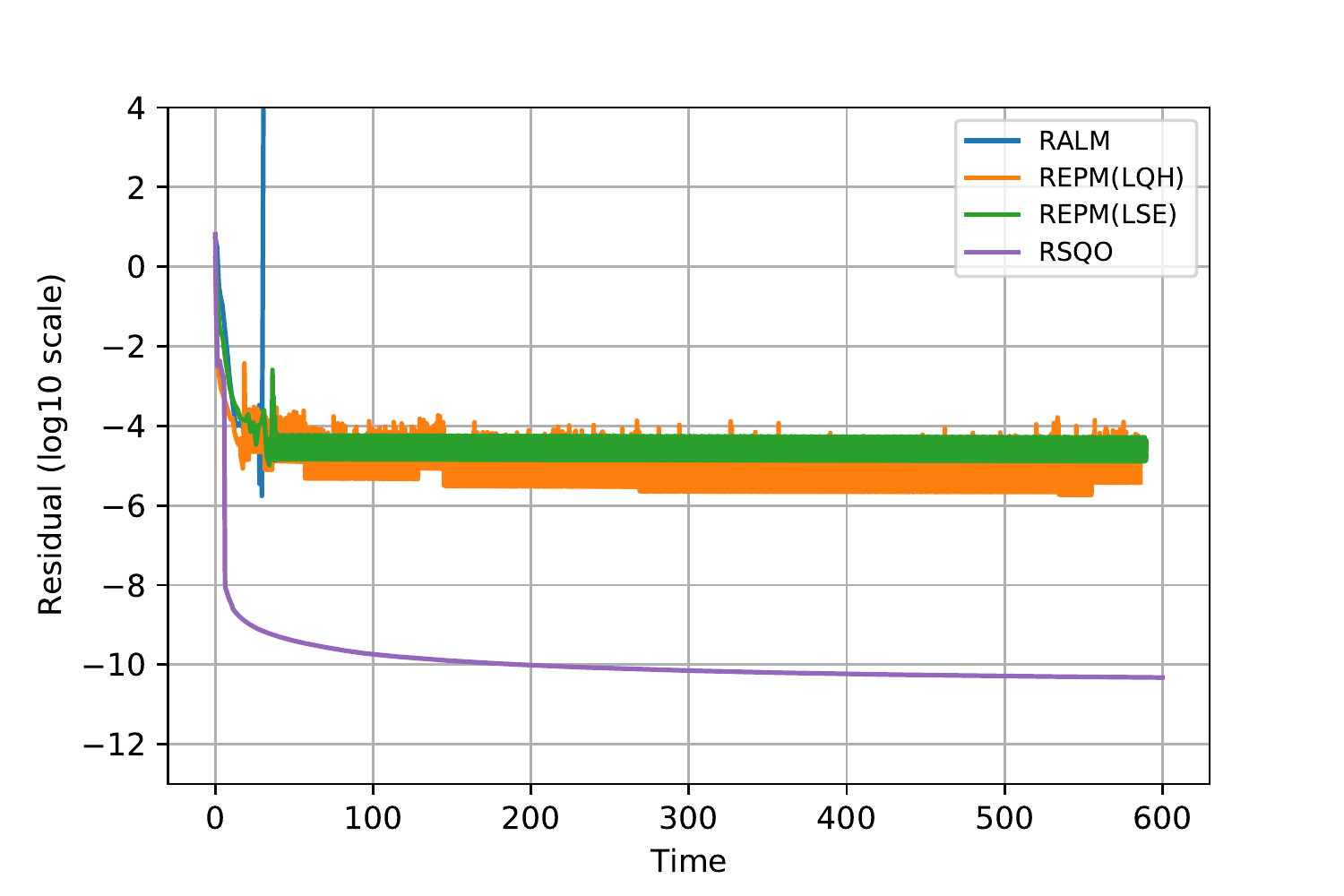}
    \end{center}
    \caption{Residual for a nonnegative low-rank matrix completion problem}
	\label{Fig:NonnegResTime}
\end{figure}

We also conducted experiments under other settings to measure the speed and robustness of the algorithms. For $\paren{q,s}=\paren{4,8}$ and $\paren{5,10}$, we conducted experiments 20 times and measured the average CPU time and iterations between cases where RSQO could reach solutions with residual $10^{-6}$. Each experiment was terminated if a solution with residual$=10^{-6}$ was found, the spent time exceeded 60 seconds, iterations went over 1,000, or neither iterate nor parameters were updated. We set $\epsilon_{\min} = 10^{-6}$ and $d_{\min} = 0$ for RALM and REPMs.

From Table~\ref{20trialsLowRankCompl}, we observed that RSQO successfully solved almost all instances both in $\paren{q,s}=\paren{4,8}$ and in $\paren{5,10}$. In comparison with RALM, since RALM reached the solution for at most 25\% of the instances, the result indicates that RSQO can solve the instances more stably. In comparison with REPMs, we can observe that RSQO solves the instances not only more stably but also faster.

\begin{table}[t]
\begin{center}
\caption{The ratio of successful trials out of 20. The average time and the average number of iterations among the successful trials are also presented in each setting.}
\begin{tabular}{c|ccc|ccc}
\hline
$\paren{q,s}$ & \multicolumn{3}{c|}{(4,8)}                  & \multicolumn{3}{c}{(5,10)}                  \\ \cline{2-7} 
              & success (\%) & time (sec.)       & \# iter. & success (\%) & time (sec.)   & \# iter. \\ \hline
RSQO          & $95$         & $3.314$           & $72$     & $100$        & $8.656$           & $118$    \\
RALM          & $25$         & $4.011$           & $22$     & $20$         & $3.948$           & $20$     \\
REPM (LQH)    & $60$         & $1.050 \times 10$ & $45$     & $45$         & $1.869 \times 10$ & $950$    \\
REPM (LES)    & $70$         & $1.009 \times 10$ & $386$    & $30$         & $2.596 \times 10$ & $721$    \\ \hline
\end{tabular}
\label{20trialsLowRankCompl}
\end{center}
\end{table}

\subsection{Further experiments}\label{subsec:furexp}
In addition to the experiments in the preceding subsections, we conducted additional experiments. 
Here, we briefly summarize the results; the details are deferred to Appendix~\ref{appx:suppexperiment}.

As for the nonnegative low-rank completion problem \eqref{RCNN}, we further investigated the behaviors of RSQO and the other methods in different sizes of the problems.
The empirical results show that RSQO tends to compute the solution more accurately than the other methods. Nevertheless, the average CPU time per step for RSQO to reach its most accurate solution drastically increases as the problem size does because of the expensive computation of the Hessian matrix.

We also experimented on another problem, a minimum balanced cut problem. This is an optimization problem on a complete manifold, called an oblique manifold, with equality constraints introduced in~\cite{LiuBoumal19Simple}. In addition to the Riemannian methods, we applied \texttt{fmincon SQO}, an Euclidean SQO solver, to this problem because the problem can be formulated as a nonlinear optimization problem on the Euclidean space.
In the experiment, RSQO drastically decreased its residual, which might be due to the quadratic convergence property of RSQO as a result of using the Hessian of the Lagrangian. 
In spite of that \textbf{fmincon SQO} and RSQO share the same SQO framework, \textbf{fmincon SQO} did not work at all. This fact may underscore an advantage of the Riemannian manifold approach.

\section{Discussion and conclusion}
\label{conclusion}
We proposed a Riemannian sequential quadratic optimization (RSQO) method for RNLO \eqref{RNLO}. We proved the global and local convergence properties of the algorithm in Section~\ref{method} and conducted numerical experiments comparing it with the Riemannian augmented Lagrangian method, Riemannian exact penalty methods, and a Matlab solver \texttt{fmincon} using SQO in Section~\ref{experiment} and Appendix~\ref{appx:suppexperiment}.
We found that RSQO solved the problems more stably and with higher accuracy. However, the execution time of RSQO increased drastically as the problem size grew. In closing, we discuss future directions for more advanced RSQO methods:
\begin{enumerate}
    \item \textbf{Large-scale optimization:} As we mentioned in Section~\ref{subsec:furexp} and Appendix~\ref{appx:suppexperiment}, the CPU time of RSQO drastically increases as the problem size does. This is mainly due to the computation of the Hessian matrix. In addition, the eigenvalue decomposition of the Hessian to retain the positive-definiteness may be also costly in general. Thus, one direction would be to develop an efficient update of the coefficient operator $B_k$ in the quadratic optimization subproblem of RSQO without the Hessian. Such a formula may be obtained by extending the Powell-symmetric-Broyden or BFGS one~\cite{BoggsTolle96SQPsurvey} for the Euclidean SQO to the Riemannian case although this work seems tough in light of the history of the development of the unconstrained Riemannian BFGS formulae~\cite{Absiletal08OptBook, RingWirth12BFGSShapeSpace, Huangetal15BroydenLBFGS, Huangetal18RBFGSforNonconv}.
    
    Another effective technique for large-scale optimization would be to use inexact solutions of subproblem~\eqref{QP}. 
    Yet, in light of inexact Euclidean SQO methods~\cite{Byrd08iSQOEqualityConsts, Curtisetal14InexcatSQPforNonlinearOptim, WaltherBiegler16InexactTrustRegionFilterSQP, IzmailovSolodov10truncatedSQP, Burkeetal20inexactSQOpenaparamwithinQP}, it seems that additional termination criteria or different algorithmic structures would be necessary so as to establish the theoretical guarantees in the Riemannian setting.

    \item \textbf{Treatment of infeasible subproblems:} In relation to Assumption A\ref{subprobfeasi}, various SQO methods have been proposed to deal with the infeasibility of the subproblems in Euclidean cases. For example, when the infeasibility of the subproblem is detected, some SQO methods~\cite{Gilletal02SNOPT, Tone83ElasticModeSQP} enter the elastic mode, where the constraints are modified to be consistent. Robust SQO~\cite{BurkeHan89robustSQP} is also a variant to circumvent the infeasibility of the subproblems by relaxing the feasible region of the subproblem. Another way is the feasibility restoration phase, which is an effective technique for filter SQO methods to restore the feasibility when the subproblem is infeasible due to the trust-region radius~\cite{FletcherLeyffer02filterSQP, Fletcheretal02filterSQPGlobalConv, Ulbrich04filterSQPsuplinconv}. Extending these techniques to the Riemannian case may mitigate the assumptions, in particular, Assumption~A\ref{subprobfeasi}.
    
    \item \textbf{Treatment of unboundedness of the Lagrange multipliers:} We supposed the boundedness of the Lagrange multiplier sequence in Assumption~A\ref{genebound}. 
    Such assumptions were circumvented or mitigated in some existing researches on Euclidean SQO. For example, in a stabilized SQO or SQP method, which was initiated by Wright~\cite{Wright97StabilizedSQPtoDegeneSol} to attain superlinear convergence to degenerate solutions of NLO problems, the global convergence results are often established in the absence of such assumptions~\cite{Gilletal17StabilizedSQPGlobConv, GillRobinson13GlobConvStabilizedSQP}. Extending such SQO methods to the Riemannian setting may help us to make RSQO methods more practical.
    
    \item \textbf{Avoidance of the Maratos effect:} In relation to Assumption~B\ref{newtonstep}, an interesting direction would be to develop a means avoiding the Maratos effect; that is, $\alpha_{k}=1$ may be rejected when using a line-search technique with the $\ell_{1}$ penalty function. In consideration of discussions in Euclidean cases, it seems to be difficult to resolve the issue without additional techniques such as the second-order correction~\cite{Fukushima86SQPavoidMaratos}.
\end{enumerate}

\section*{Acknowledgement}
The authors are grateful to the associate editor and the two anonymous referees for their valuable comments and suggestions.

\appendix
\setcounter{section}{0}
\renewcommand{\thesection}{\Alph{section}}
\setcounter{equation}{0}
\renewcommand{\theequation}{\thesection-\arabic{equation}}

\section{Proof of Proposition~\ref{prop:SuffCondforSubProbFeasi}}
\label{SubsecAppen:PrfofPropSuffcond}
Let us start by introducing the notion of convexity on Riemannian manifolds. For details, we refer the reader to \cite{Boumal20IntroRiemOptBook}.

\begin{definition} {\upshape (\cite[Definition 11.2]{Boumal20IntroRiemOptBook})}
A set $H \subseteq \mani$ is said to be a geodesically convex set with respect to the Riemannian metric $\metr[]{\cdot}{\cdot}$ if, for any $x,y\in H$, there exists a geodesic $\gamma_{xy}:\sbra{0,1}\rightarrow\mani$ that joins $x$ to $y$, i.e., $\gamma_{xy}\paren{0}=x$ and $\gamma_{xy}\paren{1}=y$, and lies entirely in $H$.
\end{definition}
Note that a connected and complete manifold $\mani$ itself is geodesically convex. Moreover, we shall define a geodesically convex function via a first-order approximation. 
\begin{definition} {\upshape(\cite[Definition 11.4, Theorem 11.17]{Boumal20IntroRiemOptBook})}
Let $H\subseteq \mani$ be a geodesically convex set with respect to $\metr[]{\cdot}{\cdot}$. A differentiable function $\theta:H\rightarrow\eucli$ is said to be a geodesically convex function with respect to $\metr[]{\cdot}{\cdot}$ if, for any $x,y\in H$ and any geodesic segment $\gamma_{xy}:\sbra{0,1}\rightarrow \mani$ that joins $x$ to $y$ and lies entirely in $H$,
\begin{align}\label{gconvexdef}
            \theta\paren{x} + t\mathrm{D}\theta\paren{x}\sbra{\xi_{x}^{y}} \leq \theta\paren{\gamma_{xy}\paren{t}}, \quad \forall t\in\sbra{0,1}
\end{align}
holds, where $\xi_{x}^{y}\in\tanspc\mani$ denotes the tangent vector corresponding to $\gamma_{xy}$. We call $\theta$ a geodesically linear function if both $\theta$ and $-\theta$ are geodesically convex. 
\end{definition}
Note that the geodesically linear function is a generalization of the standard linear function on $\eucli[d]$. Sra et al.~\cite{Sraetal18gConvexforBLC} introduced a log-determinant function as a geodesically linear function on a positive-definite cone.

\begin{proof}[Proof of Proposition~\ref{prop:SuffCondforSubProbFeasi}]
Recall that $\bar{x}$ is a feasible solution of RNLO \eqref{RNLO}. For any $x\in\mani$, there exists a geodesic segment $\gamma_{x\bar{x}}$ by the Hopf-Rinow theorem. Let $\xi_{x}^{\bar{x}}$ be the corresponding tangent vector to $\gamma_{x\bar{x}}$. Then, for all $i\in\mathcal{I}$, we have
\begin{align*}
    0 \geq g_{i}\paren{\bar{x}} \geq \text{D}g_{i}\paren{x}\sbra{\xi_{x}^{\bar{x}}} + g_{i}\paren{x} = \metr[]{\text{\upshape grad}\,g_{i}\paren{x}}{\xi_{x}^{\bar{x}}} + g_{i}\paren{x},
\end{align*}
where the first inequality follows from the feasibility of $\bar{x}$ and the second one from \eqref{gconvexdef} with $t=1$.
Similarly, for all $j\in\mathcal{E}$,
\begin{align*}
    0 = h_{j}\paren{\bar{x}} = \text{D}h_{j}\paren{x}\sbra{\xi_{x}^{\bar{x}}} + h_{j}\paren{x} = \metr[]{\text{\upshape grad}\,h_{j}\paren{x}}{\xi_{x}^{\bar{x}}} + h_{j}\paren{x}
\end{align*}
holds. Hence by setting $x=x_{k}$, we obtain that $\xi_{x_{k}}^{\bar{x}}$ is a feasible solution of \eqref{QP} for every iteration $k$.
\end{proof}

\section{Proof of Lemma~\ref{AkOperNormBounded}} \label{appendix:prfofAkOperNormBounded}
\begin{proof}[Proof of Lemma~\ref{AkOperNormBounded}]
    Without loss of generality, we may assume that $\coordmetr[x]$, the coordinate expression of the Riemannian metric at $x$, is the identity matrix and $\norm{\xi}_{x} = \sqrt{{\widehat{\xi}}^{\top}\widehat{\xi}}$; see \cite[Section 1.2.7]{Huang13PhD} for a justification of this assumption. Then, the uniform positive-definiteness of $\mathcal{A}_{x}$ reads
    \begin{align}\label{coordAkbounded}
        m {\widehat{\xi}}^{\top}\widehat{\xi} \leq {\widehat{\xi}}^{\top}\widehat{\mathcal{A}_{x}}\widehat{\xi} \leq M {\widehat{\xi}}^{\top}\widehat{\xi}.
    \end{align}
    Note that the symmetry and positive-definiteness of $\mathcal{A}_{x}$ ensure those of $\widehat{\mathcal{A}_{x}}$. Moreover, $\widehat{\mathcal{A}_{x}^{-1}} = \widehat{\mathcal{A}_{x}}^{-1}$ holds; that is, the coordinate expression of the inverse mapping of $\mathcal{A}_{x}$ is the inverse matrix of $\widehat{\mathcal{A}_{x}}$. Indeed, we have
    \begin{align*}
        \widehat{\mathcal{A}_{x}^{-1}}\widehat{\mathcal{A}_{x}} 
        &= \paren{\text{D}\varphi\paren{x} \circ \mathcal{A}_{x}^{-1} \circ \text{D}\varphi\paren{x}^{-1}}\paren{\text{D}\varphi\paren{x} \circ \mathcal{A}_{x} \circ \text{D}\varphi\paren{x}^{-1}}\\
        &= \text{D}\varphi\paren{x} \circ \text{id}_{\tanspc[x]\mani} \circ \text{D}\varphi\paren{x}^{-1}\\ 
        &= \widehat{E},
    \end{align*}
    where $\text{id}_{\tanspc[x]\mani}$ denotes the identity mapping on $\tanspc[x]\mani$ and $\widehat{E}\in\eucli[d\times d]$ is the identity matrix. Thus, \eqref{coordAkbounded} implies that all eigenvalues of the symmetric matrices $\widehat{\mathcal{A}_{x}}$ and $\widehat{\mathcal{A}_{x}^{-1}}$ are not greater than $M$ and $1 \slash m$, respectively. Hence, we have 
    \begin{align*}
        \norm{\widehat{\mathcal{A}_{x}}}_{\text{F}} \leq M\sqrt{d} \text{ and } \norm{\widehat{\mathcal{A}_{x}^{-1}}}_{\text{F}} \leq \frac{\sqrt{d}}{m},
    \end{align*}
    where $\norm{\cdot}_{\text{F}}$ denotes the Frobenius norm on $\eucli[d\times d]$. Additionally, by \cite[Lemma 6.2.6]{Huang13PhD}, 
    \begin{align*}
        \norm{\mathcal{A}_{x}}_{\rm op} \leq \norm{\widehat{\mathcal{A}_{x}}}_{\text{F}} \text{ and } \norm{\mathcal{A}_{x}^{-1}}_{\rm op} \leq \norm{\widehat{\mathcal{A}_{x}^{-1}}}_{\text{F}}
    \end{align*}
    follow, where $\norm{\cdot}_{\rm op}$ is the operator norm. By combining these inequalities above, we conclude $\norm{\mathcal{A}_{x}}_{\rm op} \leq M\sqrt{d}$ and $\norm{\mathcal{A}_{x}^{-1}}_{\rm op} \leq \sqrt{d}\slash m$.
\end{proof}

\section{Proof of Proposition~\ref{LimBIneqLimP}} \label{appendix:ExistenceofLimDeltaXBandProperties}
Here, we aim to prove Proposition~\ref{LimBIneqLimP}. First, we will introduce some concepts from Riemannian and nonsmooth optimization theories and then prove Lemma~\ref{existDxBastwithproperties} and \ref{lem:clark} as preliminary results.

\subsection{Additional preliminaries}
Here, we briefly review additional tools for Riemannian optimization, presented in \cite{Boumal20IntroRiemOptBook}, for the sake of proving Proposition~\ref{LimBIneqLimP}.
We also describe some concepts of nonsmooth optimization from \cite{Clarke90OptNonSmoothAnal}.

\subsubsection{Additional tools for Riemannian optimization}
For each $x\in\mani$, let $\text{Exp}_{x}:\tanspc[x]\mani\rightarrow\mani$ denote the exponential mapping at $x$, the mapping such that $t\mapsto \text{Exp}_{x}\paren{t\xi}$ is the unique geodesic that passes through $x$ with velocity $\xi\in\tanspc[x]\mani$ when $t=0$. Note that $\text{Exp}:\tanspc[]\mani\rightarrow\mani$ is smooth. The injectivity radius at $x$ is defined as
\begin{align*}
\text{Inj}\paren{x} \coloneqq \sup\brc{r>0 \relmiddle{|} \left.\text{Exp}_{x}\right|_{\brc{\gamma\in\tanspc[x]\mani\relmiddle{|}\norm{\gamma}<r}} \text{ is a diffeomorphism}}.
\end{align*}
Note that $\text{Inj}\paren{x}>0$ for any $x\in\mani$. For any $y\in\mani$ with $\text{dist}\paren{x,y}<\text{Inj}\paren{x}$, there is a unique minimizing geodesic connecting $x$ and $y$, which induces parallel transport along the minimizing geodesic $\Pi_{x\rightarrow y}:\tanspc[x]\mani\rightarrow\tanspc[y]\mani$. Note that the parallel transport is isometric, i.e., $\norm{\Pi_{x\rightarrow y}\sbra{\xi_{x}}}_{y} = \norm{\xi_{x}}_{x}$ for any $\xi_{x}\in\tanspc[x]\mani$ and $\Pi_{x\rightarrow x}$ is the identity mapping on $\tanspc[x]\mani$. Additionally, it follows that $\Pi_{x\rightarrow y}^{-1} = \Pi_{y\rightarrow x}$. The adjoint of the parallel transport corresponds with its inverse, that is, $\metr[y]{\Pi_{x\rightarrow y}\sbra{\xi_{x}}}{\,\zeta_{y}} = \metr[x]{\xi_{x}}{\,\Pi_{y\rightarrow x}\sbra{\zeta_{y}}}$ for all $\xi_{x}\in\tanspc[x]\mani$ and $\zeta_{y}\in\tanspc[y]\mani$.
We also introduce a property of the limit of the gradient with the parallel transport. 
\begin{lemma}{\rm (\cite[Lemma A.2.]{LiuBoumal19Simple})}\label{lem:limParallelGrad}
    Given $x\in\mani$ and a sequence $\brc{x_{k}}$ such that ${\rm dist}\paren{x_{k},x} \allowbreak < {\rm Inj}\paren{x}$ for each $k$ and $\brc{x_{k}}$ converges to $x$. Then, for a continuously differentiable function $\theta:\mani\rightarrow\eucli[]$, the following holds:
    \begin{align*}
        \lim_{k\rightarrow\infty} \Pi_{x_{k}\rightarrow x}\sbra{{\rm grad}\theta\paren{x_{k}}} = {\rm grad}\theta\paren{x},
    \end{align*}
    where $\Pi_{x_{k}\rightarrow x}$ is the parallel transport along the minimizing geodesic.
\end{lemma}

\subsubsection{Notation and terminology from nonsmooth optimization on \texorpdfstring{$\tanspc[x]\mani^{2}$}{the product space of tangent spaces}}
Let $x\in\mani$ be an arbitrary point. Recall that $\tanspc[x]\mani$ is a $d$-dimensional inner product space. Thus, $\tanspc[x]\mani\oplus\tanspc[x]\mani$ is the $2d$-dimensional inner product space, where $\oplus$ is the direct sum. An element of $\tanspc[x]\mani\oplus\tanspc[x]\mani$ is expressed as $\xi^{1} \oplus \xi^{2}$ with $\xi^{1}, \xi^{2} \in\tanspc[x]\mani$. Hereafter, for brevity, we often use the notations $\tanspc[x]\mani^{2}$ and $\xi^{\oplus}$ instead of $\tanspc[x]\mani\oplus\tanspc[x]\mani$ and $\xi^{1} \oplus \xi^{2}$, respectively. To simplify our descriptions, we will `translate' some concepts from nonsmooth analysis \cite{Clarke90OptNonSmoothAnal}. Specifically, we will redefine Clarke regularity and generalized derivatives in terms of $\tanspc[x]\mani^{2}$ and introduce some of the related properties.

Let $\zeta^{\oplus}, \xi^{\oplus}, \chi^{\oplus} \in\tanspc[x]\mani^{2}$ and $l:\tanspc[x]\mani^{2}\rightarrow\eucli[]$ be Lipschitz continuous near $\zeta^{\oplus}$; i.e., there exists a Lipschitz constant $L\geq 0$ such that $\abs{l\paren{\chi^{\oplus}} - l\paren{\xi^{\oplus}}} \leq L \norm{\chi^{\oplus} - \xi^{\oplus}}$ for all $\xi^{\oplus}, \chi^{\oplus}\in\tanspc[x]\mani^{2}$ within a neighborhood of $\zeta^{\oplus}$. The generalized directional derivative of $l$ at $\zeta^{\oplus}$ in the direction $\xi^{\oplus}$, denoted by $l^{\circ}\paren{\zeta^{\oplus};\xi^{\oplus}}$, is defined as follows:
\begin{align*}
    l^{\circ}\paren{\zeta^{\oplus};\xi^{\oplus}} \coloneqq \limsup_{\chi^{\oplus}\rightarrow \zeta^{\oplus}, t\downarrow 0}\frac{l\paren{\chi^{\oplus}+t\xi^{\oplus}} - l\paren{\chi^{\oplus}}}{t},
\end{align*}
where $\chi^{\oplus}$ is a vector in $\tanspc[x]\mani^{2}$ and $t$ is a positive scalar. Moreover, the generalized gradient of $l$ at $\zeta^{\oplus}$, denoted by $\partial l\paren{\zeta^{\oplus}}$, is defined as 
\begin{align*}
    \partial l\paren{\zeta^{\oplus}} \coloneqq \brc{\phi\in \cotanspc[x]\mani^{2} \relmiddle{|}  \phi\sbra{\xi^{\oplus}} \leq l^{\circ}\paren{\zeta^{\oplus};\xi^{\oplus}} \text{ for all } \xi^{\oplus} \text{ in } \tanspc[x]\mani^{2}},
\end{align*}
where $\cotanspc[x]\mani^{2}$ is the dual space of $\tanspc[x]\mani^{2}$, namely, the set of linear mappings from $\tanspc[x]\mani$ to $\eucli[]$.
Formally state the following as a proposition: The following holds from \cite[Proposition 2.1.5 (b)]{Clarke90OptNonSmoothAnal}.
\begin{proposition} \label{prop:genegradclosedmapping}
    $\partial l\paren{\cdot}$ is a closed point-to-set mapping: let $\brc{\zeta_{i}^{\oplus}}\subseteq\tanspc[x]\mani^{2}$ and $\brc{\phi_{i}}\subseteq\cotanspc[x]\mani^{2}$ be sequences such that $\phi_{i}\in\partial l\paren{\zeta^{\oplus}_{i}}$ for each $i$. Supposing that $\brc{\zeta^{\oplus}_{i}}$ converges to $\zeta^{\oplus}_{\ast}$ and $\phi_{\ast}$ is an accumulation point of $\brc{\phi_{i}}$, \footnote{Note that one can always take such an accumulation point $\phi_{\ast}$ in finite-dimensional cases. Indeed, for any $i$ sufficiently large, the definitions of $\partial l\paren{\zeta^{\oplus}_{\ast}}$ and $l^{\circ}$ ensure that $\phi_{i}\sbra{\xi^{\oplus}} \leq l^{\circ}\paren{\zeta^{\oplus}_{i}; \xi^{\oplus}} \leq L \norm{\xi^{\oplus}}$ holds for all $\xi^{\oplus}\in\tanspc[x]\mani^{2}$, where $L$ is the Lipschitz constant of $l$ near $\zeta^{\oplus}_{\ast}$. Thus, by using the dual norm on $\cotanspc[x]\mani^{2}$, we see that $\norm{\phi_{i}} \leq L$ holds for all $i$ sufficiently large, which ensures the existence of a convergent subsequence.} one has $\phi_{\ast}\in\partial l\paren{\zeta^{\oplus}_{\ast}}$.
\end{proposition}

Here, we define the Clarke regularity of functions on $\tanspc[x]\mani^{2}$.
\begin{definition}{\upshape (\cite[Definition 2.3.4]{Clarke90OptNonSmoothAnal})}
    The function $l:\tanspc[x]\mani^{2}\rightarrow\eucli[]$ is said to be Clarke regular at $\zeta^{\oplus}$ provided that, for all $\xi^{\oplus} \in \tanspc[x]\mani^{2}$, the one-sided directional derivative $l^{\prime}\paren{\zeta^{\oplus};\xi^{\oplus}}$ exists and $l^{\prime}\paren{\zeta^{\oplus};\xi^{\oplus}} = l^{\circ}\paren{\zeta^{\oplus};\xi^{\oplus}}$.
\end{definition}

Now let us describe some of the properties of Clarke regularity by tailoring \cite[Proposition 2.3.6, Theorem 2.3.10]{Clarke90OptNonSmoothAnal}.
\begin{proposition}\label{regularityproperties}
    Given $\zeta^{\oplus}\in\tanspc[x]\mani^{2}$, let $l_{1}:\eucli\rightarrow\eucli$ and $l_{2}:\tanspc[x]\mani^{2}\rightarrow\eucli$ be Lipschitz continuous near $l_{2}\paren{\zeta^{\oplus}}$ and $\zeta^{\oplus}$, respectively.
    \begin{enumerate}[{\rm (a)}]
        \item If $l_{1}$ is convex, then $l_{1}$ is Clarke regular at $l_{2}\paren{\zeta^{\oplus}}$. \label{convexregularity}
        \item If $l_{2}$ is continuously differentiable at $\zeta^{\oplus}$ and $l_{1}$ is Clarke regular at $l_{2}\paren{\zeta^{\oplus}}$, then the composite function $l_{1} \circ l_{2} $ is Lipschitz continuous near $\zeta^{\oplus}$ and Clarke regular at $\zeta^{\oplus}$.\label{chainruleregularity}
        \item A finite linear combination by nonnegative scalars of functions being Clarke regular at $\zeta^{\oplus}$ is Clarke regular at $\zeta^{\oplus}$.\label{lincombregularity}
    \end{enumerate}
\end{proposition}
We also have the following mean-value theorem for nonsmooth functions, from \cite[Theorem 2.3.7]{Clarke90OptNonSmoothAnal}. 
\begin{theorem} \label{LebourgMeanValue}
    Let $\xi^{\oplus}, \zeta^{\oplus}\in\tanspc[x]\mani^{2}$ and $t\in\eucli[]\backslash \brc{0} $. Suppose that $l:\tanspc[x]\mani^{2}\rightarrow \eucli$ is Lipschitz continuous on an open set containing the line segment $\sbra{\zeta^{\oplus}, \zeta^{\oplus} + t\xi^{\oplus}}$. Then, there exists some $s\in\paren{0, t}$ such that
    \begin{align*}
        \frac{1}{t}\paren{l\paren{\zeta^{\oplus} + t\xi^{\oplus}} - l\paren{\zeta^{\oplus}}} \in \partial l\paren{\zeta^{\oplus} + s\xi^{\oplus}}\sbra{\xi^{\oplus}},
    \end{align*}
    where $\partial l\paren{\zeta^{\oplus} + s\xi^{\oplus}}\sbra{\xi^{\oplus}} = \brc{\phi\sbra{\xi^{\oplus}} \in \eucli \relmiddle{|} \phi\in \partial l\paren{\zeta^{\oplus} + s\xi^{\oplus}}}$.
\end{theorem}

\subsection{Proof of Proposition~\ref{LimBIneqLimP}}
Throughout this subsection, we will reuse the notation in Proposition~\ref{LimBIneqLimP}.
In particular, recall that $\brc{(x_{k},\mu_k,\lambda_k, \alpha_{k})}_{k\in\mathcal{K}}$ is a subsequence converging to an accumulation point $(x^{\ast},\mu^{\ast},\lambda^{\ast}, 0)$.

\newcommand{\injopenball}{\mathbb{B}_{x^{\ast}}}
\newcommand{\geodesicopenball}{\mathcal{V}}
Define $\mathbb{B}_{x^{\ast}} \coloneqq \brc{ \xi\in\tanspc[x^{\ast}]\mani \relmiddle{|} \norm{\xi} < \text{Inj}\paren{x^{\ast}}}$
 and $\geodesicopenball \coloneqq \text{Exp}_{x^{\ast}}\paren{\injopenball}$. Since $\brc{x_{k}}$ converges to $x^{\ast}$, there exists sufficiently large $\tilde{k}_{2}\paren{\geq \tilde{k}_{1}}$ such that $x_{k}\in\mathcal{V}$ for any $k\geq \tilde{k}_{2}$. For such $k\geq \tilde{k}_{2}$, it holds that the parallel transport with the minimizing geodesic from $x_{k}$ to $x^{\ast}$ is well-defined. Hereafter, we will assume $k\geq\tilde{k}_{2}$. Define $\overline{B_{k}}: \tanspc[x^{\ast}]\mani\rightarrow\tanspc[x^{\ast}]\mani$ and $\overline{\Delta x_{k}^{\ast}}\in\tanspc[x^{\ast}]\mani$ by 
\begin{align*}
    \overline{B_{k}} \coloneqq \Pi_{x_{k}\rightarrow x^{\ast}} \circ B_{k} \circ \Pi_{x^{\ast}\rightarrow  x_{k}} \text{ and } \overline{\Delta x_{k}^{\ast}} \coloneqq \Pi_{x_{k}\rightarrow x^{\ast}} \sbra{\Delta x_{k}^{\ast}}.
\end{align*}

First, let us investigate the existence of accumulation points of $\left\{\overline{B_{k}}\right\}$ and $\left\{\overline{\Delta x_{k}^{\ast}}\right\}$ and their properties.
\begin{lemma}\label{existDxBastwithproperties}
    Under Assumptions A\ref{subprobfeasi}, A\ref{Bkbounded}, and A\ref{genebound}, the following hold:
    \begin{enumerate}[{\rm (a)}]
        \item $\brc{\overline{B_{k}}}$ and $\brc{\overline{\Delta x^{\ast}_{k}}}$ are bounded.\label{subseqconvLineDxkBk}
        \item For every $k$ and any $\xi,\zeta\in\tanspc[x_{k}]\mani$, $\metr[x_{k}]{B_{k}\sbra{\xi}}{\zeta}=\metr[x^{\ast}]{\overline{B_{k}}\sbra{\overline{\xi}}}{\overline{\zeta}}$, where $\overline{\xi}\coloneqq \Pi_{x_{k}\rightarrow x^{\ast}}\sbra{\xi}\in\tanspc[x^{\ast}]\mani$ and $\overline{\zeta}\coloneqq \Pi_{x_{k}\rightarrow x^{\ast}}\sbra{\zeta}\in\tanspc[x^{\ast}]\mani$.\label{statement:BkProdEqBkLine}
    \end{enumerate}
    Let $\overline{B^{\ast}}$ and $\overline{\Delta x^{\ast}}$ be accumulation points of $\brc{\overline{B_{k}}}$ and $\brc{\overline{\Delta x^{\ast}_{k}}}$, respectively.
    \begin{enumerate}[{\rm (a)}]
        \setcounter{enumi}{2}
        \item $\overline{B^{\ast}}$ is symmetric and positive-definite. Additionally, $\paren{\overline{\Delta x^{\ast}}, \mu^{\ast}, \lambda^{\ast}}$ satisfies the KKT conditions \eqref{subprobKKT} of the quadratic optimization problem \eqref{QP} with $(x_k,B_k)$ replaced by $(x^{\ast},\overline{B^{\ast}})$. \label{LineDastBastKKTeq}
    \end{enumerate}
\end{lemma}
\begin{proof}
    As for statement \eqref{subseqconvLineDxkBk}, it holds that
    \begin{align*}
    \begin{split}
    \norm{B_{k}}_{\rm op}
    &= \max_{\xi\in\tanspc[x_{k}]\mani} \frac{\norm{B_{k}\sbra{\xi}}_{x_{k}}}{\norm{\xi}_{x_{k}}}\\
    &= \max_{\xi\in\tanspc[x_{k}]\mani} \frac{\norm{\Pi_{x_{k}\rightarrow x^{\ast}} \circ B_{k} \circ \Pi_{x^{\ast}\rightarrow  x_{k}} \sbra{\Pi_{x_{k}\rightarrow x^{\ast}}\sbra{\xi}} }_{x^{\ast}}}{\norm{\Pi_{x_{k}\rightarrow x^{\ast}}\sbra{\xi}}_{x^{\ast}}}\\
    &= \max_{\overline{\xi}\in\tanspc[x^{\ast}]\mani} \frac{\norm{\overline{B_{k}}\sbra{\overline{\xi}}}_{x^{\ast}}}{\norm{\overline{\xi}}_{x^{\ast}}}\\ 
    &= \norm{\overline{B_{k}}}_{\rm op},
    \end{split}
    \end{align*}
    where the second equality follows from the isometry of $\Pi_{x_{k}\rightarrow x^{\ast}}$ and $\Pi_{x^{\ast}\rightarrow x_{k}}^{-1} = \Pi_{x_{k}\rightarrow x^{\ast}}$, and the third one follows from the fact that $\Pi_{x_{k}\rightarrow x^{\ast}}$ is bijective. 
    Thus, by using Assumption~A\ref{Bkbounded} and Lemma~\ref{AkOperNormBounded} with $\mathcal{A}_x=B_k$, $\brc{\overline{B_{k}}}$ is bounded. Similarly, from the isometry of the parallel transport $\Pi_{x_{k}\rightarrow x^{\ast}}$, we have $\norm{\overline{\Delta x^{\ast}_{k}}}_{x^{\ast}}=\norm{\Delta x^{\ast}_{k}}_{x_k}$, and hence, by Proposition~\ref{DeltaxkBounded}, $\brc{\overline{\Delta x^{\ast}_{k}}}$ is bounded.
    
    Note that since $\brc{\overline{B_{k}}}$ and $\brc{\overline{\Delta x^{\ast}_{k}}}$ are bounded sequences contained in {\it fixed} finite-dimensional normed vector spaces, there exist convergent subsequences of $\brc{\overline{B_{k}}}$ and $\brc{\overline{\Delta x^{\ast}_{k}}}$. Let $\overline{B^{\ast}}$ and $\overline{\Delta x^{\ast}}$ be accumulation points of $\brc{\overline{B_{k}}}$ and $\brc{\overline{\Delta x^{\ast}_{k}}}$, respectively.
    
    As for statements \eqref{statement:BkProdEqBkLine} and \eqref{LineDastBastKKTeq}, for any $k$ and $\xi, \zeta\in\tanspc[x_{k}]\mani$, we have 
    \begin{align*}
        \metr[x_{k}]{B_{k}\sbra{\xi}}{\zeta} = \metr[x_{k}]{\Pi_{x_{k}\rightarrow x^{\ast}} \circ B_{k} \circ \Pi_{x^{\ast}\rightarrow  x_{k}} \sbra{\Pi_{x_{k}\rightarrow x^{\ast}}\xi} }{\Pi_{x_{k}\rightarrow x^{\ast}}\zeta} = \metr[x^{\ast}]{\overline{B_{k}}\sbra{\overline{\xi}}}{\overline{\zeta}}, 
    \end{align*}
    where $\overline{\xi} = \Pi_{x_{k}\rightarrow x^{\ast}}\sbra{\xi}\in\tanspc[x^{\ast}]\mani$ and $\overline{\zeta} = \Pi_{x_{k}\rightarrow x^{\ast}}\sbra{\zeta}\in\tanspc[x^{\ast}]\mani$, which ensures that statement \eqref{statement:BkProdEqBkLine} is true. Since $\xi\in\tanspc[x_{k}]\mani$ has been chosen arbitrarily and $\Pi_{x_{k}\rightarrow x^{\ast}}$ is bijective, the symmetry of $B_{k}$ induces that of $\overline{B_{k}}$ by using statement \eqref{statement:BkProdEqBkLine}, and moreover, the uniform positive-definiteness in Assumption~A\ref{Bkbounded} with $B_{k}$ replaced by $\overline{B_{k}}$ is valid. Thus, symmetry and positive-definiteness are kept at the accumulation point $\overline{B^{\ast}}$ of $\left\{\overline{B_k}\right\}$. In addition, for each $k$, the optimal solution $\Delta x^{\ast}_{k}$ of subproblem \eqref{QP} satisfies the KKT conditions \eqref{subprobKKT}, which can be represented as
    \begin{align*}
        &\overline{B_{k}}\sbra{\overline{\Delta x^{\ast}_{k}}} + \Pi_{x_{k}\rightarrow x^{\ast}}\sbra{\text{\upshape grad}\,f\paren{x_{k}}}\\
        &\qquad + \sum_{i\in\mathcal{I}} \mu_{ki}^{\ast} \Pi_{x_{k}\rightarrow x^{\ast}}\sbra{\text{\upshape grad}\,g_{i}\paren{x_{k}}} + \sum_{j\in\mathcal{E}}\lambda_{kj}^{\ast}\Pi_{x_{k}\rightarrow x^{\ast}} \sbra{\text{\upshape grad}\,h_{j}\paren{x_{k}}} = 0,\\
        &\mu^{\ast}_{ki}\geq 0, \, g_{i}\paren{x_{k}} + \metr[]{\Pi_{x_{k}\rightarrow x^{\ast}}\sbra{\text{grad}\,g_{i}\paren{x_{k}}}}{\overline{\Delta x^{\ast}_{k}}}\leq 0,\text{ and}\\
        &\mu^{\ast}_{ki}\paren{g_{i}\paren{x_{k}} + \metr[]{\Pi_{x_{k}\rightarrow x^{\ast}}\sbra{\text{grad}\,g_{i}\paren{x_{k}}}}{\overline{\Delta x^{\ast}_{k}}}} = 0, \text{ for all } i\in\mathcal{I},\\
        & h_{j}\paren{x_{k}} + \metr[]{\Pi_{x_{k}\rightarrow x^{\ast}}\sbra{\text{grad}\,h_{j}\paren{x_{k}}}}{\overline{\Delta x^{\ast}_{k}}} = 0, \text{ for all } j\in\mathcal{E}.
    \end{align*}
    By letting $k$ go to infinity in the above and recalling that $\brc{\paren{\overline{B_k},\overline{\Delta x_{k}^{\ast}}}}$ accumulates at $\brc{\paren{\overline{B^{\ast}}, \overline{\Delta x^{\ast}}}}$, it follows from Lemma~\ref{lem:limParallelGrad} that
    \begin{align*}
        &\overline{B^{\ast}}\sbra{\overline{\Delta x^{\ast}}} + \text{\upshape grad}\,f\paren{x^{\ast}} + \sum_{i\in\mathcal{I}} \mu_{i}^{\ast} \text{\upshape grad}\,g_{i}\paren{x^{\ast}} + \sum_{j\in\mathcal{E}}\lambda_{j}^{\ast} \text{\upshape grad}\,h_{j}\paren{x^{\ast}} = 0,\\
        &\mu^{\ast}_{i}\geq 0, \, g_{i}\paren{x^{\ast}} + \metr[]{\text{grad}\,g_{i}\paren{x^{\ast}}}{\overline{\Delta x^{\ast}}} \leq 0,\text{ and}\\
        &\mu^{\ast}_{i}\paren{g_{i}\paren{x^{\ast}} + \metr[]{\text{grad}\,g_{i}\paren{x^{\ast}}}{\overline{\Delta x^{\ast}}}} = 0, \text{ for all } i\in\mathcal{I},\\
        & h_{j}\paren{x^{\ast}} + \metr[]{\text{grad}\,h_{j}\paren{x^{\ast}}}{\overline{\Delta x^{\ast}}} = 0, \text{ for all } j\in\mathcal{E},
    \end{align*}
    which ensures that statement \eqref{LineDastBastKKTeq} is true.
\end{proof}

Relevant to the penalty function $P_{\bar{\rho}}$, we define the following functions
\begin{align*}
    C\paren{\xi\oplus\zeta} \coloneqq R_{\text{Exp}_{x^{\ast}}\paren{\xi}}\paren{\Pi_{x^{\ast}\rightarrow\text{Exp}_{x^{\ast}}\paren{\xi}} \sbra{\zeta}} \text{ and } F\paren{\xi\oplus\zeta} \coloneqq P_{\bar{\rho}}\circ C\paren{\xi \oplus \zeta}
\end{align*}
for $\xi\in \mathbb{B}_{x^{\ast}} \subseteq \tanspc[x^{\ast}]\mani$ and $\zeta \in\tanspc[x^{\ast}]\mani$. Recall $\injopenball = \brc{ \xi\in\tanspc[x^{\ast}]\mani \relmiddle{|} \norm{\xi} < \text{Inj}\paren{x^{\ast}}}
$.
As shown in the next lemma, the function $F$ is actually Clarke regular. This property will play a key role for proving inequality\,\eqref{eq:key2}.
\begin{lemma}\label{lem:clark}
    For all $\xi\in \mathbb{B}_{x^{\ast}}$ and $\zeta \in\tanspc[x^{\ast}]\mani$, the function $F$ is Clarke regular at $\xi\oplus\zeta$. 
\end{lemma}
\begin{proof}
    Since, by Proposition~\ref{regularityproperties}\eqref{lincombregularity}, a finite linear combination by nonnegative scalars of Clark-regular functions is Lipschitz continuous and Clarke regular, and $F$ is of the form 
    \begin{align*}
        F(\cdot)=f\circ C(\cdot) + \bar{\rho}\left(\sum_{i\in\mathcal{I}} \max\paren{0,g_{i}\circ C(\cdot)}+\sum_{j\in\mathcal{E}}\abs{h_{j}\circ C(\cdot)}\right), 
    \end{align*}
    it suffices to show that each term in $F$ is Clark regular at $\xi\oplus\zeta$ for any $\xi\in \mathbb{B}_{x^{\ast}}$ and $\zeta \in\tanspc[x^{\ast}]\mani$. 
    To this end, we first prove that $C$ is smooth at $\xi\oplus\zeta$ by showing that it is actually a composite function of smooth ones: since the smoothness of the mapping $\paren{x,\zeta} \mapsto \Pi_{x^{\ast} \rightarrow x}\sbra{\zeta}$ follows from the proof of \cite[Lemma A.1]{LiuBoumal19Simple}\footnote{Though in the statement of \cite[Lemma A.1]{LiuBoumal19Simple} the smoothness of the parallel transport at $x$ sufficiently near $x^{\ast}$ is claimed only with respect to $x$, it is in fact proved with respect to both $x$ and $\zeta$ in the proof there.} and both the retraction and the exponential mapping are smooth by definition, we see that $C$ is smooth at $\xi\oplus\zeta$ in view of the definition of $C$. Thus, the functions $f\circ C, \brc{g_{i}\circ C}_{i\in\mathcal{I}}$, and $\brc{h_{j}\circ C}_{j\in\mathcal{E}}$, which are composite functions of continuously differentiable ones, are all continuously differentiable at $\xi\oplus\zeta$, and hence Clarke regular at $\xi\oplus\zeta$. Next, since the functions $\max\paren{0,\cdot}$ and $\abs{\cdot}$ are convex, Proposition~\ref{regularityproperties}\eqref{convexregularity} and \eqref{chainruleregularity} ensure that $\brc{\max\paren{0,g_{i}\circ C}}_{i\in\mathcal{I}}$ and $\brc{\abs{h_{j}\circ C}}_{j\in\mathcal{E}}$ are also Lipschitz continuous near $\xi\oplus\zeta$ and Clarke regular at $\xi\oplus\zeta$. This shows the Clarke regularity of $F$ at $\xi\oplus\zeta$.
\end{proof}

Now we are ready to prove Proposition~\ref{LimBIneqLimP}.
\begin{proof}[Proof of Proposition~\ref{LimBIneqLimP}]
    To begin with, extract a subsequence $\overline{\mathcal{K}}$ from $\mathcal{K}$ such that
    \begin{align}\label{al:okuno3}
        \lim_{k\in\overline{\mathcal{K}}, k\rightarrow \infty} \metr[x^{\ast}]{\overline{B_{k}}\sbra{\overline{\Delta x_{k}^{\ast}}}}{\overline{\Delta x_{k}^{\ast}}} = \limsup_{k\in\mathcal{K}, k\rightarrow \infty} \metr[x^{\ast}]{\overline{B_{k}}\sbra{\overline{\Delta x_{k}^{\ast}}}}{\overline{\Delta x_{k}^{\ast}}}.
    \end{align}
    Letting $v_{k}\coloneqq \text{Exp}^{-1}_{x^{\ast}}\paren{x_{k}}\in\tanspc[x^{\ast}]\mani$ for each $k$, we have $v^{\ast}\coloneqq \lim_{k\in \overline{\mathcal{K}},k\to \infty}v_{k}=\text{Exp}_{x^{\ast}}^{-1}\paren{x^{\ast}} = 0_{x^{\ast}}$ by the smoothness of $\text{Exp}_{x^{\ast}}^{-1}$. 
    Moreover, it follows from Lemma~\ref{existDxBastwithproperties}\eqref{subseqconvLineDxkBk} and $\overline{\mathcal{K}}\subseteq \mathcal{K}$ that the subsequences $\left\{\overline{B_k}\right\}_{k\in\overline{\mathcal{K}}}$ and $\left\{\overline{\Delta x_k}\right\}_{k\in\overline{\mathcal{K}}}$ are bounded.
    Thus, without loss of generality, we can assume that they converge to $\overline{B^{\ast}}$ and $\overline{\Delta x^{\ast}}$, respectively. 
    
    In fact, the desired assertion can be verified by putting together the following facts:
    \begin{enumerate}[(\fact 1)]
        \item ${P_{\bar{\rho}}\circ R_{x_{k}}\paren{\frac{\alpha_{k}}{\beta} \Delta x^{\ast}_{k}} - P_{\bar{\rho}}\paren{x_{k}}}={F\paren{ v_{k}\oplus \frac{\alpha_{k}}{\beta}\overline{\Delta x^{\ast}_{k}}} - F\paren{ v_{k}\oplus 0_{x^{\ast}}}}$,\label{F2:PrhoTransF}
        \item $\displaystyle{\liminf_{k\in\overline{\mathcal{K}}, k\rightarrow \infty}}\frac{\beta}{\alpha_{k}}\paren{ F\paren{v_{k}\oplus\frac{\alpha_{k}}{\beta} \overline{\Delta x^{\ast}_{k}}} - F\paren{v_{k}\oplus 0_{x^{\ast}}}}\le \paren{P_{\bar{\rho}}\circ R_{x^{\ast}}}^{\prime}\paren{0_{x^{\ast}};\overline{\Delta x^{\ast}}}$,\label{F3:liminfFIneqDirPrho}
        \item $\metr[]{\overline{B^{\ast}}\sbra{\overline{\Delta x^{\ast}}}}{\overline{\Delta x^{\ast}}} = \displaystyle{\limsup_{k\in\mathcal{K}, k\rightarrow \infty}}\metr[x_{k}]{B_{k}\sbra{\Delta x_{k}^{\ast}}}{\Delta x_{k}^{\ast}}$. \label{F1:BDeltaxAstEqLimSupBDeltaxk}
    \end{enumerate}
    Indeed, the assertion follows from
    \begin{align*}
       &\limsup_{k\in\mathcal{K}, k\to\infty}\frac{\beta}{\alpha_{k}} \paren{P_{\bar{\rho}}\circ R_{x_{k}}\paren{0_{x_k}}- P_{\bar{\rho}}\circ R_{x_{k}}\paren{\frac{\alpha_{k}}{\beta} \Delta x^{\ast}_{k}}}\\
       &\geq\displaystyle{\limsup_{k\in\overline{\mathcal{K}}, k\rightarrow \infty}}\frac{\beta}{\alpha_{k}}\paren{F\paren{v_{k}\oplus 0_{x^{\ast}}}-F\paren{v_{k}\oplus\frac{\alpha_{k}}{\beta} \overline{\Delta x^{\ast}_{k}}}}\\
       &\geq -\paren{P_{\bar{\rho}}\circ R_{x^{\ast}}}^{\prime}\paren{0_{x^{\ast}};\overline{\Delta x^{\ast}}} \\
       &\geq \metr[]{\overline{B^{\ast}}\sbra{\overline{\Delta x^{\ast}}}}{\overline{\Delta x^{\ast}}}\\ 
       &=\limsup_{k\in\mathcal{K}, k\rightarrow\infty}\metr[x_{k}]{B_{k}\sbra{\Delta x^{\ast}_{k}}}{\Delta x^{\ast}_{k}},
    \end{align*}
    where the first inequality comes from (\fact\ref{F2:PrhoTransF}) and $\mathcal{\overline{K}}\subseteq\mathcal{K}$, the second one from (\fact\ref{F3:liminfFIneqDirPrho}), 
    the third one from  Lemma~\ref{existDxBastwithproperties}\eqref{LineDastBastKKTeq} and Proposition~\ref{prop:meritineq},
    and the equality from (\fact\ref{F1:BDeltaxAstEqLimSupBDeltaxk}). In what follows, we will prove (\fact\ref{F2:PrhoTransF}), (\fact\ref{F3:liminfFIneqDirPrho}), and (\fact\ref{F1:BDeltaxAstEqLimSupBDeltaxk}).

    \textit{Proof of {\rm (\fact\ref{F2:PrhoTransF})}.} 
    Recall the definition of $F$ and the linearity of the parallel transport. (\fact\ref{F2:PrhoTransF}) follows from
    \begin{align*}
        &{P_{\bar{\rho}}\circ R_{x_{k}}\paren{\frac{\alpha_{k}}{\beta} \Delta x^{\ast}_{k}} - P_{\bar{\rho}}\paren{x_{k}}}\\
        &= {P_{\bar{\rho}}\circ R_{x_{k}}\paren{\frac{\alpha_{k}}{\beta} \Delta x^{\ast}_{k}} - P_{\bar{\rho}}\circ R_{x_{k}}\paren{0_{x_k}}}\\
        &=P_{\bar{\rho}}\circ R_{\text{Exp}_{x^{\ast}}\paren{\text{Exp}_{x^{\ast}}^{-1}\paren{x_{k}}}} \paren{\Pi_{x^{\ast}\rightarrow x_{k}} \circ \Pi_{x_{k}\rightarrow x^{\ast}} \sbra{\frac{\alpha_{k}}{\beta}\Delta x^{\ast}_{k}}}\\
        &\qquad \qquad- P_{\bar{\rho}}\circ R_{\text{Exp}_{x^{\ast}}\paren{\text{Exp}_{x^{\ast}}^{-1}\paren{x_{k}}}}\paren{\Pi_{x^{\ast}\rightarrow x_{k}} \circ \Pi_{x_{k}\rightarrow x^{\ast}} \sbra{0_{x_{k}}}} \\
        &= P_{\bar{\rho}}\circ R_{\text{Exp}_{x^{\ast}}\paren{v_{k}}}\paren{\Pi_{x^{\ast}\rightarrow \text{Exp}_{x^{\ast}}\paren{\text{Exp}_{x^{\ast}}^{-1}\paren{x_{k}}}}\sbra{\frac{\alpha_{k}}{\beta}\overline{\Delta x^{\ast}_{k}}}}\\
        &\qquad - P_{\bar{\rho}}\circ R_{\text{Exp}_{x^{\ast}}\paren{v_{k}}}\paren{\Pi_{x^{\ast}\rightarrow \text{Exp}_{x^{\ast}}\paren{\text{Exp}_{x^{\ast}}^{-1}\paren{x_{k}}}}\sbra{0_{x^{\ast}}}}\\
        &= {P_{\bar{\rho}}\circ R_{\text{Exp}_{x^{\ast}}\paren{v_{k}}}\paren{\Pi_{x^{\ast}\rightarrow\text{Exp}_{x^{\ast}}\paren{v_{k}}} \sbra{ \frac{\alpha_{k}}{\beta} \overline{\Delta x^{\ast}_{k}}}} - P_{\bar{\rho}}\circ R_{\text{Exp}_{x^{\ast}}\paren{v_{k}}}\paren{\Pi_{x^{\ast}\rightarrow\text{Exp}_{x^{\ast}}\paren{v_{k}}} \sbra{0_{x^{\ast}}}}}\\
        &={F\paren{ v_{k}\oplus \frac{\alpha_{k}}{\beta}\overline{\Delta x^{\ast}_{k}}} - F\paren{ v_{k}\oplus 0_{x^{\ast}}}}.
    \end{align*}

    \textit{Proof of {\rm (\fact\ref{F3:liminfFIneqDirPrho})}.} 
    Since $F$ is Clarke regular from Lemma \ref{lem:clark}, we can set $\left(l, \tanspc[x]\mani^{2}, \zeta^{\oplus}, \xi^{\oplus}, t\right)$ to $\left(F, \tanspc[x^{\ast}]\mani^2,v_{k}\oplus 0_{x^{\ast}},0_{x^{\ast}}\oplus \overline{\Delta x_{k}^{\ast}},\frac{\alpha_k}{\beta}\right)$ in Theorem~\ref{LebourgMeanValue}, and then have some $s_{k}\in\paren{0,\frac{\alpha_{k}}{\beta}}$ such that
    \begin{align}
        \frac{\beta}{\alpha_{k}}\paren{F\paren{ v_{k}\oplus \frac{\alpha_{k}}{\beta}\overline{\Delta x^{\ast}_{k}}} - F\paren{v_{k}\oplus 0_{x^{\ast}}}} \in \partial F\paren{v_{k}\oplus s_{k}\overline{\Delta x^{\ast}_{k}}}\sbra{0_{x^{\ast}}\oplus \overline{\Delta x_{k}^{\ast}}}.\label{al:okuno2}
    \end{align}
    Note that $\lim_{k\in \overline{\mathcal{K}}, k\to \infty}s_{k}=0$ under the assumption that $\lim_{k\in \mathcal{K},k\to \infty}\alpha_k=0$ and $\overline{\mathcal{K}}\subseteq \mathcal{K}$. It follows from \eqref{al:okuno2} and Proposition~\ref{prop:genegradclosedmapping} that
    \begin{align*}
        \liminf_{k\in\overline{\mathcal{K}}, k\rightarrow \infty} \frac{\beta}{\alpha_{k}}\paren{ F\paren{v_{k}\oplus\frac{\alpha_{k}}{\beta} \overline{\Delta x^{\ast}_{k}}} - F\paren{v_{k}\oplus 0_{x^{\ast}}}} \in \partial F\paren{v^{\ast}\oplus 0_{x^{\ast}}}\sbra{0_{x^{\ast}}\oplus\overline{\Delta x^{\ast}}}.
    \end{align*} Hence, \allowbreak from the definition of the generalized gradient $F^{\circ}$ and Clarke regularity of $F$,
    \begin{align}
        \begin{split}\label{LimsupFIneqDirDeiv}
            &\liminf_{k\in\overline{\mathcal{K}}, k\rightarrow \infty} \frac{\beta}{\alpha_{k}}\paren{ F\paren{v_{k}\oplus\frac{\alpha_{k}}{\beta} \overline{\Delta x^{\ast}_{k}}} - F\paren{v_{k}\oplus 0_{x^{\ast}}}}\\
            &\leq F^{\circ}\paren{v^{\ast}\oplus 0_{x^{\ast}} ; 0_{x^{\ast}}\oplus\overline{\Delta x^{\ast}}}\\
            &= F^{\prime}\paren{v^{\ast}\oplus 0_{x^{\ast}} ; 0_{x^{\ast}}\oplus\overline{\Delta x^{\ast}}}.
        \end{split}
    \end{align}
            Furthermore, by noting $\text{Exp}_{x^{\ast}}\paren{v^{\ast}}=\text{Exp}_{x^{\ast}}\paren{0_{x^{\ast}}}=x^{\ast}$, we have
    \begin{align}
        \begin{split}\label{genederiveqdirderiv}
            &F^{\prime}\paren{v^{\ast}\oplus 0_{x^{\ast}} ;0_{x^{\ast}}\oplus\overline{\Delta x^{\ast}}}\\
            &= \lim_{t\downarrow 0} \frac{P_{\bar{\rho}}\circ R_{\text{Exp}_{x^{\ast}}\paren{v^{\ast}}}\paren{\Pi_{x^{\ast}\rightarrow\text{Exp}_{x^{\ast}}\paren{v^{\ast}}} \sbra{0_{x^{\ast}} + t \overline{\Delta x^{\ast}}}} - P_{\bar{\rho}}\circ R_{\text{Exp}_{x^{\ast}}\paren{v^{\ast}}}\paren{\Pi_{x^{\ast}\rightarrow\text{Exp}_{x^{\ast}}\paren{v^{\ast}}}\sbra{0_{x^{\ast}}}}}{t}\\
            &=\lim_{t\downarrow 0}\frac{ P_{\bar{\rho}}\circ R_{x^{\ast}}\paren{\Pi_{x^{\ast} \rightarrow x^{\ast}} \sbra{t \overline{\Delta x^{\ast}}}} - P_{\bar{\rho}}\circ R_{x^{\ast}}\paren{\Pi_{x^{\ast} \rightarrow x^{\ast}}\sbra{0_{x^{\ast}}}}}{t}\\
            &=\lim_{t\downarrow 0}\frac{ P_{\bar{\rho}}\circ R_{x^{\ast}}\paren{t \overline{\Delta x^{\ast}}} - P_{\bar{\rho}}\circ R_{x^{\ast}}\paren{0_{x^{\ast}}}}{t}\\
            &=\paren{P_{\bar{\rho}}\circ R_{x^{\ast}}}^{\prime}\paren{0_{x^{\ast}};\overline{\Delta x^{\ast}}}.
        \end{split}
    \end{align}
    Finally, (\fact\ref{F3:liminfFIneqDirPrho}) is obtained by combining \eqref{LimsupFIneqDirDeiv} and \eqref{genederiveqdirderiv}. 

    \textit{Proof of {\rm (\fact\ref{F1:BDeltaxAstEqLimSupBDeltaxk})}.}
    Since $\metr[x^{\ast}]{\overline{B_{k}}\sbra{\overline{\Delta x_{k}^{\ast}}}}{\overline{\Delta x_{k}^{\ast}}} =\metr[x_{k}]{B_{k}\sbra{\Delta x_{k}^{\ast}}}{\Delta x_{k}^{\ast}}$ holds by Lemma~\ref{existDxBastwithproperties}\eqref{statement:BkProdEqBkLine} with $\xi=\zeta=\Delta x_{k}^{\ast}$, equation~\eqref{al:okuno3} yields 
    \begin{align*}
        \metr[]{\overline{B^{\ast}}\sbra{\overline{\Delta x^{\ast}}}}{\overline{\Delta x^{\ast}}} = \limsup_{k\in\mathcal{K}, k\rightarrow \infty} \metr[]{B_{k}\sbra{\Delta x_{k}^{\ast}}}{\Delta x_{k}^{\ast}},
    \end{align*}
    which is nothing but the equation in (\fact\ref{F1:BDeltaxAstEqLimSupBDeltaxk}). The whole proof is now complete.
\end{proof}

\section{Riemannian Newton method}\label{SubsecAppen:RiemNewton}
\begin{algorithm}[t]
\caption{Riemannian Newton method for real-valued functions}
\label{RiemNewton}
\begin{algorithmic}      
\REQUIRE Riemannian manifold $\mani$, Riemannian metric$\metr[]{\cdot}{\cdot}$, three times continuously differentiable functions $\theta:\mathcal{M}\rightarrow\eucli$, retraction $R: \tanspc[]\mani\rightarrow\mani$.
\ENSURE Initial iterate $x_{0} \in \mani$. \quad \textbf{Output:} $x^{\ast}\in\mani$ such that $\text{grad}\,\theta\paren{x^{\ast}}=0$.
\FOR  {$k=0,1,\ldots$}
\STATE Solve the Newton equation 
    \begin{align}\label{RiemNeweq}
        \text{Hess}\theta\paren{x_{k}}\sbra{\zeta_{k}} = - \text{grad}\theta\paren{x_{k}}
    \end{align}
for the unknown $\zeta_{k}\in\tanspc[x_{k}]\mani$, where $\text{Hess}\theta\paren{x_{k}}\sbra{\zeta_{k}}=\nabla_{\zeta_{x_{k}}}\text{grad}\theta$;
\STATE Set $x_{k+1} = R_{x_{k}}\paren{\zeta_{k}}$;
\ENDFOR
\end{algorithmic}
\end{algorithm}
We briefly review the Riemannian Newton method from Absil et al.~\cite[Chapter 6]{Absiletal08OptBook}. This is an algorithm for finding a critical point of a three times continuously differentiable function $\theta:\mani\rightarrow\eucli$, i.e., $x\in\mani$ such that grad$\theta\paren{x}=0$. The search direction $\zeta_{k}\in\tanspc[x_{k}]\mani$ is obtained by solving the Newton equation \eqref{RiemNeweq}
and the next iterate is determined by means of a retraction along $\zeta_k$. Here, the step length is fixed to $1$. We formalize this method as Algorithm~\ref{RiemNewton}.
The following theorem holds for Algorithm~\ref{RiemNewton}. Note that Theorem~\ref{RiemNewQuadConv} was originally established for the geometric Newton method, which includes the Riemannian Newton method as an instance~\cite{Absiletal08OptBook}.
\begin{theorem} {\upshape(\cite[Theorem 6.3.2]{Absiletal08OptBook})}\label{RiemNewQuadConv}
Under the requirements and notation of Algorithm~\ref{RiemNewton}, assume that there exists $x^{\ast}\in\mani$ such that $\text{{\upshape grad}}\theta\paren{x^{\ast}} = 0$ and $\text{{\upshape Hess}}\theta\paren{x^{\ast}}^{-1}$ exists. Then, there exists a neighborhood $\mathcal{U}$ of $x^{\ast}$ in $\mani$ such that, for all $x_{0}\in\mathcal{U}$, Algorithm~\ref{RiemNewton} generates an infinite sequence $\brc{x_{k}}_{k=0,1\ldots}$ converging quadratically to $x^{\ast}$.
\end{theorem}

\section{Further numerical experiments}\label{appx:suppexperiment}
Here, we report the further results of the additional numerical experiments. We consider the additional experiments on the nonnegative matrix low-rank completion problems introduced in Section~\ref{experiment}.
Then, we newly consider another problem, that is, a minimum balanced cut problem.

\subsection{Additional experiments on nonegative low-rank matrix completion}\label{sec:suppnonnegcompl}
In addition to the experiment on $\paren{q,s} = \paren{5,10}$ in the first experiment of Section~\ref{subsub:resultnonneglowrank}, we conduct these on $\paren{q,s} = \paren{4,8},\paren{6,12},$ and $\paren{7,14}$. 

Tables~\ref{RCNNResTab} and \ref{RCNNResTabPart2} show the spent time and the number of iterations until the residual of each algorithm reached $10^{-i}$ for $i=-1,0,1,\ldots$.
For example, RSQO spent 1.332 seconds until it obtained a solution with residual$=10^{-4}$ in $\paren{q,s} = \paren{4,8}$. 
Note that the case of $\paren{q,s} = \paren{5,10}$ in Table~\ref{RCNNResTab} is another expression of Figure~\ref{Fig:NonnegResTime} in Section~\ref{subsub:resultnonneglowrank}.
From the tables, we can see that RSQO tends to compute the solution more accurately than the others. Indeed, RSQO successfully solved the problems for $\paren{q,s}=\paren{4,8}, \paren{5,10}$ and $\paren{7,14}$, while the other Riemannian methods failed to find a solution with the same accuracy. 
As we see in Figure~\ref{Fig:NonnegResTime}, RSQO tends to steadily decrease the residual at first and then much more time is necessary to get more accurate solutions in other cases of RSQO of Tables~\ref{RCNNResTab} and \ref{RCNNResTabPart2}.

\begin{table}[t]
\caption{Residual vs. CPU time (sec.) for nonnegative low-rank matrix completion (1/2)}
\small
\centering
\begin{tabular}{S[table-format=2.3e-1]S[table-format=2.3e-1]S[table-format=2.3e-1]S[table-format=2.3e-1]S[table-format=2.3e-1]}
\multicolumn{5}{c}{$\paren{q,s} =  \paren{4,8}$}                                                                            \\ \hline
\multicolumn{1}{c|}{Residual}   & {RSQO}               & {RALM}               & {REPM (LQH)}         & {REPM (LSE)}         \\ \hline
\multicolumn{1}{c|}{$10$}       & 2.327E-04            & 0.000E+00            & 0.000E+00            & 0.000E+00            \\
\multicolumn{1}{c|}{$1$}        & 3.386E-01            & 1.721E+00            & 0.000E+00            & 1.483E+00            \\
\multicolumn{1}{c|}{$10^{-1}$}  & 5.733E-01            & 1.849E+00            & 1.407E-01            & 1.483E+00            \\
\multicolumn{1}{c|}{$10^{-2}$}  & 8.962E-01            & 1.960E+00            & 1.407E-01            & 3.629E+00            \\
\multicolumn{1}{c|}{$10^{-3}$}  & 1.272E+00            & 1.960E+00            & 3.748E-01            & 4.210E+00            \\
\multicolumn{1}{c|}{$10^{-4}$}  & 1.332E+00            & 1.960E+00            & 3.871E-01            & 4.678E+00            \\
\multicolumn{1}{c|}{$10^{-5}$}  & 1.392E+00            & 2.021E+00            & 4.611E-01            & 4.975E+00            \\
\multicolumn{1}{c|}{$10^{-6}$}  & 1.392E+00            & 2.133E+00            & 6.936E-01            & 5.395E+00            \\
\multicolumn{1}{c|}{$10^{-7}$}  & 1.121E+01            & 2.214E+00            & 7.614E-01            & {-}                  \\
\multicolumn{1}{c|}{$10^{-8}$}  & 1.371E+01            & {-}                  & {-}                  & {-}                  \\
\multicolumn{1}{c|}{$10^{-9}$}  & 3.422E+01            & {-}                  & {-}                  & {-}                  \\
\multicolumn{1}{c|}{$10^{-10}$} & 2.639E+02            & {-}                  & {-}                  & {-}                  \\ \hline
\multicolumn{1}{l}{}            & \multicolumn{1}{l}{} & \multicolumn{1}{l}{} & \multicolumn{1}{l}{} & \multicolumn{1}{l}{} \\
\multicolumn{5}{c}{$\paren{q,s} = \paren{5,10}$}                                                                            \\ \hline
\multicolumn{1}{c|}{Residual}   & {RSQO}               & {RALM}               & {REPM (LQH)}         & {REPM (LSE)}         \\ \hline
\multicolumn{1}{c|}{$10$}       & 1.394E-04            & 0.000E+00            & 0.000E+00            & 0.000E+00            \\
\multicolumn{1}{c|}{$1$}        & 3.975E-01            & 2.245E+00            & 1.207E+00            & 1.819E+00            \\
\multicolumn{1}{c|}{$10^{-1}$}  & 5.139E-01            & 5.463E+00            & 1.207E+00            & 1.819E+00            \\
\multicolumn{1}{c|}{$10^{-2}$}  & 1.057E+00            & 7.826E+00            & 1.207E+00            & 6.255E+00            \\
\multicolumn{1}{c|}{$10^{-3}$}  & 5.517E+00            & 9.969E+00            & 3.471E+00            & 9.275E+00            \\
\multicolumn{1}{c|}{$10^{-4}$}  & 5.725E+00            & 1.635E+01            & 1.218E+01            & 2.246E+01            \\
\multicolumn{1}{c|}{$10^{-5}$}  & 5.827E+00            & 2.789E+01            & 1.739E+01            & {-}                  \\
\multicolumn{1}{c|}{$10^{-6}$}  & 5.930E+00            & {-}                  & {-}                  & {-}                  \\
\multicolumn{1}{c|}{$10^{-7}$}  & 6.035E+00            & {-}                  & {-}                  & {-}                  \\
\multicolumn{1}{c|}{$10^{-8}$}  & 6.241E+00            & {-}                  & {-}                  & {-}                  \\
\multicolumn{1}{c|}{$10^{-9}$}  & 2.226E+01            & {-}                  & {-}                  & {-}                  \\
\multicolumn{1}{c|}{$10^{-10}$} & 1.927E+02            & {-}                  & {-}                  & {-}                  \\ \hline
\end{tabular}\par
    \bigskip
    \begin{center}
        ``-'' means that the algorithm cannot reach the residual.
    \end{center}
    \label{RCNNResTab}
\end{table}

\begin{table}[t]
\caption{Residual vs. CPU time (sec.) for nonnegative low-rank matrix completion (2/2)}
\small
\centering
\begin{tabular}{S[table-format=2.3e-1]S[table-format=2.3e-1]S[table-format=2.3e-1]S[table-format=2.3e-1]S[table-format=2.3e-1]}
\multicolumn{5}{c}{$\paren{q,s} =  \paren{6,12}$}                                                                           \\ \hline
\multicolumn{1}{c|}{Residual}   & {RSQO}               & {RALM}               & {REPM (LQH)}         & {REPM (LSE)}         \\ \hline
\multicolumn{1}{c|}{$10$}       & 1.501E-04            & 0.000E+00            & 0.000E+00            & 0.000E+00            \\
\multicolumn{1}{c|}{$1$}        & 3.746E-01            & 2.873E+00            & 0.000E+00            & 2.099E+00            \\
\multicolumn{1}{c|}{$10^{-1}$}  & 7.890E-01            & 6.185E+00            & 1.483E+00            & 2.099E+00            \\
\multicolumn{1}{c|}{$10^{-2}$}  & 1.587E+00            & 1.384E+01            & 9.138E+00            & {-}                  \\
\multicolumn{1}{c|}{$10^{-3}$}  & 2.196E+00            & 2.444E+01            & 4.196E+01            & {-}                  \\
\multicolumn{1}{c|}{$10^{-4}$}  & 3.138E+00            & 2.444E+01            & 1.792E+02            & {-}                  \\
\multicolumn{1}{c|}{$10^{-5}$}  & 1.429E+01            & 2.600E+01            & {-}                  & {-}                  \\
\multicolumn{1}{c|}{$10^{-6}$}  & 1.489E+02            & 2.626E+01            & {-}                  & {-}                  \\ \hline
\multicolumn{1}{l}{}            & \multicolumn{1}{l}{} & \multicolumn{1}{l}{} & \multicolumn{1}{l}{} & \multicolumn{1}{l}{} \\
\multicolumn{5}{c}{$\paren{q,s} =  \paren{7,14}$}                                                                           \\ \hline
\multicolumn{1}{c|}{Residual}   & {RSQO}               & {RALM}               & {REPM (LQH)}         & {REPM (LSE)}         \\ \hline
\multicolumn{1}{c|}{$10$}       & 1.462E-04            & 0.000E+00            & 0.000E+00            & 0.000E+00            \\
\multicolumn{1}{c|}{$1$}        & 4.558E-01            & 3.797E+00            & 1.172E+00            & 2.614E+00            \\
\multicolumn{1}{c|}{$10^{-1}$}  & 2.018E+00            & 9.133E+00            & 1.172E+00            & 5.070E+00            \\
\multicolumn{1}{c|}{$10^{-2}$}  & 3.031E+01            & 1.966E+01            & 1.172E+00            & 1.184E+01            \\
\multicolumn{1}{c|}{$10^{-3}$}  & 3.118E+01            & {-}                  & 1.583E+00            & 2.699E+01            \\
\multicolumn{1}{c|}{$10^{-4}$}  & 3.146E+01            & {-}                  & 2.758E+00            & 6.508E+01            \\
\multicolumn{1}{c|}{$10^{-5}$}  & 3.176E+01            & {-}                  & 3.298E+00            & {-}                  \\
\multicolumn{1}{c|}{$10^{-6}$}  & 3.233E+01            & {-}                  & 4.282E+00            & {-}                  \\
\multicolumn{1}{c|}{$10^{-7}$}  & 3.260E+01            & {-}                  & 6.315E+00            & {-}                  \\
\multicolumn{1}{c|}{$10^{-8}$}  & 3.315E+01            & {-}                  & {-}                  & {-}                  \\
\multicolumn{1}{c|}{$10^{-9}$}  & 3.372E+01            & {-}                  & {-}                  & {-}                  \\ \hline
\end{tabular}\par
    \bigskip
    \begin{center}
        ``-'' means that the algorithm cannot reach the residual.
    \end{center}
    \label{RCNNResTabPart2}
\end{table}

In addition, the average CPU time per step for RSQO to reach its most accurate solution drastically increased as the problem size did: 
in $\paren{q,s}= \paren{4,8}$, RSQO reached the most accurate solution with residual$=10^{-10}$ by the $4471$-st iteration.
Similarly, the number of the iterations for RSQO to reach the most accurate solutions were $1795, 766$, and $110$ in $\paren{q,s} = \paren{5,10}, \paren{6,12}$, and $\paren{7,14}$, respectively. Thus, the average time was $2.639\times 10^{2} \slash 4471 = 5.903 \times 10^{-4}$ seconds in $\paren{q,s} = \paren{4,8}$ and similarly $1.074 \times 10^{-1}, 1.944 \times 10^{-1}$ and $3.034 \times 10^{-1}$ seconds in $\paren{q,s} = \paren{5,10}, \paren{6,12}$ and $\paren{7,14}$, respectively. 
Here, the most expensive procedure of RSQO was to call \texttt{hessianmatrix} as well as that in Section~\ref{subsub:resultnonneglowrank}; 
for all cases, RSQO terminated its computations due to the excess of the maximal time and it took $4.639\times 10^{2}, 4.958 \times 10^{2}, 5.062\times 10^{2}$ and $5.137\times 10^{2}$ out of $600$ seconds to call \texttt{hessianmatrix} in $\paren{q,s}=\paren{4,8},\paren{5,10},\paren{6,12}$ and $\paren{7,14}$, respectively. Hence, using the Hessian of the Lagrangian seems to be particularly expensive as the problem size becomes large.

\subsection{Minimum balanced cut for graph bisection via relaxation}\label{sec:minbalcut}
\subsubsection{Problem setting}\label{subsec:probsetminbalcut}
The following problem setting was introduced by Liu and Boumal~\cite{LiuBoumal19Simple}. Let $L$ be a $q\times q$ matrix and $e$ be a $q$-dimensional vector whose entries are all ones. Define an oblique manifold by $\text{Oblique}\paren{q,s} \coloneqq \brc{X\in\eucli[q\times s] \relmiddle{|} \text{diag}\paren{XX^{\top}} = e}$, where $\text{diag}\paren{\cdot}$ returns a vector consisting of the diagonal elements of the argument matrix. Then, the problem can be represented as
\begin{align}
    \begin{split}\label{BC}
    \min_{X\in\text{Oblique}\paren{q,s}} \quad -\frac{1}{4} \text{tr}\paren{X^{\top}LX} \quad
    \text{s.t.} \quad   X^{\top}e = 0.
    \end{split}
\end{align}

\textbf{Input.} We set $q=50$ and generate $L$ by following \cite{Liu19RALMcode} with a hyperparameter $density = 0.01$. We also set $s = 2$; that is, the number of variables equals $100$.

\subsubsection{Experimental environment}\label{subsec:expenvminbalcut}

In addition to RSQO, RALM, and REPMs, 
we solve the minimum balanced cut problem with \texttt{fmincon}, a solver for constrained nonlinear optimization in a Euclidean space. Since \eqref{BC} can be formulated as a Euclidean optimization problem having only equality constraints in $\eucli[q\times s]$, we opt to solve \eqref{BC} not with an interior-point method but rather with SQO in \texttt{fmincon}.

To measure the deviation of an iterate from the set of KKT points, we use residuals based on the KKT conditions \eqref{KKT} and the manifold constraints of the problems: in the minimum balanced cut problem, the residual for Riemannian methods is defined by
    \begin{align*}
        \sqrt{\norm{\text{grad}\,\mathcal{L}_{\lambda}\paren{X}}^{2} + \sum_{j\in\mathcal{E}} \abs{h_{j}\paren{X}}^{2} + \text{Manvio}\paren{X}^{2}},
    \end{align*}
where the first two terms in the square root originate from the KKT conditions \eqref{KKTLag} and \eqref{KKTeq}, and the last one means violation of the manifold constraints, defined by $\text{Manvio}\paren{X} \coloneqq \norm{\text{diag}\paren{XX^{\top}} - e}$. For \textbf{fmincon SQO}, we also define the residual based on the KKT conditions for the Euclidean form of \eqref{BC}, namely, $\min_{X\in\eucli[q\times s]} -\frac{1}{4} \text{tr}\paren{X^{\top}LX}$ s.t. $\text{diag}\paren{XX^{\top}} = e, X^{\top}e = 0$.

\subsubsection{Numerical result}\label{subsec:minbalcutresult}
We applied the algorithms to a randomly generated instance of the minimum balanced cut problem under the same settings as in the first experiment of the nonegative low-rank matrix completion problem except that the initial point was generated uniformly at random. 

Figure~\ref{Fig:BCResIter} shows the residual of the algorithms for the first 7 seconds. RSQO successfully solved the instance with the highest accuracy of the residual around $10^{-13}$, while the accuracies of other solutions were at most $10^{-6}$ achieved by RALM. The residual of RSQO dramatically decreased around 1.5 seconds. This might be due to the fact that RSQO possesses a quadratic convergence property as a result of using the Hessian of the Lagrangian. In spite of that \textbf{fmincon SQO} and RSQO share the same SQO framework, \textbf{fmincon SQO} did not work at all, even when computing a solution with residual$=1$, as we can see in Figure~\ref{Fig:BCResIter}. There was almost no improvement in the residual after it reached a feasible solution of the problem. This fact may underscore an advantage of the Riemannian manifold approach. RSQO terminated its computation because of the excess of the maximal time, 600 seconds. In the computation of RSQO, it occupied more than $85\%$ of the whole running time to call the function \texttt{hessianmatrix} as well as in Sections~\ref{subsub:resultnonneglowrank} and Appendix~\ref{sec:suppnonnegcompl}. 
This means that constructing the Hessian of the Lagrangian is the most expensive.

\begin{figure}[ht]
    \begin{center}
        \includegraphics[width=0.7\linewidth]{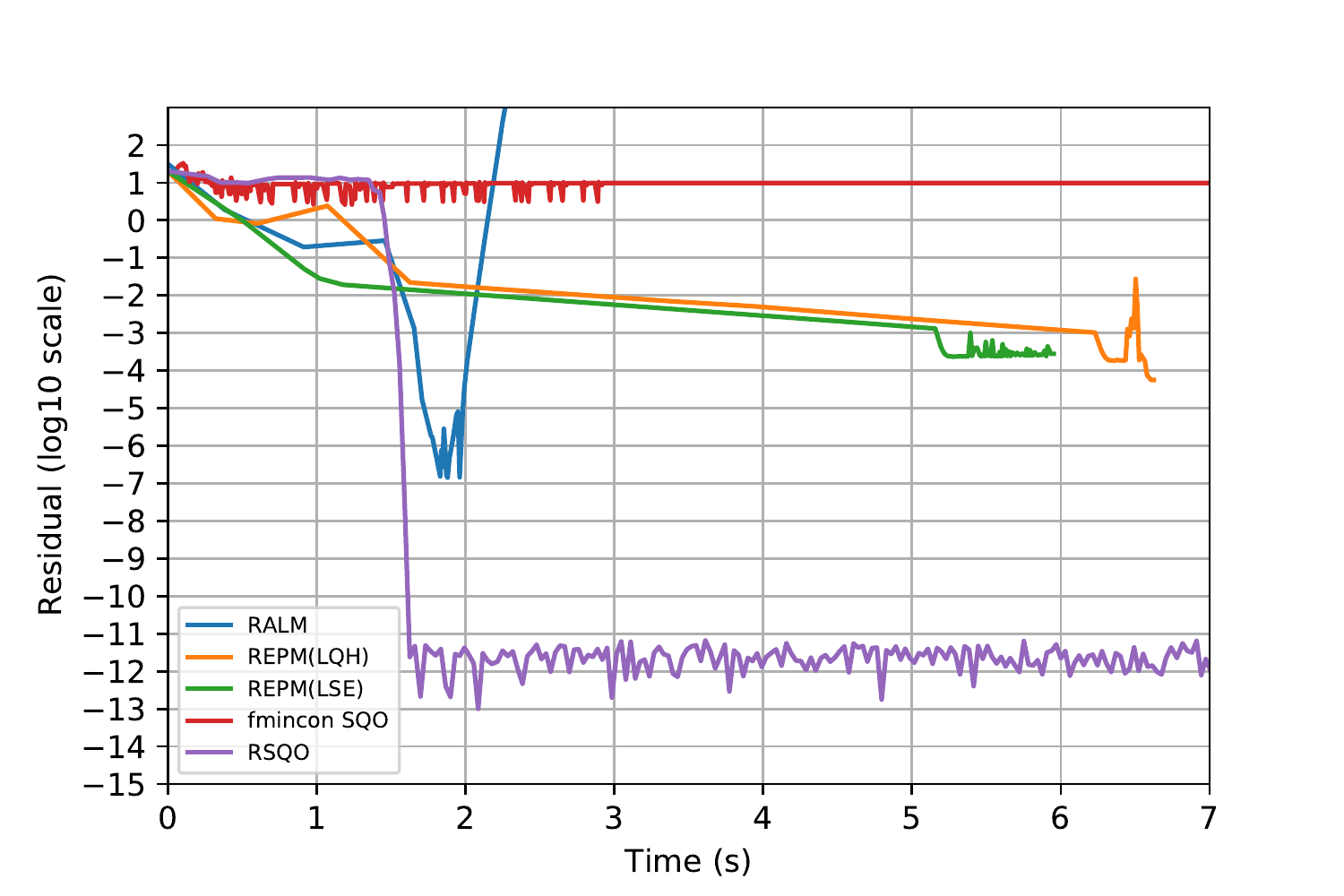}
    \end{center}
    \caption{Residual for the first 7 seconds for a minimum balanced cut problem. REPMs terminate by the 7 second because they do not update any parameter. RALM keeps increasing its residual after it excesses the residual$=10^{3}$. RSQO and \textbf{fmincon SQO} keep oscillating around the residuals$=10^{-12}$ and $10$ after 7 seconds, respectively.}
	\label{Fig:BCResIter}
\end{figure}

\newpage

\bibliographystyle{siamplain}
\bibliography{references.bib}
\end{document}